\documentclass[a4paper, 11pt, american]{amsart}

\usepackage[colorlinks,pdfpagelabels,pdfstartview = FitH,bookmarksopen
= true,bookmarksnumbered = true,linkcolor = blue,plainpages =
false,hypertexnames = false,citecolor = red,pagebackref=false]{hyperref}
\usepackage{times,latexsym,amssymb}
\usepackage{amsmath,amsthm,bm}
\usepackage{graphics, epsfig}
\usepackage{color}
\usepackage{appendix}
\usepackage[margin=1.4in]{geometry}

\usepackage{amsbsy}
\usepackage{amstext}
\usepackage{amssymb}
\usepackage{esint}
\usepackage{stmaryrd}
\usepackage{mathrsfs}
\renewcommand{\d}{\mathrm{d}}
\newcommand{\dx}{\mathrm{d}x}

\newcommand{\dt}{\mathrm{d}t}
\renewcommand{\rho}{\varrho}

\let\TeXchi\chi
\newbox\chibox
\setbox0 \hbox{\mathsurround0pt $\TeXchi$}
\setbox\chibox \hbox{\raise\dp0 \box 0 }
\def\chi{\copy\chibox}

\def\Xint#1{\mathchoice
    {\XXint\displaystyle\textstyle{#1}}%
    {\XXint\textstyle\scriptstyle{#1}}%
    {\XXint\scriptstyle\scriptscriptstyle{#1}}%
    {\XXint\scriptscriptstyle\scriptscriptstyle{#1}}%
    \!\int}
\def\XXint#1#2#3{\setbox0=\hbox{$#1{#2#3}{\int}$}
    \vcenter{\hbox{$#2#3$}}\kern-0.5\wd0}
\def\bint{\Xint-}
\def\dashint{\Xint{\raise4pt\hbox to7pt{\hrulefill}}}

\def\XXiint#1#2#3{\setbox0=\hbox{$#1{#2#3}{\iint}$}
    \vcenter{\hbox{$#2#3$}}\kern-0.5\wd0}

\author[N. Liao]{Naian Liao}
\address{Naian Liao,
Fachbereich Mathematik, Universit\"at Salzburg,
Hellbrunner Str. 34, 5020 Salzburg, Austria}
\email{naian.liao@sbg.ac.at}
\newtheorem{proposition}{Proposition}[section]
\newtheorem{theorem}{Theorem}[section]
\newtheorem{lemma}{Lemma}[section]

\newtheorem{remark}{Remark}[section]
\numberwithin{equation}{section}
\numberwithin{theorem}{section}
\numberwithin{proposition}{section}
\numberwithin{lemma}{section}
\numberwithin{remark}{section}
\newcommand{\noi}{\noindent}

\newcommand{\dsty}{\displaystyle}

\newcommand{\al}{\alpha}

\newcommand{\be}{\beta}

\newcommand{\gm}{\gamma}
\newcommand{\dl}{\delta}
\newcommand{\Dl}{\Delta}
\newcommand{\lm}{\lambda}

\newcommand{\varep}{\varepsilon}

\newcommand{\vp}{\varphi}
\newcommand{\sig}{\sigma}

\newcommand{\om}{\omega}

\newcommand{\z}{\zeta}

\newcommand{\nn}{\mathbb{N}}

\newcommand{\rr}{\mathbb{R}}
\newcommand{\rn}{\rr^N}

\newcommand{\bl}[1]{\mathbf{#1}}
\newcommand{\dvg}{\operatorname{div}}

\newcommand{\essup}{\operatornamewithlimits{ess\,sup}}
\newcommand{\essinf}{\operatornamewithlimits{ess\,inf}}
\newcommand{\essosc}{\operatornamewithlimits{ess\,osc}}
\newcommand{\osc}{\operatornamewithlimits{osc}}

\newcommand{\loc}{\operatorname{loc}}

\newcommand{\dist}{\operatorname{dist}}

\newcommand{\pl}{\partial}

\newcommand{\intl}{\int\limits}

\def\Xint#1{\mathchoice
    {\XXint\displaystyle\textstyle{#1}}%
    {\XXint\textstyle\scriptstyle{#1}}%
    {\XXint\scriptstyle\scriptscriptstyle{#1}}%
    {\XXint\scriptscriptstyle\scriptscriptstyle{#1}}%
    \!\int}
\def\XXint#1#2#3{\setbox0=\hbox{$#1{#2#3}{\int}$}
    \vcenter{\hbox{$#2#3$}}\kern-0.5\wd0}
\def\bint{\Xint-}
\def\dashint{\Xint{\raise4pt\hbox to7pt{\hrulefill}}}
\def\dashiint{\bint\kern-0.15cm\bint}
\newcommand{\ovl}[3]{\int_{#1}^{#2}\kern-#3pt\raise4pt\hbox to7pt{\hrulefill}\ }

\newcommand{\ovll}[3]{\intl_{#1}^{#2}\kern-#3pt\raise4pt\hbox to7pt{\hrulefill}\ }

\newcommand{\tvl}[2]{\iint_{#1}\kern-#2pt\raise4pt\hbox to7pt{\hrulefill}\ }

\newcommand{\bye}{\end{document}}

\begin{document}
%%%%%%%%%%%%%%%%%%%%%%%%%%%%%%%%%%%%%%%%%%%%%%%%%%%
\title[Boundary continuity for the Stefan problem]{On the logarithmic type boundary modulus of continuity for  the Stefan problem}
%%%%%%%%%%%%%%%%%%%%%%%%%%%%%%%%%%%%%%%%%%%%%%%%%%%
%%%%%%%%%%%%%%%%%%%%%%%%%%%%%%%%%%%%%%%%%%%%%%%%
%\author{
%%%%%%%%%%%%%%%%%%%%%%%%%%%%%%%%%%%%%%%%%%%%%%%%%%%
%Naian Liao\\
%%\thanks{Corresponding author; N. L. is partially supported by NSFC 11701054 grant.}\\
%%%%%%%%%%%%%%%%%%%%%%%%%%%%%%%%%%%%%
%Fachbereich Mathematik, Universit\"at Salzburg\\
%Hellbrunner Str. 34, 5020 Salzburg, Austria\\
%email: {\tt naian.liao@sbg.ac.at}
%}
%%%%%%%%%%%%%%%%%%%%%%%%%%%%%%%%%%%%%%%%%%%%%%%%%%%
\date{}
\maketitle
%%%%%%%%%%%%%%%%%%%%%%%%%%%%%%%%%%%%%%%%%%%%%%%%%%%
\begin{abstract}
A logarithmic type modulus of continuity is established for weak solutions to a two-phase Stefan problem,
up to the parabolic boundary of a cylindrical space-time domain. For the Dirichlet problem, we
merely assume
that the spatial domain satisfies a measure density property, 
and the boundary datum has a logarithmic type modulus of continuity. 
For the Neumann problem, we assume that the lateral boundary is smooth, and the boundary datum is bounded.
The proofs are measure theoretical in nature, exploiting De Giorgi's iteration and refining DiBenedetto's approach.
Based on the sharp quantitative estimates, construction of continuous weak (physical) solutions is also indicated.
The logarithmic type  modulus of continuity has been conjectured to be optimal as a structural property
 for weak solutions to such partial differential equations.
%%%%%%%%%%%%%%%%%%%%%%%%%%%%%%%%%%%%%%%%%%%%%%%%%%%
\vskip.2truecm
%%%%%%%%%%%%%%%%%%%%%%%%%%%%%%%%%%%%%%%%%%%%%%%%%%%

\vskip.2truecm
%%%%%%%%%%%%%%%%%%%%%%%%%%%%%%%%%%%%%%%%%%%%%%%%%%%
\end{abstract}
%%%%%%%%%%%%%%%%%%%%%%%%%%%%%%%%%%%%%%%%%%%%%%%%%%%%%%%%%%%%%%%%%%%%%%%
%%%%%%%%%%%%%%%%%%%%%%%%%%%%%%%%%%%%%%%%%%%%%%%%%%%%%%%%%%%%%%%%%%%%%%%
\tableofcontents
%%%%%%%%%%
\section{Introduction}%\label{S:intro}
The classical Stefan problem aims to describe the evolution of the moving boundary between two phases
of a material undergoing a phase change, for instance the melting of ice to water.
The temperature of the material in both phases is governed by the heat equation, subject to the usual
initial and boundary conditions. However, the interface of the two phases is free, and additional conditions need to be imposed
on that part of the boundary (an unknown hyper-surface), which can be viewed as a kind of energy balance law.
As such the classical Stefan problem is genuinely a parabolic free boundary problem, cf.~\cite{Caff-Salsa}.
The model accounts for a large class of physical phenomena, cf.~\cite{Visintin-08}. 

The Stefan problem also admits a variational perspective, cf.~\cite{Friedman-68, Kameno-61, Oleinik} 
and \cite[Chapter~V, Section~9]{LSU}. 
Under this point of view, we are led to formulate
the initial-boundary value problem for the following nonlinear parabolic equation in a fixed space-time domain
$E_T:=E\times(0,T]\subset \rr^{N+1}$, that is,
\begin{equation}\label{Eq:1:0}
\pl_t\be(u)-\Dl u\ni 0\quad\text{ weakly in }E_T.
\end{equation}
Here $\be(\cdot)$ represents the energy (more precisely, the {\it enthalpy}) of the two phases, i.e. $[u>0]$ and $[u<0]$,
and it is a maximal monotone graph that permits a jump when the temperature $u$ is zero.
The equation \eqref{Eq:1:0} is understood in the sense of differential inclusions.
No explicit reference is made to any free boundary in this formulation.
Instead, a {\it mushy region}, i.e. $[u=0]$, which consists of a mixture of the two states of the material,
 is allowed.
The solutions are sought in Sobolev spaces and thus are certain weak ones.
As a result, the question concerning the regularity of the weak solutions is naturally raised.
This is the problem we try to tackle in this note.
See   \cite[\S~1.3]{Visintin-08} for a comparison between the classical and the weak formulations.
See \cite{Kinder} for another formulation in terms of variational inequalities.

Pioneering works of Caffarelli \& Evans \cite{Caff-Evans-83}, DiBenedetto \cite{DB-82, DB-86}, 
Sacks \cite{Sacks-83}, and Ziemer \cite{Ziemer-82}, have shown that while the energy admits a jump, 
the temperature is always continuous, even across the mushy region $[u=0]$.
Among these contributions, we single out \cite{DB-82, DB-86} whose approach is closely followed in this work.
In \cite{DB-82} the interior and up to the boundary continuity of the temperature was obtained,
given general Neumann data or homogeneous Dirichlet  data on the lateral boundary.
However, the case of nonhomogeneous Dirichlet data turned out to be more involved
and was solved later in \cite{DB-86}. An explicit modulus of continuity of the temperature, both in the interior and at the boundary,
can be derived from the method used in \cite{DB-82, DB-86}. Roughly speaking, it is of the following type (cf.~\cite{DB-Friedman-84}):
\begin{equation}\tag{I}\label{Eq:modulus:1}
(0,1)\ni r\mapsto \boldsymbol\om(r)=\big(\ln|\ln (cr)| \big)^{-\sig}\quad\text{ for some }c,\,\sig>0.
\end{equation}
The method employed in \cite{DB-82, DB-86} is  flexible enough to deal with the situation when the Laplacian $\Dl$ in \eqref{Eq:1:0}
is replaced by a general quasilinear diffusion part    %as long as it admits a linear growth in $|Du|$
and when a lower order term appears.

Subsequently, the equation \eqref{Eq:1:0} replacing the Laplacian $\Dl$ by the $p$-Laplacian $\Dl_p$ for $p\ge2$
has been studied in \cite{Urbano-14, Urbano-17, Urbano-97, Urbano-00, Urbano-08} by Urbano  et al. 
%As such the equation presents double nonlinearity; see also \cite{BDL} in this regard.
In particular, we single out the recent work \cite{Urbano-14} where the {\it interior} modulus of continuity
has been improved to be of the following type
\begin{equation}\tag{II}\label{Eq:modulus:2}
(0,1)\ni r\mapsto \boldsymbol\om(r)=|\ln (cr)|^{-\sig}\quad\text{ for some }c,\,\sig>0.
\end{equation}
%where the constants can be quantified by $N$ and $p$.
Discarding a logarithm, this type of modulus of continuity represents an improvement even in the case of Laplacian,
let alone the complication brought by the $p$-Laplacian.
The  modulus \eqref{Eq:modulus:2} has been conjectured to be optimal in \cite{Urbano-14}.
(See also \cite{Caff-Fried} for the one-phase problem.)
%The {\it boundary} continuity is established in \cite{Urbano-17}, given proper Dirichlet data. 
Although a {\it boundary} modulus of continuity is derived in \cite{Urbano-17} for Dirichlet data, it is still of type \eqref{Eq:modulus:1}.

How to achieve a boundary modulus of type \eqref{Eq:modulus:2} with  Dirichlet data,
and how to deal with the same issue for Neumann data, remain elusive.
We will answer these questions in this note.
The significance of continuity for weak solutions (temperatures) stems from their physical bearings;
an explicit and sharp modulus has further mathematical implications, cf.~\cite{Caff-Fried}.
Our quantitative estimates up to the boundary allow us to construct  physical solutions
with the  type \eqref{Eq:modulus:2} modulus of continuity.
%%%%%%%%
\subsection{Statement of the  results}
Our main goal is to establish a modulus of continuity of type \eqref{Eq:modulus:2}
for weak solutions to the Stefan problem, up to the parabolic boundary,
given either Dirichlet or Neumann conditions. Moreover, for the Dirichlet problem, the $p$-Laplacian
with $p\ge2$ is considered. Evidently, our results are new even for the Laplacian.

More precisely, denoting an open set in $\rr^N$ ($N\ge1$) by $E$
 and setting $E_T:=E\times(0,T]$, we are concerned with %the boundary continuity of weak solutions to 
 the following
 nonlinear parabolic equation
\begin{equation}\label{Eq:1:1}
\pl_t\be(u)-\dvg\bl{A}(x,t,u, Du) \ni 0\quad \text{ weakly in }\> E_T.
\end{equation}
Here $\be$ is a maximal monotone graph in $\rr\times\rr$ defined by
\begin{equation*}%\label{Eq:beta}
\be(u)=\left\{
\begin{array}{cl}
u,\quad& u>0,\\[5pt] 
\left[-\nu,0\right],\quad& u=0,\\[5pt]
u-\nu,\quad& u<0,
\end{array}
\right.
\end{equation*}
for some constant $\nu>0$ that represents the exchange of {\it latent heat}.

The function $\bl{A}(x,t,u,\xi)\colon E_T\times\rr^{N+1}\to\rn$ is assumed to be
measurable with respect to $(x, t) \in E_T$ for all $(u,\xi)\in \rr\times\rn$,
and continuous with respect to $(u,\xi)$ for a.e.~$(x,t)\in E_T$.
Moreover, we assume the structure conditions
\begin{equation}  \label{Eq:1:2}
\left\{
\begin{array}{l}
\bl{A}(x,t,u,\xi)\cdot \xi\ge C_o|\xi|^p \\[5pt]
|\bl{A}(x,t,u,\xi)|\le C_1|\xi|^{p-1}%
\end{array}%
\right .\quad \text{ a.e.}\> (x,t)\in E_T,\, \forall\,u\in\rr,\,\forall\xi\in\rn,
\end{equation}
where $C_o$ and $C_1$ are given positive constants, and we take $p\ge2$.
The set of parameters $\{\nu, p,N,C_o,C_1\}$ will be referred to as the (structural) data in the sequel.

Before considering the initial-boundary value problems for \eqref{Eq:1:1},
let us recall the parabolic $p$-Laplace type equation, which is pertinent to \eqref{Eq:1:1}:
\begin{equation}\label{Eq:p-Laplace}
u_t-\dvg\bl{A}(x,t,u, Du) = 0\quad \text{ weakly in }\> E_T,
\end{equation}
where the properties of $\bl{A}(x,t,u, \xi)$ are retained.

The Dirichlet problem  for \eqref{Eq:1:1} is formulated as:
\begin{equation}\label{Dirichlet}
\left\{
\begin{aligned}
&\pl_t\be(u)-\dvg\bl{A}(x,t,u, Du) \ni 0\quad \text{ weakly in }\> E_T\\
&u(\cdot,t)\Big|_{\partial E}=g(\cdot,t)\quad \text{ a.e. }\ t\in(0,T]\\
&u(\cdot,0)=u_o.
\end{aligned}
\right.
\end{equation}
%%%%%
The boundary datum  satisfies
\begin{equation}\tag{D}\label{D}
\left\{
\begin{aligned}
&g\in L^p\big(0,T;W^{1,p}( E)\big),\\ 
&g \text{ continuous on}\ \overline{E}_T\ \text{with modulus of continuity }\boldsymbol\om_g(\cdot).
\end{aligned}
\right.
\end{equation}
In addition, the initial datum satisfies
\begin{equation}\tag{U$_o$}\label{I}
\mbox{ $u_o$ is continuous in $\overline{E}$ with modulus of continuity $\boldsymbol\om_{o}(\cdot)$.}
\end{equation}
%We do not impose any {\it a priori} requirements on the boundary of the domain
%$E\subset\rn$. 
Regarding the geometry of the boundary $\pl E$, we assume a measure density condition, that is,
\begin{equation}\tag{G}\label{geometry}
\left\{\;\;
	\begin{minipage}[c][1.5cm]{0.7\textwidth}
	there exist $\al_*\in(0,1)$ and $\bar\rho\in(0,1)$, such that for all $x_o\in\pl E$,
	for every cube $K_\rho(x_o)$ and $0<\rho\le\bar\rho$, there holds
	$$
	|E\cap K_{\rho}(x_o)|\le(1-\al_*)|K_\rho|.
	$$
	\end{minipage}
\right.
\end{equation}
%%%%%%
Here we have denoted by $K_\rho(x_o)$ the cube of side length $2\rho$
and center $x_o$, with faces parallel with the coordinate planes of $\rr^N$.
Intuitively, the condition \eqref{geometry} means that one can place a cone
at $x_o$ exterior to $E$, with an angle quantified by $\al_*$.

Throughout the rest of this note, 
we will use the symbols 
\begin{equation*}
\left\{
\begin{aligned}
Q_\rho(\theta)&:=K_{\rho}(x_o)\times(t_o-\theta\rho^p,t_o),\\[5pt]
Q_{R,S}&:=K_R(x_o)\times (t_o-S,t_o),
\end{aligned}\right.
\end{equation*} 
to denote (backward) cylinders with the indicated positive parameters; 
we omit the vertex $(x_o,t_o)$
from the notations for simplicity.
We will also denote the lateral boundary by $S_T:=\pl E\times(0,T]$
and the parabolic boundary by $\pl_{\mathcal{P}}E_T:=S_T\cup [\overline{E}\times\{0\}]$.

The formal definition of solution to the Dirichlet problem \eqref{Dirichlet} will be postponed to Section~\ref{S:1:2}.
Now we first present the regularity theorem concerning the Dirichlet problem \eqref{Dirichlet}
up to the parabolic boundary $\pl_{\mathcal{P}}E_T$.
\begin{theorem}\label{Thm:1:1}
Let $u$ be a bounded weak solution to the Dirichlet problem 
\eqref{Dirichlet} under the condition \eqref{Eq:1:2} with $p\ge2$. Assume that \eqref{D}, \eqref{I} and \eqref{geometry} hold. 
Then $u$ is continuous in %any compact set 
$\overline{ E_T}$.
More precisely, there is a modulus of continuity $\boldsymbol\om(\cdot)$,
determined by the data, $\al_*$, $\bar\rho$, $\|u\|_{\infty,E_T}$, $\boldsymbol\om_{o}(\cdot)$ and $\boldsymbol\om_{g}(\cdot)$, such that
	\begin{equation*}
	\big|u(x_1,t_1)-u(x_2,t_2)\big|
	\le
	%\boldsymbol \gm
	%\|u\|_{\infty,E_T}\,
	\boldsymbol\om\!\left(|x_1-x_2|+|t_1-t_2|^{\frac1p}\right),
	\end{equation*}
for every pair of points $(x_1,t_1), (x_2,t_2)\in \overline{E_T}$.
In particular, there exists $\sig\in(0,1)$ depending only on the data and $\al_*$,
such that if  %if $g$ satisfies
\[
\boldsymbol\om_g(r)\le \frac{C_g}{|\ln r|^{\lm}}\quad\text{ and }\quad\boldsymbol \om_o(r)\le \frac{C_{u_o}}{|\ln r|^{\lm}}\quad\text{ for all }r\in(0,\bar\rho),
\]
where $C_g,\,C_{u_o}>0$ and $\lm>\sig$,
then the modulus of continuity is
$$\boldsymbol\om(r)=C \Big(\ln \frac{\bar\rho}{r}\Big)^{-\frac{\sig}2}\quad\text{ for all }r\in(0,  \bar\rho)$$ with some $C>0$
 depending on the data, $\lm$, $\al_*$,   $\|u\|_{\infty,E_T}$, $C_{u_o}$ and $C_{g}$.
\end{theorem}
%%
%\begin{remark}\label{Rmk:Thm:1}\upshape
%We have stated Theorem~\ref{Thm:1:1} in a global fashion. However its proof is entirely local.
%Indeed, we will deal with the oscillation decay of $u$ in the vicinity of the lateral boundary in Sections~\ref{S:2} -- \ref{S:bdry};
%in the interior in Section~\ref{S:interior}; at the initial level in Section~\ref{S:Thm-proof}.
%Afterwards the proof of Theorem~\ref{Thm:1:1} follows from standard covering arguments based on these three cases.
%\end{remark}

Now we consider the Neumann problem. In order to deal with possible variational data on $S_T$,
we assume $\pl E$ is of class $C^1$, %and fulfills the segment property (cf.),
such that the outward unit normal, which we denote by {\bf n},
is defined on $\pl E$ pointwise.
Let us consider the initial-boundary value problem of Neumann type:
\begin{equation}\label{Neumann}
\left\{
\begin{aligned}
&\pl_t\be(u)-\dvg\bl{A}(x,t,u, Du) \ni 0\quad \text{ weakly in }\> E_T\\
&\bl{A}(x,t,u, Du)\cdot {\bf n}=\psi(x,t, u)\quad \text{ a.e. }\ \text{ on }S_T\\
&u(\cdot,0)=u_o(\cdot),
\end{aligned}
\right.
\end{equation}
where the structure conditions \eqref{Eq:1:2} and the initial condition \eqref{I} are retained. 
On the Neumann datum $\psi$ we assume for simplicity that, for some absolute constant $C_2$,
there holds
\begin{equation}\label{N-data}\tag{\bf{N}}
|\psi(x,t, u)|\le C_2\quad \text{ for a.e. }(x,t, u)\in S_T\times\rr.
\end{equation}
More general conditions should also work (cf.~Section~2, Chapter~II, \cite{DB}).
The formal definition of weak solution to \eqref{Neumann} will be given in Section~\ref{S:1:2}.
Now we are ready to present the results concerning regularity of solutions to the Neumann problem
\eqref{Neumann} up to the parabolic boundary $\pl_{\mathcal{P}}E_T$.

\begin{theorem}\label{Thm:1:2}
Let $u$ be a bounded weak solution to the Neumann problem 
\eqref{Neumann} under the condition \eqref{Eq:1:2} with $p=2$. 
Assume that $\pl E$ is of class $C^1$, and \eqref{N-data} and \eqref{I} hold. 
 Then $u$ is continuous in $\overline{E_T}$.
 More precisely, there is a modulus of continuity $\boldsymbol\om(\cdot)$,
determined by the data, the structure of $\pl E$, $C_2$, $\|u\|_{\infty,E_T}$  and $\boldsymbol\om_{o}(\cdot)$, such that
	\begin{equation*}
	\big|u(x_1,t_1)-u(x_2,t_2)\big|
	\le
	%\boldsymbol \gm
	%\|u\|_{\infty,E_T}\,
	\boldsymbol\om\!\left(|x_1-x_2|+|t_1-t_2|^{\frac12}\right),
	\end{equation*}
for every pair of points $(x_1,t_1), (x_2,t_2)\in \overline{E_T}$.
In particular, there exists $\sig\in(0,1)$ depending only on the data, $C_2$ and the structure of $\pl E$,
such that if %if $g$ satisfies
\[
 \boldsymbol\om_o(r)\le \frac{C_{u_o}}{|\ln r|^{\lm}}\quad\text{ for all }r\in(0,1),
\]
where $C_{u_o}>0$ and $\lm>\sig$,
then the modulus of continuity is 
$$\boldsymbol\om(r)=C \Big(\ln \frac{1}{r}\Big)^{-\frac{\sig}2}\quad\text{ for all }r\in(0, 1),$$ with some $C>0$
 depending on the data, the structure of $\pl E$, $C_2$, $\|u\|_{\infty,E_T}$ and $C_{u_o}$.
%More precisely, there exist constants $\boldsymbol\gm>1$ and $\be\in(0,1)$
%determined by the data, $C_2$, $\dist(\mathcal{K};\{t=0\})$ and the structure of $\pl E$, such that
%	\begin{equation*}
%	\big|u(x_1,t_1)-u(x_2,t_2)\big|
%	\le
%	\boldsymbol \gm
%	\|u\|_{\infty,E_T}\,
%	\left(|x_1-x_2|+|t_1-t_2|^{\frac1p}\right)^\be,
%	\end{equation*}
%for every pair of points $(x_1,t_1), (x_2,t_2)\in \mathcal{K}$.
\end{theorem}
\begin{remark}\upshape
The type \eqref{Eq:modulus:2} boundary modulus of continuity for the Neumann problem \eqref{Neumann}
 has been left open in Theorem~\ref{Thm:1:2} for the case $p\neq 2$.  %unless we know $\psi\equiv0$.
This is mainly because of our limited knowledge on the {\it parabolic De Giorgi class} modeled on 
the parabolic $p$-Laplacian \eqref{Eq:p-Laplace}; see further Remark~\ref{Rmk:6:2}. Nevertheless, when $p>2$, we do have
a type \eqref{Eq:modulus:1} boundary modulus for the Neumann problem \eqref{Neumann} with a general $\psi$ as in \eqref{N-data}.
This can be achieved by adapting the proof in \cite{Urbano-00} for the interior regularity.
\end{remark}
\begin{remark}\upshape
We have stated Theorems~\ref{Thm:1:1} -- \ref{Thm:1:2} in a global fashion. However their proofs are entirely local,
and they could be stated near a distinguished part of $\pl_{\mathcal{P}}E_T$.
\end{remark}
\begin{remark}\upshape
 Theorems~\ref{Thm:1:1} -- \ref{Thm:1:2} continue to hold for \eqref{Eq:1:1} %with more general structure conditions and 
 with lower order terms: %similar to the ones in \cite[Chapter~II]{DB}.
 the convection resulting from the heat transfer is reflected by a lower order term.
 Also the function $\be(u)$ could present more general forms
 and reflect thermal properties of the material that may slightly change according to the temperature, 
 as long as it permits only one jump. (For the case of multiple jumps, see \cite{DBV-95}.) %cf. \cite{DB-82}.
 The modifications of the proofs can be modeled on the arguments in \cite{DB,DB-82,DB-86}.
 However, we will not pursue generality in this direction, instead focus will be made on actual novelties.
\end{remark}
%%%%%%%%%%%%%%%%%
\subsection{Definitions of solution}\label{S:1:2}
The notion of local, weak solution to the parabolic $p$-Laplace type equation \eqref{Eq:p-Laplace},
and the notions for its initial--boundary value problems
can be found in \cite[Chapter~II]{DB}. In particular, weak solutions to \eqref{Eq:p-Laplace}
are defined in the function space
\begin{equation}   \label{Eq:func-space-1}
	u\in C\big(0,T;L^2(E)\big)\cap L^p\big(0,T; W^{1,p}(E)\big),
\end{equation}
which do not assume any knowledge on the time derivative.

%Here we will focus on the notions for the equation \eqref{Eq:1:1}.
In contrast to \eqref{Eq:func-space-1}, a special character for the notion of solution  to the Stefan problem \eqref{Eq:1:1}
is that we  require solutions to possess
a time derivative in the Sobolev sense. This is necessary to justify the calculations in the proofs of
the theorems. On the other hand, we will explain in Section~\ref{S:approx} 
how to construct  continuous weak solutions to \eqref{Eq:1:1}
without any knowledge on the time derivative.
\subsubsection{Notion of Local Solution}\label{S:1:2:1}
A function
\begin{equation*}%  \label{Eq:1:3p}
	u\in W_{\loc}^{1,2}\big(0,T;L^2_{\loc}(E)\big)\cap L^p_{\loc}\big(0,T; W^{1,p}_{\loc}(E)\big)
\end{equation*}
is a local, weak sub(super)-solution to \eqref{Eq:1:1} with the structure
conditions \eqref{Eq:1:2}, if for every compact set $K\subset E$ and every sub-interval
$[t_1,t_2]\subset (0,T]$, there is a selection $v\subset\be(u)$, i.e.
\[
\left\{\big(z,v(z)\big): z\in E_T\right\}\subset \left\{\big(z,\be[u(z)]\big): z\in E_T\right\},
\]
 such that
\begin{equation*}%  \label{Eq:1:4p}
	\int_K v\z \,\dx\bigg|_{t_1}^{t_2}
	+
	\iint_{K\times(t_1,t_2)} \big[-v\pl_t\z+\bl{A}(x,t,u,Du)\cdot D\z\big]\dx\dt
	\le(\ge)0
\end{equation*}
for all non-negative test functions
\begin{equation}\label{Eq:function-space}
\z\in W^{1,2}_{\loc}\big(0,T;L^2(K)\big)\cap L^p_{\loc}\big(0,T;W_o^{1,p}(K)%
\big).
\end{equation}
Observe that $v\in L^{\infty}_{\loc} \big(0,T;L^2_{\loc}(E)\big)$ and hence  all the integrals %in \eqref{Eq:1:4p} 
are well-defined. Note that the above integral formulation itself does not involve
the time derivative of $u$.
On the other hand, if we use the time derivative of $u$, the integral formulation may be written as
\begin{equation}\label{Eq:int-form}
\begin{aligned}
-\int_K&\nu(x,t)\chi_{[u\le0]}\z\,\dx\bigg|_{t_1}^{t_2}+\iint_{K\times(t_1,t_2)}\nu(x,t)\chi_{[u\le0]}\pl_t\z\,\dx\dt\\
&+\iint_{K\times(t_1,t_2)} \big[\pl_t u\z+\bl{A}(x,t,u,Du)\cdot D\z\big]\dx\dt
	\le(\ge)0,
\end{aligned}
\end{equation}
where $\z$ is as in \eqref{Eq:function-space} and $\nu(x,t)\ge0$ is given by
\begin{equation*}
\nu(x,t):=\left\{
\begin{array}{cl}
\nu,&\quad (x,t)\in [u<0],\\[5pt]
-v(x,t),&\quad (x,t)\in [u=0],
\end{array}\right.
\end{equation*}
and $\chi$ is the characteristic function of the indicated set.

A function $u$ that is both a local weak sub-solution and a local weak super-solution
to \eqref{Eq:1:1} %-- \eqref{Eq:1:2} 
is a local weak solution.

%\subsubsection{Notion of Parabolicity and Local Boundedness of Solutions}\label{S:1:4:2}
%%The structure condition \eqref{Eq:1:2p} is not
%%sufficient to characterize parabolic equations. 
%%See \cite[Chapter 3, Section 1 \& 5]{DBGV-mono} for a discussion on 
%%the notion of parabolicity.
% For any $k\in\rr$, let
%\[
%(u-k)_-=\max\{-(u-k),0\},\qquad(u-k)_+=\max\{u-k,0\}.
%\]
%Accordingly, we notice that
%\[
%k-(u-k)_-=\min\{u,k\},\qquad k+(u-k)_+=\max\{u,k\}.
%\]
%Using \eqref{Eq:1:2p}$_1$ and employing a similar method as in
%({\bf A}$_6$) of \cite[Chapter II]{DB} or Lemma~1.1 of \cite[Chapter 3]{DBGV-mono},
%we can show that the equation \eqref{Eq:1:1} with \eqref{Eq:1:2}
%is {\it parabolic}, in the sense that 
%\begin{equation*}%\label{Eq:para}
%	\left\{
%	\begin{aligned}
%		&\mbox{whenever $u$ is a local weak sub(super)-solution,}\\ 
%		&\mbox{the function $k\pm(u-k)_\pm$ is a local weak sub(super)-solution, for all $\ k\in\rr$.}
%		%&\mbox{$\bl{A}(x,t,u,Du)$ replaced by $\pm\bl{A}\big(x,t,k\pm(u-k)_\pm,\pm D(u-k)_\pm\big)$ and }\\
%		%&\mbox{$|u|^{p-2}u$ replaced by $\pm\big|k\pm(u-k)_\pm\big|^{p-2}\big(k\pm(u-k)_\pm\big)$. }
%	\end{aligned}
%	\right.
%\end{equation*}
%We will give a proof of this claim in Appendix~\ref{Append:1}.
%In particular, when $u$ is a local weak solution,
% $u_{+}$ and $u_-$ are non-negative, local weak sub-solutions to \eqref{Eq:1:1} -- \eqref{Eq:1:2}.
%Since it has been shown that non-negative, local sub-solutions
%are locally bounded, we may always work with locally bounded solutions.
%See \cite{GV, KK} in this regard.

\subsubsection{Notion of Solution to the Dirichlet Problem}\label{S:1:4:3}
A function
\begin{equation*}  %\label{Eq:1:7}
	u\in W^{1,2}\big(0,T;L^2(E)\big)\cap L^p\big(0,T; W^{1,p}(E)\big)
\end{equation*}
is a weak sub(super)-solution to \eqref{Dirichlet}, 
if there is a selection $v\subset\be(u)$, for every sub-interval
$[t_1,t_2]\subset (0,T]$,
 such that
\begin{equation*} % \label{Eq:1:8}
\begin{aligned}
	\int_{E} v\z \,\dx\bigg|_{t_1}^{t_2}
	&+
	\iint_{E\times(t_1,t_2)} \big[-v\pl_t\z+\bl{A}(x,t,u,Du)\cdot D\z\big]\dx\dt
	\le(\ge)0
\end{aligned}
\end{equation*}
for all non-negative test functions $\z$ satisfying \eqref{Eq:function-space}.
An equivalent form can be given as in \eqref{Eq:int-form} involving $\pl_t u$.
Clearly, now the compact set $K$ in \eqref{Eq:function-space} could touch the boundary $\pl E$, i.e. $K\subset \overline{E}$.
%\begin{equation*}
%\z\in W_{\loc}^{1,p}\big(0,T;L^p(E)\big)\cap L_{\loc}^p\big(0,T;W_o^{1,p}(E)%
%\big).
%\end{equation*}
Moreover, %setting $\hat{p}:=\min\{2,p\}$,
the initial datum is taken in the sense that there exists a selection $v_o\subset\be(u_o)$
and for any compact set $K\subset\overline{E}$,
\[
\int_{K\times\{t\}}(v-v_o)^{2}_{\pm}\,\dx\to0\quad\text{ as }t\to0.
\]
The Dirichlet datum $g$ is attained under $u\le(\ge)g$
on $\pl E$ in the sense that the traces of $(u-g)_{\pm}$
vanish as functions in $W^{1,p}(E)$ for a.e. $t\in(0,T]$, i.e. $(u-g)_{\pm}\in L^p(0,T; W^{1,p}_o(E))$.
Notice that no {\it a priori} information is needed on the smoothness of $\pl E$ for the time being.

A function $u$ that is both a weak sub-solution and a weak super-solution
to \eqref{Dirichlet} is a weak solution.
%%%%%
\subsubsection{Notion of Solution to the Neumann Problem}\label{S:1:4:4}
A function
\begin{equation*}  %\label{Eq:1:7}
	u\in W^{1,2}\big(0,T;L^2(E)\big)\cap L^p\big(0,T; W^{1,p}(E)\big)
\end{equation*}
is a weak sub(super)-solution to \eqref{Neumann}, 
if there is a selection $v\subset\be(u)$, for every compact set $K\subset \rr^N$ and every sub-interval
$[t_1,t_2]\subset (0,T]$,
 such that
\begin{equation*}  %\label{Eq:1:8}
\begin{aligned}
	\int_{K\cap E} v\z \,\dx\bigg|_{t_1}^{t_2}
	&+
	\iint_{\{K\cap E\}\times(t_1,t_2)} \big[-v\pl_t\z+\bl{A}(x,t,u,Du)\cdot D\z\big]\dx\dt\\
	&\le(\ge)\iint_{\{K\cap\pl E\}\times(t_1,t_2)}\psi(x,t,u)\z\,\d\sig\dt
\end{aligned}
\end{equation*}
for all non-negative test functions $\z$ satisfying \eqref{Eq:function-space},
where $K$ is an arbitrary compact set in $\rn$ now,
%\begin{equation*}
%\z\in W_{\loc}^{1,p}\big(0,T;L^p(K)\big)\cap L_{\loc}^p\big(0,T;W_o^{1,p}(K)%
%\big).
%\end{equation*}
and $\d\sig$ denotes the surface measure on $\pl E$.
The Neumann datum $\psi$ is reflected in the boundary integral on the right-hand side.
Moreover, the initial datum is taken %in the sense of $L^p_{\loc}(E)$
as in the Dirichlet problem.
An equivalent form can be given as in \eqref{Eq:int-form} involving $\pl_t u$.
%\[
%\int_{K\times\{t\}}(u-u_o)^p_{\pm}\,dx\to0\quad\text{ as }t\to0.
%\]
%The second of \eqref{Dirichlet} holds under $u\le(\ge)g$
%on $\pl E$ in the sense of the trace of functions in $W^{1,p}(E)$ for a.e. $t\in(0,T]$.

A function $u$ that is both a weak sub-solution and a weak super-solution
to \eqref{Neumann} is a weak solution.
%%%%%%%%%%%%%%%%%%%
\subsection{DiBenedetto's approach: revisit and refinement}
%%Our approach mainly follows that of DiBenedetto in \cite{}. 
%The thrust of DiBenedetto's approach stems from De Giorgi's  measure theoretical method \cite{DG}.
%%The idea of De Giorgi is to derive pointwise estimates from measure information.
%The idea of De Giorgi connects deeply with that of Nash \cite{Nash} and that of Moser \cite{Moser1,Moser2,Moser3}. 
%See also Trudinger~\cite{Trud-67,Trud-68},
%and Han-Lin~\cite[Chapter~4]{Han-Lin} for  a comparison.
%%
%This idea has witnessed many important
%further developments in various forms, cf.~\cite{Caff-Vass, Silv} for instance. 
%In particular, the recent achievement in the degenerate parabolic setting \cite{DBGV-acta, DBGV-mono}
%is pertinent here.
Any known approach to the local continuity issue for the equation \eqref{Eq:1:0} relies on, in one way or another,
De Giorgi's method. %cf.~\cite{Caff-Evans-83, DB-82, DB-86, Sacks-83, Ziemer-82}.
Since there are a lot of technicalities
to follow that might obscure the main ideas behind them, we briefly review the original approach
of DiBenedetto in \cite{DB-82, DB-86} for $p=2$, and highlight the main improvement to the method in order to achieve a type \eqref{Eq:modulus:2} modulus of continuity.

For a cylinder $Q_\rho=K_{\rho}(x_o)\times(t_o-\rho^2,t_o)\subset E_T$, we introduce the numbers $\mu^{\pm}$ and $\om$ satisfying
\begin{equation*}
	\mu^+=\essup_{Q_{\rho}} u,
	\quad 
	\mu^-=\essinf_{Q_{\rho}} u,
	\quad
	\om=\essosc_{Q_{\rho}} u =\mu^+ - \mu^-.
\end{equation*}
The goal is to reduce the oscillation of $u$ over a cylinder smaller than $Q_\rho$ with the same vertex $(x_o,t_o)$.
To this end, we first observe that one of following must hold:
%we may assume that $\mu^+\ge|\mu^-|$, as the other case can be treated similarly.
%As a result we have 
\begin{equation}\label{Eq:mu-pm-0}
\mu^+-\tfrac14\om\ge  \tfrac14\om\quad\text{ or }\quad\mu^-+\tfrac14\om\le -\tfrac14\om.
\end{equation}
Let us suppose the first one, i.e. \eqref{Eq:mu-pm-0}$_1$, holds as the other case is similar.

The first step lies in showing a De Giorgi type lemma, which asserts that there exists a positive constant   $c_o$
depending only on the structural data, such that if
\begin{equation}\label{Eq:DG:measure}
|[u\le\mu^-+\tfrac14\om]\cap Q_{\rho} |\le \al |Q_{\rho}|,\quad\text{ where }\al=c_o\om^{\frac{N+2}2},
\end{equation}
%for some positive constant $c_o$ depending only on the structural data,
then
\[
u\ge \mu^-+\tfrac18\om\quad\text{ a.e. in }Q_{\frac12\rho},
\]
which in turn yields a reduction of oscillation
\[
\essosc_{Q_{\frac12\rho}}u\le\tfrac78\om.
\]
The singularity of $\be(\cdot)$ at $[u=0]$ is reflected by the dependence on $\om$ of $\al$ in \eqref{Eq:DG:measure}.

The second step is to consider the case when \eqref{Eq:DG:measure} does not hold.
Since $\mu^+-\frac14\om\ge \mu^-+\frac14\om$ always holds, the reverse of the measure information \eqref{Eq:DG:measure} implies that
\[
|[\mu^+-u\ge\tfrac14\om]\cap Q_{\rho} |> \al |Q_{\rho}|.
\]
Based on this, it is not hard to check that there exists $t_*\in[t_o-\rho^2, t_o-\tfrac12\al\rho^2]$, such that
\begin{equation}\label{Eq:DG:measure:1}
|[\mu^+-u(\cdot, t_*)\ge\tfrac14\om]\cap K_{\rho}(x_o) |> \tfrac12\al |K_{\rho}|.
\end{equation}

The key to proceeding now is to observe that because of \eqref{Eq:mu-pm-0}$_1$, the function $v$ defined by
$$v:=\tfrac14\om-(u-k)_+\quad\text{ with }k=\mu^+-\tfrac14\om>0,$$
is actually a non-negative, weak super-solution to the   parabolic equation \eqref{Eq:p-Laplace}
with $p=2$ in $Q_\rho$,
as the set $[u>k]$ excludes potential singularity of $\be(u)$. The information in \eqref{Eq:DG:measure:1}
can be rephrased in terms of $v$:
\[
|[v(\cdot, t_*)\ge\tfrac14\om]\cap K_{\rho}(x_o) |> \tfrac12\al |K_{\rho}|.
\]
As such the machinery of De Giorgi
can be applied to $v$ like in \cite[Chapter~II, Section~7]{LSU} 
and the  measure information \eqref{Eq:DG:measure:1} translates into pointwise positivity for $v$
(thus also for $\mu^+-u$) up to the top of the cylinder $Q_{\rho}$. 
This kind of property for non-negative, weak super-solutions is called {\it expansion of positivity}.
%(It is called the {\it density theorem} in \cite[Theorem~4.9]{Han-Lin}.)
More precisely, there holds
\[
\mu^+-u\ge\eta(\al)\om\quad\text{ a.e. in }Q_{\frac12\rho}(\al),
\]
which gives the reduction of oscillation
\[
\essosc_{Q_{\frac12\rho}(\al)}u\le\big(1-\eta(\al)\big)\om.
\]
The singularity of $\be(\cdot)$, reflected in the dependence of $\al$ on $\om$,
has now been passed from the measure information \eqref{Eq:DG:measure:1} to
this reduction of oscillation, through the dependence of $\eta$ on $\al$ and hence on $\om$.
Such dependence of $\eta$ on $\al$ is traced by $\eta\approx2^{-\frac{1}{\al^q}}$,
for a generic positive constant $q$ determined by the data; 
see also \cite[Chapter~4, Section~2]{DBGV-mono}.
Apart from technical complications, these are the main steps in \cite{DB-82}.

Strikingly, the dependence of $\eta$ on $\al$ can be ameliorated in the sense that $\eta\approx \al^{q}$, 
for a generic positive constant $q$ determined by the data.
%for some $\kappa>0$ depending only on the data. 
This was first discovered by DiBenedetto \& Trudinger in \cite{DT}
for the elliptic De Giorgi class via a Krylov-Safonov type covering argument. 
Various parabolic versions have been developed since then; see for instance \cite{DBGV-mono, Liao, W}. 

It is exactly the improvement of the dependence of $\eta$ on $\al$ in the expansion of positivity that
leads to the refinement of the modulus of continuity from  \eqref{Eq:modulus:1} to  \eqref{Eq:modulus:2}.
Indeed, iterating the arguments presented above yields the reduction of oscillation 
\[
\essosc_{Q_n}u\le \om_n,\quad n=0,1,\cdots,
\]
along a nested family of cylinders $\{Q_n\}$
with their common vertex at $(x_o,t_o)$. The sequence $\{\om_n\}$ obeys two types of recurrences,
contingent upon the dependence of $\eta$ on $\al$, recalling also the dependence of $\al$ on $\om$ in \eqref{Eq:DG:measure}:
\begin{equation*}%\label{Eq:omega-n}
\left\{
\begin{array}{lc}
\om_{n+1}=\big(1-2^{-\frac1{\om^q_n}}\big)\om_n,\quad&\text{(Type I)},\\[5pt]
\om_{n+1}=\big(1-\om^q_n\big)\om_n,\quad&\text{(Type II)}.
\end{array}\right.
\end{equation*}
Here we have again used $q>0$ as a generic  constant determined by the data.
An inspection on the sequence $\{\om_n\}$ reveals that for $n$ large
\begin{equation*}%\label{Eq:omega-n}
\left\{
\begin{array}{lc}
\om_{n}\lesssim(\ln n)^{-\sig},\quad&\text{(Type I)},\\[5pt]
\om_{n}\lesssim n^{-\sig},\quad&\text{(Type II)},
\end{array}\right.
\end{equation*}
 for some proper $\sig\in(0,\frac1q)$. If the size of $Q_r\approx Q_n$ is quantified in a geometric fashion,   $r\approx(\tfrac12)^n$ for instance,
 then we have $n\approx\ln\frac1r$. Consequently, the modulus of continuity can be estimated by
 \begin{equation*}%\label{Eq:omega-n}
\left\{
\begin{array}{lc}
\dsty\essosc_{Q_r}u\approx\essosc_{Q_n}u\le\om_{n}\lesssim(\ln n)^{-\sig}\approx(\ln \ln\tfrac1r)^{-\sig},\quad&\text{(Type I)},\\[5pt]
\dsty\essosc_{Q_r}u\approx\essosc_{Q_n}u\le\om_{n}\lesssim n^{-\sig}\approx (\ln\tfrac1r)^{-\sig},\quad&\text{(Type II)}.
\end{array}\right.
\end{equation*}

The type \eqref{Eq:modulus:2} {\it interior} modulus of continuity has been achieved in \cite{Urbano-14}.
Instead of the expansion of positivity, a weak Harnack inequality was employed in \cite{Urbano-14}. 
One purpose of this note is to demonstrate that the expansion of positivity is more flexible when
the type  \eqref{Eq:modulus:2} modulus of continuity   is requested not only in the interior,
but also at the {\it boundary}, given Neumann boundary data.
Needless to say, the expansion of positivity lies at the heart of any kind of Harnack's estimates, cf.~\cite{DBGV-acta, DBGV-mono}.

In the above outline, the singularity of $\be(\cdot)$ at $[u=0]$ is accommodated
by choosing proper levels $k$ for $u$, either close to $\mu^-$ or to $\mu^+$.
However, when the boundary regularity is considered, given general Dirichlet data,
the level $k$ has to obey extra restrictions, such that $(u-k)_\pm$ vanish on the lateral boundary, cf.~\eqref{Eq:k-restriction}.
As a result, the analysis becomes more involved.
The most crucial part of DiBenedetto's approach in \cite{DB-86} 
consists in quantifying the largeness
of $|[u\le0]|$ where the singularity appears. If it is small enough, then the usual parabolic scaling prevails. 
Otherwise, the largeness of $|[u\le0]|$ is incorporated into the singular term on left-hand side of the energy estimate, 
and then balances off the relative largeness brought by the singularity on the right-hand side;  
in this case, the usual parabolic scaling fails and intrinsically scaled cylinders have to be introduced.
Retrospectively, this idea lies at the origin of the {\it method of intrinsic scaling}, cf.~\cite{DB, DBGV-mono, Urbano-08}.

We will adapt this idea, in further balancing the singularity of $\be(\cdot)$ and the degeneracy brought by $\Dl_p$,
and in refining the approach by an argument from 
 \cite[Section~4.4]{DBGV-09} and \cite[Proposition~5.1]{DBG-16},
 which allows us to shrink the measure of the set where $u\approx0$
in a faster fashion than the original De Giorgi method, cf.~Lemma~\ref{Lm:shrink:1}.
We will not try to recapitulate all the steps here beforehand, as they are more involved and delicate
than the interior case. Instead we feel it is more appropriate to add remarks, which may assist understanding
of the technicalities, along the course of the proof.

\

\noi{\it Acknowledgement. } This research has been funded by 
the FWF--Project P31956--N32 “Doubly nonlinear evolution equations”.
%The author is grateful to Ugo Gianazza for carefully reading an early version
%and valuable suggestions.
%%%%%%%%%%%%%%%%%%%%%%%%%%%%%
\section{The Dirichlet problem: preliminary boundary estimates}\label{S:2}
%The proof of Theorem~\ref{Thm:1:1} will be given in Section~\ref{S:Thm-proof}.
%The main components are Proposition~\ref{Prop:bdry:D} in Section~\ref{S:bdry} and Proposition~\ref{Prop:int-reg} in Section~\ref{S:interior}.
We will deal with the oscillation decay of $u$ in the vicinity of the lateral boundary in Sections~\ref{S:2} -- \ref{S:bdry};
in the interior in Section~\ref{S:interior}; at the initial level in Section~\ref{S:Thm-proof}.
Afterwards the proof of Theorem~\ref{Thm:1:1} follows from standard covering arguments based on these three cases
in Section~\ref{S:Thm-proof}.
In the present section, we derive energy estimates for sub(super)-solutions to the Dirichlet problem \eqref{Dirichlet},
near the lateral boundary. Consequently, two De Giorgi type lemmas are implied by the energy estimates.

%%%%%%
\subsection{Energy estimates}\label{S:energy-1}
In this section, we derive some energy estimates near $S_T$,
for solutions to the Dirichlet problem \eqref{Dirichlet}.
We will work
within the cylinder $Q_{R,S}:=K_R(x_o)\times (t_o-S,t_o)$ for some $(x_o,t_o)\in S_T$, such that $t_o-S>0$.
In order to employ the truncated functions $(u-k)_\pm$ across the lateral boundary, we need to impose
certain restrictions on the level $k$, i.e.,
%$Q:=K_R(x_o)\times (t_o-S,t_o)\Subset E_T$
\begin{equation}\label{Eq:k-restriction}
	\left\{
	\begin{aligned}
	&k\ge\sup_{Q_{R,S}\cap S_T}g\quad\text{ for sub-solutions},\\
	&k\le\inf_{Q_{R,S}\cap S_T}g\quad\text{ for super-solutions}.
	\end{aligned}
	\right.
\end{equation}
Let $\z$ be a non-negative, piecewise smooth cutoff function
vanishing on $\pl K_{R}(x_o)\times(t_o-S,t_o)$.
In such a way, the test functions %in \eqref{Test}
\[
Q_{R,S}\cap E_T\ni(x,t)\mapsto\varphi(x,t) 
	= 
	\zeta^p(x,t) \big(u(x,t)-k\big)_\pm
\]
become admissible in the weak formulation of \eqref{Dirichlet}, as the functions $x\mapsto\big(u(x,t)-k\big)_\pm$ vanish
on $Q_{R,S}\cap S_T$ in the sense of traces for a.e. $t\in(t_o-S,t_o)$.
This fact does not require any smoothness of $\pl E$ (cf.~\cite[Lemma~2.1]{GLL}). 

In what follows, when we deal with integrals of $(u-k)_\pm$ across the lateral boundary $S_T$,
the functions $(u-k)_\pm$ will be understood as zero outside $E_T$.
Likewise, when we deal with a sub(super)-solution $u$ across the lateral boundary $S_T$, we tacitly understand it as
\begin{equation*}
u_k^{\pm}:=\left\{
\begin{array}{cl}
k\pm(u-k)_\pm\quad&\text{ in }Q_{R,S}\cap E_T,\\[5pt]
k\quad&\text{ in }Q_{R,S}\setminus E_T,
\end{array}\right.
\end{equation*}
for $k$ satisfying \eqref{Eq:k-restriction}.
%An extension of $\bl{A}$ can be defined as
%\begin{equation*}
%\widetilde{\mathbf A}(x,t,u,\z)\df{=}\left\{
%\begin{array}{cl}
%\mathbf A(x,t,u,\z)\quad&\text{ in }Q_{R,S}\cap E_T,\\[5pt]
%|\z|^{p-2}\z\quad&\text{ in }Q_{R,S}\setminus E_T.
%\end{array}\right.
%\end{equation*}
%In this way, $\widetilde{\mathbf A}$ is a Caratheodory function satisfying \eqref{Eq:1:2p} with $C_o$
%and $C_1$ replaced by $\min\{1,C_o\}$ and $\max\{1,C_1\}$ respectively.

Upon using the above test function $\vp$ in the boundary version of 
\eqref{Eq:int-form}, we have the following energy estimate; the detailed calculation is analogous to the one for \cite[(2.7)]{DB-82}.
%To this end, we introduce
\begin{proposition}\label{Prop:2:1}
	Let $u$ be a  local weak sub(super)-solution to \eqref{Dirichlet} with \eqref{Eq:1:2} in $E_T$.
	There exists a constant $\gm (C_o,C_1,p)>0$, such that
 	for all cylinders $Q_{R,S}=K_R(x_o)\times (t_o-S,t_o)$ with the vertex $(x_o,t_o)\in S_T$,
 	every $k\in\rr$ satisfying \eqref{Eq:k-restriction}, and every non-negative, piecewise smooth cutoff function
 	$\z$ vanishing on $\pl K_{R}(x_o)\times(t_o-S,t_o)$,  there holds
\begin{align*}
	\essup_{t_o-S<t<t_o}&\Big\{\int_{K_R(x_o)\times\{t\}}	
	\z^p (u-k)_\pm^2\,\dx +\Phi_{\pm}(k, t_o-S,t,\z)\Big\}\\
	&\quad+
	\iint_{Q_{R,S}}\z^p|D(u-k)_\pm|^p\,\dx\dt\\
	&\le
	\gm\iint_{Q_{R,S}}
		\Big[
		(u-k)^{p}_\pm|D\z|^p + (u-k)_\pm^2|\partial_t\z^p|
		\Big]
		\,\dx\dt\\
	&\quad
	+\int_{K_R(x_o)\times \{t_o-S\}} \z^p (u-k)_\pm^2\,\dx,
\end{align*}
where 
\begin{align*}
\Phi_{\pm}(k, t_o-S,t,\z)=&-\int_{K_R(x_o)\times\{\tau\}}\nu(x,\tau)\chi_{[u\le0]}[\pm(u-k)_\pm\z^p]\,\dx\Big|_{t_o-S}^t\\
&+\int_{t_o-S}^t\int_{K_R(x_o)}\nu(x,\tau)\chi_{[u\le0]}\pl_t[\pm(u-k)_\pm\z^p]\,\dx\d\tau
\end{align*}
and $\nu(x,t)$ is a selection out of $[0,\nu]$ for $u(x,t)=0$, and $\nu(x,t)=\nu$ for $u(x,t)<0$.
\end{proposition}
Based on the energy estimate in Proposition~\ref{Prop:2:1}, we may derive 
 two tailored versions. Note that the functions $\Phi_\pm$ carry information near the singularity of $\be(\cdot)$.
The first tailored version examines the case when $\Phi_\pm=0$, and hence  singularity disappears and
it recovers the energy estimate for the parabolic $p$-Laplacian \eqref{Eq:p-Laplace}.
\begin{proposition}\label{Prop:2:2}
Let the hypotheses in Proposition~\ref{Prop:2:1} hold.
If, in addition, $k\ge0$ in the case of sub-solutions or $k\le0$ in the case of super-solutions,
then for every non-negative, piecewise smooth cutoff function
 	$\z$ vanishing on $\pl_{\mathcal{P}} Q_{R,S}$, there holds
\begin{align*}
	\essup_{t_o-S<t<t_o}&\int_{K_R(x_o)\times\{t\}}	
	\z^p (u-k)_\pm^2\,\dx 
	+
	\iint_{Q_{R,S}}\z^p|D(u-k)_\pm|^p\,\dx\dt\\
	&\le
	\gm\iint_{Q_{R,S}}
		\Big[
		(u-k)^{p}_\pm|D\z|^p + (u-k)_\pm^2|\partial_t\z^p|
		\Big]
		\,\dx\dt.%\\
%	&\quad
%	+\int_{\{K_R(x_o)\cap E\}\times \{t_o-S\}} \z^p (u-k)_\pm^2\,\dx,
\end{align*}
\end{proposition}
\begin{proof}
We need only to notice that under the current restrictions on $k$, 
the functions $\Phi_{\pm}$ vanish. Also now the cutoff function $\z$ is required to vanish at the bottom level $t_o-S$.
\end{proof}
The second one examines the effect of the singularity of $\be(\cdot)$ more carefully.
\begin{proposition}\label{Prop:2:3}
Let the hypotheses in Proposition~\ref{Prop:2:1} hold.
If, in addition, $k\ge0$ in the case of super-solutions,
then for every non-negative, piecewise smooth cutoff function
 	$\z$ vanishing on $\pl_{\mathcal{P}} Q_{R,S}$, there holds
\begin{align*}
	\essup_{t_o-S<t<t_o}&\Big\{\int_{K_R(x_o)\times\{t\}}	
	\z^p (u-k)_-^2\,\dx +k\int_{K_R(x_o)\times\{t\}}\nu\chi_{[u\le0]}\z^p\,\dx\Big\}\\
	&\quad+
	\iint_{Q_{R,S}}\z^p|D(u-k)_-|^p\,\dx\dt\\
	&\le
	\gm\iint_{Q_{R,S}}
		\Big[
		(u-k)^{p}_-|D\z|^p + (u-k)_-^2 |\partial_t\z^p|
		\Big]
		\,\dx\dt\\
		&\quad+\iint_{Q_{R,S}}\nu\chi_{[u\le0]}(k-u)|\pl_t\z^p|\,\dx\dt.
%	&\quad
%	+\int_{\{K_R(x_o)\cap E\}\times \{t_o-S\}} \z^p (u-k)_\pm^2\,\dx,
\end{align*}
\end{proposition}
\begin{proof}
Let us omit the reference to $x_o$ for simplicity.
We only need to calculate
\begin{align*}
\Phi_{-}&(k, t_o-S,t,\z)=\int_{K_R}\nu(x,\tau)\chi_{[u\le0]}[(u-k)_-\z^p]\,\dx\Big|_{t_o-S}^t\\
&\quad-\iint_{Q_{R,S}}\nu(x,\tau)\chi_{[u\le0]}\pl_t[(u-k)_-\z^p]\,\dx\d\tau\\
&=k\int_{K_R}\nu\chi_{[u\le0]}\z^p\,\dx\Big|_{t_o-S}^t+\int_{K_R}\nu u_-\z^p\,\dx\Big|_{t_o-S}^t\\
&\quad-\iint_{Q_{R,S}}\nu \pl_tu_-\z^p\,\dx
 -\iint_{Q_{R,S}}\nu \chi_{[u\le0]}(u-k)_-\pl_t\z^p\,\dx\d\tau\\
&\ge k\int_{K_R\times\{t\}}\nu\chi_{[u\le0]}\z^p\,\dx
-\iint_{Q_{R,S}}\nu \chi_{[u\le0]}(k-u)\pl_t\z^p\,\dx\d\tau.
\end{align*}
Then substituting it back to the energy estimate in Proposition~\ref{Prop:2:1}, we conclude.
\end{proof}
%%%%
\subsection{De Giorgi type lemmas}\label{S:DG}
%For a compact set $K\subset\rr^N$ and
For $(x_o,t_o)\in S_T$ and a cylinder $Q_{R,S}=K_{R}(x_o)\times(t_o-S,t_o)$ as in Section~\ref{S:energy-1},
we introduce the numbers $\mu^{\pm}$ and $\om$ satisfying
\begin{equation*}
	\mu^+\ge\essup_{Q_{R,S}} u,
	\quad 
	\mu^-\le\essinf_{Q_{R,S}} u,
	\quad
	\om\ge\mu^+ - \mu^-.
\end{equation*}
Let $\widetilde\theta\ge\theta>0$ be  parameters to be determined.
The cylinders $Q_\rho(\theta)$ and $Q_\rho(\widetilde\theta)$ are coaxial with $Q_{R,S}$ and with the same vertex $(x_o,t_o)$;
moreover, we will impose that 
\[
\rho<R, \quad\widetilde\theta(8\rho)^p<S,\quad\text{ such that }\quad Q_\rho(\theta)\subset Q_\rho(\widetilde\theta)\subset Q_{R,S}.
\]

We now present the first De Giorgi type lemma.
\begin{lemma}\label{Lm:DG:1}
Let $u$ be a   weak super-solution to \eqref{Dirichlet} with \eqref{Eq:1:2} in $E_T$. For $\xi\in(0,1)$,
set $\theta=(\xi\om)^{2-p}$ and suppose that $$\mu^-+\xi\om\le \inf_{Q_{R,S}\cap S_T}g.$$
There exists a  constant $c_o\in(0,1)$ depending only on the data, such that if
\[
|[u\le\mu^-+\xi\om]\cap Q_{\rho}(\theta)|\le c_o(\xi\om)^{\frac{N+p}p}|Q_{\rho}(\theta)|,
\]
then
\[
u\ge \mu^-+\tfrac12\xi\om\quad\text{ a.e. in }Q_{\frac12\rho}(\theta).
\]
\end{lemma}
\begin{proof}
Let us assume $(x_o,t_o)=(0,0)$.
In order to use the energy estimates in Propositions~\ref{Prop:2:2} -- \ref{Prop:2:3}, we set
\begin{align}\label{choices:k_n}
	\left\{
	\begin{array}{c}
	\displaystyle k_n=\mu^-+\frac{\xi\om}2+\frac{\xi\om}{2^{n+1}},\quad \tilde{k}_n=\frac{k_n+k_{n+1}}2,\\[5pt]
	\displaystyle \varrho_n=\frac{\varrho}2+\frac{\varrho}{2^{n+1}},
	\quad\tilde{\varrho}_n=\frac{\varrho_n+\varrho_{n+1}}2,\\[5pt]
	\displaystyle K_n=K_{\varrho_n},\quad \widetilde{K}_n=K_{\tilde{\varrho}_n},\\[5pt] 
	\displaystyle Q_n=Q_{\rho_n}(\theta),\quad
	\widetilde{Q}_n=Q_{\tilde\rho_n}(\theta).
	\end{array}
	\right.
\end{align}
Recall that $Q_{\rho_n}(\theta)=K_n\times(-\theta\varrho_n^p,0)$ and $Q_{\tilde\rho_n}(\theta)=\widetilde{K}_n\times(-\theta\tilde{\varrho}_n^p,0)$. 
Introduce the cutoff function $0\le\z\le 1$ vanishing on the parabolic boundary of $Q_{n}$ and
equal to identity in $\widetilde{Q}_{n}$, such that
\begin{equation*}
	|D\z|\le \gm\frac{2^n}{\varrho}
	\quad\text{and}\quad 
	|\pl_t\z|\le \gm\frac{2^{pn}}{\theta\varrho^p}.
\end{equation*}
In this setting, the energy estimates in Propositions~\ref{Prop:2:2} -- \ref{Prop:2:3} combined will yield that
\begin{align*}
	\essup_{-\theta\tilde{\varrho}_n^p<t<0}
	&\int_{\widetilde{K}_n} (u-\tilde{k}_n)_-^2\,\dx
	+
	\iint_{\widetilde{Q}_n}|D(u-\tilde{k}_n)_-|^p \,\dx\dt\\
	&\qquad\le
	 \gm \frac{2^{pn}}{\varrho^p}(\xi\om)^{p-1}|A_n|,
\end{align*}
where
\begin{equation*}
	A_n=\big[u<k_n\big]\cap Q_n.
\end{equation*}
Now setting $0\le\phi\le1$ to be a cutoff function which vanishes on the parabolic boundary of $\widetilde{Q}_n$
and equals the identity in $Q_{n+1}$, an application of the H\"older inequality  and the Sobolev imbedding
\cite[Chapter I, Proposition~3.1]{DB} gives that
\begin{align*}
	\frac{\xi\om}{2^{n+3}}
	|A_{n+1}|
	&\le 
	\iint_{\widetilde{Q}_n}\big(u-\tilde{k}_n\big)_-\phi\,\dx\dt\\
	&\le
	\bigg[\iint_{\widetilde{Q}_n}\big[\big(u-\tilde{k}_n\big)_-\phi\big]^{p\frac{N+2}{N}}
	\,\dx\dt\bigg]^{\frac{N}{p(N+2)}}|A_n|^{1-\frac{N}{p(N+2)}}\\
	&\le \gm
	\bigg[\iint_{\widetilde{Q}_n}\big|D\big[(u-\tilde{k}_n)_-\phi\big]\big|^p\,
	\dx\dt\bigg]^{\frac{N}{p(N+2)}}\\
	&\quad\ 
	\times\bigg[\essup_{-\widetilde{\theta}\tilde{\varrho}_n^p<t<0}
	\int_{\widetilde{K}_n}\big(u-\tilde{k}_n\big)^{2}_-\,\dx\bigg]^{\frac{1}{N+2}}
	 |A_n|^{1-\frac{N}{p(N+2)}}\\
	&\le 
	%\gm (\xi\om)^{-\frac{b}{N+2}}(\dl\xi\om)^{\frac1{N+2}}\\
	\gm\bigg[(\xi\om)^{p-1}\frac{2^{pn}}{\rho^p}\bigg]^{\frac{N+p}{p(N+2)}}
	|A_n|^{1+\frac{1}{N+2}} %\\
\end{align*}
In the second-to-last line we used the above energy estimate.
In terms of $  Y_n=|A_n|/|Q_n|$, this can be rewritten as
\begin{equation*}
	 Y_{n+1}
	\le
	\frac{\gm C^n}{(\xi\om)^{\frac{N+p}{p(N+2)}}} Y_n^{1+\frac{1}{N+2}},
\end{equation*}
for positive constants $\gm$ depending only on the data and  $C=C(p,N)$.
Hence, by \cite[Chapter I, Lemma~4.1]{DB}, 
there exists
a positive constant $c_o$ depending only on the data, such that
$Y_n\to0$ if we require that $Y_o\le c_o(\xi\om)^{\frac{N+p}p}$, which amounts to
\begin{equation*}
	|A_o|=\big|\big[u<k_o\big]\cap Q_o\big|
	= 
	\Big|\Big[u<\mu^-+\tfrac12 \xi\om\Big]\cap Q_{\varrho}(\theta)\Big|
	\le
	c_o(\xi\om)^{\frac{N+p}p}\big| Q_{\varrho}(\theta)\big|.
\end{equation*}
Since $Y_n\to 0$ as $n\to\infty$, we have
\begin{equation*}
	\Big|\Big[u<\mu^-+\tfrac12 \xi\om \Big]\cap Q_{\frac12 \varrho}(\theta)\Big|
	= 
	0.
\end{equation*}
This concludes the proof of the lemma.
\end{proof}
The second De Giorgi type lemma is as follows.
\begin{lemma}\label{Lm:DG:2}
Let the hypotheses in Lemma~\ref{Lm:DG:1} hold. There exists $\xi\in(0,1)$ depending only on the data and $\al_*$,
such that if 
\[
\iint_{Q_{\rho}(\theta)}\nu\chi_{[u\le0]}\,\dx\dt
%|[u\le 0]\cap Q_\rho(\theta)|
\le\xi\om|[u\le \mu^-+\tfrac12\xi\om]\cap Q_{\frac12\rho}(\theta)|
\]
and
\[
|\mu^-|\le\xi\om,
\]
then
\[
u\ge\mu^-+\tfrac14\xi\om\quad\text{ a.e. in }Q_{\frac12\rho}(\theta).
\]
\end{lemma}
\begin{proof}
We tend to use the energy estimate in Proposition~\ref{Prop:2:3} with $Q_{R,S}=Q_r(\theta)$
for $\frac12\rho\le r\le \rho$, and with $k=\mu^-+\xi\om\ge0$.
To this end, let us first estimate the last integral on the right-hand side of
 the energy estimate via the given measure theoretical information:
\begin{align*}
\iint_{Q_{r}(\theta)}&\nu\chi_{[u\le0]}(k-u) |\pl_t\z^p|\,\dx\dt\\
&\le(k-\mu^-)\iint_{Q_{r}(\theta)}\nu\chi_{[u\le0]} |\pl_t\z^p|\,\dx\dt\\
&\le\xi\om  \|\pl_t\z^p\|_{\infty} \iint_{Q_{\rho}(\theta)}\nu\chi_{[u\le0]}  \,\dx\dt\\
&\le (\xi\om)^2  \|\pl_t\z^p\|_{\infty} |[u\le \mu^-+\tfrac12\xi\om]\cap Q_{\frac12 \rho}(\theta)|\\
&\le 4 \|\pl_t\z^p\|_{\infty} \iint_{Q_{r}(\theta)}[u-(\mu^-+\xi\om)]^2_-\,\dx\dt.
\end{align*}
This means that the integral can be combined with a similar term involving $\pl_t\z^p$ on the right-hand side
of the energy estimate in Proposition~\ref{Prop:2:3}.
Consequently, we obtain an energy estimate of $(u-k)_-$ from which the theory of parabolic $p$-Laplacian applies, 
cf.~Lemmas~\ref{Lm:reduc-osc-1} -- \ref{Lm:reduc-osc-2} and Appendix~\ref{A:1}.
Therefore we obtain a constant $\xi$ determined only by the data and $\al_*$ of \eqref{geometry}, such that
\[
u\ge\mu^-+\tfrac14\xi\om\quad\text{ a.e. in }Q_{\frac12\rho}(\theta).
\]
The proof is now completed.
\end{proof}
%%%
\begin{remark}\label{Rmk:2:1}\upshape
Before heading to the reduction of boundary oscillation,  %in the next section,
we try to shed some light on the roles of Lemmas~\ref{Lm:DG:1} -- \ref{Lm:DG:2}
as a preparation for the technically involved Section~\ref{S:bdry}. 
Both lemmas are examining the {\it measure information} needed
in order for the degeneracy of $\Dl_p$ to dominate the singularity of $\be(\cdot)$.
As such the second term on the left-hand side of the energy estimate in Proposition~\ref{Prop:2:3} is discarded and
 the time scaling $\theta=(\xi\om)^{2-p}$ is employed to accommodate the degeneracy of $\Dl_p$; 
 the reduction of oscillation in this case will be demonstrated in Section~\ref{S:reduc-osc-1}. 
Attention is called to the difference between Section~\ref{S:reduc-osc-1} and Section~\ref{S:reduc-osc-4}, where the same time scaling
is evoked. Indeed, {\it pointwise information} is imposed in Section~\ref{S:reduc-osc-4} to avoid the singularity of $\be(\cdot)$ completely,
such that the case falls into the scope of the theory for the parabolic $p$-Laplacian in \cite{DB}.
On the other hand, violation of the measure information in    Lemmas~\ref{Lm:DG:1} -- \ref{Lm:DG:2} means that
the singularity of $\be(\cdot)$ overtakes the degeneracy of $\Dl_p$ and
will yield an estimate from below on $|[u\le0]|$ (see \eqref{Eq:opposite:3}), which bears the singularity of $\be(\cdot)$.
This lower bound of $|[u\le0]|$ in turn gives a lower bound for the second term on the left-hand side of the energy estimate in Proposition~\ref{Prop:2:3}; in this case the first term on the left-hand side of the energy estimate in Proposition~\ref{Prop:2:3}
is discarded (cf.~Lemma~\ref{Lm:energy}) and consequently, 
the ``longer" time scaling $\widetilde\theta=(\dl\xi\om)^{1-p}$ will be needed to accommodate
the singularity, cf.~Remark~\ref{Rmk:3:2}.
The reduction of oscillation in this case will be taken on in Section~\ref{S:reduc-osc-2}.
\end{remark}
%%%%%%%%%%%%%%%%%%%%%%%%%%%%%%%%%%%%%%
\section{The Dirichlet problem: reduction of boundary oscillation}\label{S:bdry}
Let $u$ be a solution to the Dirichlet problem \eqref{Dirichlet}.
Fix $(x_o,t_o)\in S_T$ and define the numbers $\mu^\pm$ and $\om$ as in Section~\ref{S:DG}
with $Q_{R,S}$ replaced by 
$$\widetilde{Q}_o:=K_{8\rho}(x_o)\times(t_o-\rho^{p-1}, t_o),\quad \text{for some }\rho\in(0,\tfrac18\bar\rho],$$
where $\bar\rho$ is defined in \eqref{geometry}, which we may assume to be so small that $t_o\ge (\frac18\bar\rho)^{p-1}$.
We now state the main result of this section, which describes the oscillation decay
of $u$, along shrinking cylinders with the same vertex $(x_o,t_o)$ fixed on the lateral boundary $S_T$.

\begin{proposition}\label{Prop:bdry:D}
Suppose the hypotheses in Theorem~\ref{Thm:1:1} hold true.
There exists  $\sig\in(0,1)$ depending only on the data and $\al_*$,
such that if $g$ verifies
\[
\osc_{K_{r}(x_o)\times(t_o-r^{p-1}, t_o)}g\le \frac{C_g}{|\ln r|^{\lm}}\quad\text{ for all }r\in(0,\bar\rho),
\]
with $C_g>0$ and $\lm>\sig$,
then for some $\gm>1$ depending only on the data,  $\al_*$ and $C_g$,
\[
\essosc_{Q_r(\om^{2-p})}u\le\gm\max\{1,\om\}\big(1-\ln\min\{1,\om\}\big)^{\frac{\sig}2}\Big(\ln\frac{\bar\rho}{r}\Big)^{-\frac{\sig}2}\quad\text{ for all }r\in(0,\bar\rho),
\]
where $Q_r(\om^{2-p})=K_r(x_o)\times(t_o-\om^{2-p}r^p,t_o)$.
\end{proposition}
The proof of Proposition~\ref{Prop:bdry:D} will be expanded on in Sections~\ref{S:reduc-osc-1} -- \ref{S:bdry-modulus}.
To start with, we introduce the following intrinsic cylinders 
\begin{equation*}
\left\{
\begin{array}{ll}
Q_{\rho}(\theta)=K_\rho(x_o)\times(t_o-\theta\rho^p,t_o),\quad\theta=(\xi\om)^{2-p},\\[5pt]
Q_{\rho}(\widetilde\theta)=K_\rho(x_o)\times(t_o-\widetilde\theta\rho^p,t_o),\quad \widetilde\theta=(\dl\xi\om)^{1-p},
\end{array}\right.
\end{equation*}
%with the same vertex $(x_o,t_o)$. Here $\theta=(\xi\om)^{2-p}$
%and $\widetilde\theta=(\dl\xi\om)^{1-p}$
where $\dl=\bar\xi\om^q$ for some positive $\bar\xi$ and $q$, which together with the positive $\xi$  are 
 to be determined in terms of the data and $\al_*$. 
We may assume with no loss of generality that
\begin{equation}\label{Eq:start-cylinder}
\widetilde\theta(8\rho)^p\le\rho^{p-1}, \quad\text{ such that }\quad\essosc_{Q_{8\rho}(\widetilde{\theta})}u\le \om,
\end{equation}
since otherwise we would have
\begin{equation}\label{Eq:rho-control}
8^{\frac{p}{1-p}}\xi\bar\xi\om^{1+q}\le\rho^{\frac1{p-1}}.
\end{equation}
Moreover, one of the following must hold:
%$\mu^+\ge|\mu^-|$.%
%As a result, one easily verifies that
\begin{equation}\label{Eq:mu-pm}
\mu^+-\tfrac14\om\ge\tfrac14\om\quad\text{ or }\quad \mu^-+\tfrac14\om\le-\tfrac14\om.
\end{equation}
Let us take  for instance the first one, i.e. \eqref{Eq:mu-pm}$_1$, as the other one can be treated analogously.

Next, we observe that one of the following two cases must hold:
\begin{equation}\label{Eq:alternative}
\mu^+-\tfrac14\om\ge\sup_{\widetilde{Q}_o\cap S_T}g\quad\text{ or }\quad \mu^-+\tfrac14\om\le\inf_{\widetilde{Q}_o\cap S_T}g.
\end{equation}
Otherwise, the violation of both inequalities implies that 
\begin{equation}\label{Eq:g-control}
\om\le 2\osc_{\widetilde{Q}_o\cap S_T}g.
\end{equation}

If \eqref{Eq:alternative}$_1$ holds, then thanks to \eqref{Eq:mu-pm}$_1$, the function $(u-k)_+$ with $k=\mu^+-2^{-n}\om$ will
satisfy the energy estimate in Proposition~\ref{Prop:2:2} for all $n\ge2$.
In this case, the reduction of oscillation is achieved in Section~\ref{S:reduc-osc-4}, cf.~Lemma~\ref{Lm:reduc-osc-1}.
Consequently we need only to examine the case when \eqref{Eq:alternative}$_2$ holds.
This case splits into two sub-cases.
In Sections~\ref{S:reduc-osc-1} -- \ref{S:reduc-osc-3}  we deal with the situation when $|\mu^-|$
is near zero; whereas the opposite situation is treated in Section~\ref{S:reduc-osc-4}, cf.~Lemma~\ref{Lm:reduc-osc-2}. 

In summary, Sections~\ref{S:reduc-osc-1} -- \ref{S:reduc-osc-3} will be devoted to reducing 
the oscillation of $u$ near the set $[u\approx0]$, where the singularity of $\be(u)$ appears;
Section~\ref{S:reduc-osc-4} treats the cases when the singularity of $\be(u)$ is avoided,
and hence the behavior of $u$ resembles the one to the parabolic $p$-Laplacian \eqref{Eq:p-Laplace}
and the theory in \cite{DB} applies.

%If \eqref{Eq:alternative}$_2$ holds and $\mu+\frac14\om<0$, then the function $(u-k)_-$ with $k=\mu^-+2^{-n}\om$ will
%satisfy the energy estimate in Proposition~\ref{Prop:2:2} for all $n\ge2$. Thus this case again falls into the scope of Appendix~\ref{}.
%
%Now we consider the case when the following occurs,
%\[
%\mu^+-\tfrac14\om<\sup_{Q_{\rho}\cap S_T}g
%\]
%%%%%%%%%
\subsection{Reduction of oscillation near zero}\label{S:reduc-osc-1}
Let us deal with the case when \eqref{Eq:alternative}$_2$ holds. Throughout Sections~\ref{S:reduc-osc-1} -- \ref{S:reduc-osc-3}, 
we always assume that
\begin{equation}\label{Eq:mu-minus}
|\mu^-|\le\tfrac12\dl\xi\om\quad\text{ where }\dl=\bar\xi\om^q,
\end{equation} 
for some positive parameters $\{\xi,\,\bar\xi,\, q\}$ to be determined in terms of the data and $\al_*$ only.
When this restriction \eqref{Eq:mu-minus} does not hold, the case will be examined in Section~\ref{S:reduc-osc-4}.
In Sections~\ref{S:reduc-osc-1} -- \ref{S:reduc-osc-3}, 
we will work with $u$ as a super-solution and reduce its oscillation near its infimum $\mu^-$
 in a smaller, quantified cylinder with the same vertex $(x_o,t_o)$. 

To this end, we turn our attention to
Lemma~\ref{Lm:DG:1} and Lemma~\ref{Lm:DG:2}. 
Suppose $\xi$ is determined in Lemma~\ref{Lm:DG:2} in terms of the data and $\al_*$,
assume $\xi<\frac14$ with no loss of generality, and let $\theta=(\xi\om)^{2-p}$.
If
\[
|[u\le\mu^-+\tfrac12\xi\om]\cap Q_{\frac12\rho}(\theta)|\le c_o(\xi\om)^{\frac{N+p}p}|Q_{\frac12\rho}(\theta)|,
\]
then by Lemma~\ref{Lm:DG:1} we have
\begin{equation}\label{Eq:lower-bd-1}
u\ge \mu^-+\tfrac14\xi\om\quad\text{ a.e. in }Q_{\frac14\rho}(\theta).
\end{equation}
Likewise, if 
\begin{equation}\label{Eq:measure-2}
%|[u\le 0]\cap Q_\rho(\theta)|
\iint_{Q_{\rho}(\theta)}\nu\chi_{[u\le0]}\,\dx\dt\le\xi\om|[u\le \mu^-+\tfrac12\xi\om]\cap Q_{\frac12\rho}(\theta)|
\end{equation}
then by Lemma~\ref{Lm:DG:2} we have
\begin{equation}\label{Eq:lower-bd-2}
u\ge \mu^-+\tfrac14\xi\om\quad\text{ a.e. in }Q_{\frac12\rho}(\theta),
\end{equation}
since $|\mu^-|\le\xi\om$ always holds true in this section. 
Hence either \eqref{Eq:lower-bd-1} or \eqref{Eq:lower-bd-2} yields a reduction of oscillation
\begin{equation}\label{Eq:reduc-osc-1}
\essosc_{Q_{\frac14\rho}(\theta)}u\le(1-\tfrac14\xi)\om.
\end{equation}
%%%%%%%%%%%%%%%%%%
\subsection{Reduction of oscillation near zero continued}\label{S:reduc-osc-2}
In this section, we continue to examine the situation when the measure condition in Lemma~\ref{Lm:DG:1} is violated:
\begin{equation}\label{Eq:opposite:1}
|[u\le\mu^-+\tfrac12\xi\om]\cap Q_{\frac12\rho}(\theta)|> c_o(\xi\om)^{\frac{N+p}p}|Q_{\frac12\rho}(\theta)|,
\end{equation}
and when the condition in Lemma~\ref{Lm:DG:2} is violated:
\begin{equation}\label{Eq:opposite:2}
 %|[u\le 0]\cap Q_\rho(\theta)|
 \iint_{Q_{\rho}(\theta)}\nu\chi_{[u\le0]}\,\dx\dt>\xi\om|[u\le \mu^-+\tfrac12\xi\om]\cap Q_{\frac12\rho}(\theta)|.
 %\quad\text{ or }\quad|\mu^-|>\xi\om.
\end{equation}

%If \eqref{Eq:opposite:2}$_2$, we may use Appendix~\ref{} to reduce the oscillation.
%Consequently let us assume that \eqref{Eq:opposite:1} and \eqref{Eq:opposite:2}$_1$ both hold.
Combining \eqref{Eq:opposite:1} and \eqref{Eq:opposite:2}, we may obtain that for all $r\in[2\rho, 8\rho]$,
\begin{equation}\label{Eq:opposite:3}
\begin{aligned}
%|[u\le0]\cap Q_{r}(\theta)|&\ge|[u\le 0]\cap Q_\rho(\theta)|\\
\iint_{Q_{r}(\theta)}\nu\chi_{[u\le0]}\,\dx\dt&\ge\iint_{Q_{\rho}(\theta)}\nu\chi_{[u\le0]}\,\dx\dt\\
&\ge\xi\om|[u\le \mu^-+\tfrac12\xi\om]\cap Q_{\frac12\rho}(\theta)|\\
&\ge c_o(\xi\om)^b|Q_{\frac12\rho}(\theta)|\ge\widetilde{\gm}(\xi\om)^b|Q_r(\theta)|,
\end{aligned}
\end{equation}
where $\widetilde{\gm}=c_o16^{-N-p}$ and
$b=1+\tfrac{N+p}p$.% \quad\gm_o:=\widetilde{\gm}\xi^b
%for some $\widetilde{\gm}$ depending only on the data.

To proceed, let $\bar\dl\in(\dl,1)$ be a free parameter and set $\bar\theta:=(\bar\dl\xi\om)^{1-p}$.
Recall that  $\widetilde{\theta}=(\dl\xi\om)^{1-p}$, $\theta=(\xi\om)^{2-p}$,
and  we have assumed $\widetilde{\theta}(8\rho)^p<\rho^{p-1}$ in \eqref{Eq:start-cylinder}.
Therefore 
\[
Q_r(\theta)\subset Q_r(\bar{\theta})\subset Q_r(\widetilde{\theta})\subset \widetilde{Q}_o
\quad\text{ for any } r\in[2\rho, 8\rho].
\]
The estimate \eqref{Eq:opposite:3} implies that there exists $t_*\in [t_o-\theta r^p,t_o]$,
such that
\begin{equation}\label{Eq:time:1}
%|[u(\cdot, t_*)\le 0]\cap K_{r}|
\int_{K_{r}\times\{t_*\}}\nu\chi_{[u\le0]}\,\dx\ge\widetilde{\gm}(\xi\om)^b|K_r|.
\end{equation}
 Observe also that for any $t\in[t_o-\bar{\theta} r^p, t_o]$ and any $\bar\dl\in(\dl,1)$, there holds
 \begin{equation}\label{Eq:time:2}
\begin{aligned}
|K_r|&\ge|[u\le\mu^-+\bar\dl\xi\om]\cap K_r|\\
&\ge(\bar\dl\xi\om)^{-2}\int_{K_r\times\{t\}}[u-(\mu^-+\bar\dl\xi\om)]_-^2\,\dx.
\end{aligned}
\end{equation}
Combining \eqref{Eq:time:1} and \eqref{Eq:time:2} gives that
\begin{equation*}%\label{Eq:time:3}
\begin{aligned}
&\essup_{t_o-\bar{\theta} r^p<t<t_o} \int_{K_{r}\times\{t\}}\nu\chi_{[u\le0]}\,\dx\\
&\quad\ge \widetilde{\gm}(\xi\om)^b
(\bar\dl\xi\om)^{-2}\essup_{t_o-\bar{\theta} r^p<t<t_o}\int_{K_r\times\{t\}}[u-(\mu^-+\bar\dl\xi\om)]_-^2\,\dx.
\end{aligned}
\end{equation*}
The above analysis together with Proposition~\ref{Prop:2:3}
yields the following  energy estimate.
\begin{lemma}\label{Lm:energy}
Let $u$ be a  weak super-solution to \eqref{Dirichlet} with \eqref{Eq:1:2} in $E_T$.
Let \eqref{Eq:mu-pm}$_1$, \eqref{Eq:alternative}$_2$, \eqref{Eq:opposite:1} and \eqref{Eq:opposite:2} hold true.
Denoting $b:=1+\tfrac{N+p}p$ and 
setting $k=\mu^-+\bar\dl\xi\om$ with $\bar\dl\in(\dl,1)$, there exists a positive constant $\gm$
depending only on the data, such that  for all $\sig\in(0,1)$ and all $r\in[2\rho, 8\rho]$ we have
\begin{align*}
k(\bar\dl&\xi\om)^{-2}(\xi\om)^b\essup_{t_o-\bar{\theta}(\sig r)^p<t<t_o}\int_{K_{\sig r}\times\{t\}}(u-k)_-^2\,\dx\\
&\quad+\iint_{Q_{\sig r}(\bar{\theta})}|D(u-k)_-|^p\,\dx\dt\\
&\le\frac{\gm}{(1-\sig)^p r^p}\iint_{Q_{r}(\bar{\theta})}
		(u-k)^{p}_- \dx\dt\\
		&\quad+ \frac{\gm}{(1-\sig) \bar{\theta} r^p}\iint_{Q_{r}(\bar{\theta})}(u-k)_-^2\,\dx\dt\\
		&\quad+\frac{\gm\nu}{(1-\sig) \bar{\theta} r^p}\iint_{Q_{r}(\bar{\theta})}(u-k)_-\,\dx\dt,
\end{align*}
provided that
\[
|\mu^-|\le\dl\xi\om.
\]
\end{lemma}
\begin{remark}\upshape
Note that the restriction of $|\mu^-|$ in the last line of Lemma~\ref{Lm:energy}
has been imposed to ensure that the level $k\ge0$ as required in Proposition~\ref{Prop:2:3}.
Recall that such a restriction has been always assumed in this section, cf.~\eqref{Eq:mu-minus}.
Note carefully that the number $\bar\dl\in(\dl,1)$ in Lemma~\ref{Lm:energy} is a free parameter,
whereas $\dl$ is yet to be determined.
%The final choice of $\dl$ in $\widetilde\theta$ will be of the form $\bar\xi\om^q$
%for some positive $\bar\xi$ and $q$ depending only on the data and $\al_*$.
%The final choice of $\dl$ is to impose the restriction
\end{remark}
\begin{remark}\label{Rmk:3:2}\upshape
When we generate the energy estimate of Lemma~\ref{Lm:energy},
the first term of the energy estimate in Proposition~\ref{Prop:2:3} has been discarded. 
Hence the singular effect of $\be(\cdot)$ overtakes
the degeneracy brought by the $p$-Laplacian in Lemma~\ref{Lm:energy}. The machinery of De Giorgi has to be reproduced
to reflect this change. This is the reason why the time scaling will be changed from $\theta=(\xi\om)^{2-p}$
to $\widetilde{\theta}=(\dl\xi\om)^{1-p}$ in the following two lemmas.
\end{remark}

Based on the energy estimate in Lemma~\ref{Lm:energy}, 
we then show that the measure density of the set $[u\approx\mu^-]$ can be made arbitrarily small in a way that favors
a later application for our purpose. 
%To this end, we borrow an idea from \cite[Section~4.4]{DBGV-09} and \cite[Proposition~5.1]{DBG-16}.
The geometric condition \eqref{geometry}
of the boundary $\pl E$ comes into play.
\begin{lemma}\label{Lm:shrink:1}
Suppose the hypotheses in Lemma~\ref{Lm:energy} hold.
There exists a positive constant $j_*$ depending only on the data and $\al_*$, such that
for any $m\in\nn$ we have either
\[
|\mu^-|>\frac{\xi\om}{2^{mj_*}},
\]
or
\begin{equation*}
	\Big|\Big[
	u\le\mu^-+\frac{\xi \om}{2^{mj_*}}\Big]\cap Q_\rho(\widetilde{\theta})\Big|
	%\le\frac{\gm^m}{j_*^{m\frac{p-1}p}}
	\le\frac1{2^m}|Q_\rho(\widetilde{\theta})|,\quad\text{ where }\widetilde{\theta}=\Big(\frac{\xi\om}{2^{mj_*}}\Big)^{1-p}.
\end{equation*}
\end{lemma}
\begin{proof}
%We only show the case of super-solutions, the case of sub-solutions  being similar.
%Moreover, we assume $(x_o,t_o)=(0,0)$.
We employ the energy estimate of Lemma~\ref{Lm:energy} in $Q_{2\rho}(\widetilde{\theta})$
%=K_{2\rho}\times (-\widetilde{\theta}\rho^p,0]$ 
with $\sig=\tfrac12$ and the levels
\begin{equation*}
	k_j:=\mu^-+\frac{\xi \om}{2^{j}},\quad j=0,1,\cdots, j_*.
\end{equation*}
Restricting $|\mu^-|\le 2^{-j_*}\xi\om$ implies that $k_j\ge0$ for all $j=0,1,\cdots, j_*$.
Therefore, assuming $j_*$ has been chosen and $m$ is fixed for now, and taking into account the definition of $\widetilde{\theta}(j_*, m)$, the energy estimate in Lemma~\ref{Lm:energy} yields that
\begin{equation*}
\begin{aligned}
	\iint_{Q_\rho(\widetilde{\theta})}|D(u-k_j)_-|^p\,\dx\dt&\le\frac{\gm}{\varrho^p}\bigg(\frac{\xi \om}{2^j}\bigg)^p\bigg[1+\frac1{\widetilde{\theta}}\bigg(\frac{\xi\om}{2^j}\bigg)^{1-p}\bigg] |A_{j,2\rho}|\\
	&\le\frac{\gm}{\varrho^p}\bigg(\frac{\xi \om}{2^j}\bigg)^p|A_{j,2\rho}|,
\end{aligned}
\end{equation*}
where $ A_{j,2\rho}:= \big[u<k_{j}\big]\cap Q_{2\varrho}(\widetilde{\theta})$.
Next, we apply \cite[Chapter I, Lemma 2.2]{DB} slicewise 
to $u(\cdot,t)$ for 
$t\in(-\widetilde{\theta}\rho^p,0]$
 over the cube $K_\varrho$,
for levels $k_{j+1}<k_{j}$.
Taking into account the measure theoretical information derived from \eqref{geometry}:
\begin{equation*}
	\Big|\Big[u(\cdot, t)>\mu^-+\xi \om\Big]\cap K_{\varrho}\Big|\ge\al_* |K_\varrho|
	\qquad\mbox{for all $t\in(-\widetilde{\theta}\rho^p,0]$,}
\end{equation*}
this gives
\begin{align*}
	&(k_j-k_{j+1})\big|\big[u(\cdot, t)<k_{j+1} \big]
	\cap K_{\varrho}\big|
	\\
	&\qquad\le
	\frac{\gm \varrho^{N+1}}{\big|\big[u(\cdot, t)>k_{j}\big]\cap K_{\varrho}\big|}	
	\int_{\{ k_{j+1}<u(\cdot,t)<k_{j}\} \cap K_{\varrho}}|Du(\cdot,t)|\,\dx\\
	&\qquad\le
	\frac{\gm\varrho}{\al_*}
	\bigg[\int_{\{k_{j+1}<u(\cdot,t)<k_{j}\}\cap K_{\varrho}}|Du(\cdot,t)|^p\,\dx\bigg]^{\frac1p}
	\big|\big[ k_{j+1}<u(\cdot,t)<k_{j}\big]\cap K_{\varrho}\big|^{1-\frac1p}
	\\
	&\qquad=
	\frac{ \gm\varrho}{\al_*}
	\bigg[\int_{K_{\varrho}}|D(u-k_j)_-(\cdot,t)|^p\,\dx\bigg]^{\frac1p}
	\big[ |A_{j,\rho}(t)|-|A_{j+1,\rho}(t)|\big]^{1-\frac1p}.
\end{align*}
Here we used in the last line the  notation $ A_{j,\rho}(t):= \big[u(\cdot,t)<k_{j}\big]\cap K_{\varrho}$.
We now integrate the last inequality with respect to $t$ over  $(-\widetilde{\theta}\rho^p,0]$ and apply H\"older's inequality in time.
With the abbreviation $A_{j,\rho}=[u<k_j]\cap  Q_\rho(\widetilde{\theta})$ this procedure leads to
\begin{align*}
	\frac{\xi \om}{2^{j+1}}\big|A_{j+1,\rho}\big|
	&\le
	\frac{\gm\varrho}{\al_*}\bigg[\iint_{Q_\rho(\widetilde{\theta})}|D(u-k_j)_-|^p\,\dx\dt\bigg]^\frac1p
	\big[|A_{j,\rho}|-|A_{j+1,\rho}|\big]^{1-\frac{1}p}\\
	&\le
	\gm \frac{\xi \om}{2^j}|A_{o,2\rho}|^{\frac1p}\big[|A_{j,\rho}|-|A_{j+1,\rho}|\big]^{1-\frac{1}p}.
\end{align*}
%Recall that $\delta$ depends on the data and $\alpha$. 
 Here the constant $\gm$ in the last line has absorbed $\alpha_*$. 
 
Now take the power $\frac{p}{p-1}$ on both sides of the above inequality to obtain
\[
	\big|A_{j+1,\rho}\big|^{\frac{p}{p-1}}
	\le
	\gm|A_{o,2\rho}|^{\frac1{p-1}}\big[|A_{j,\rho}|-|A_{j+1,\rho}|\big].
\]
Add these inequalities from $0$ to $j_*-1$ to obtain
\[
	j_* |A_{j_*,\rho}|^{\frac{p}{p-1}}\le\gm|A_{o,2\rho}|^{\frac1{p-1}}|A_{o,\rho}|.
\]
From this we conclude
\[
	|A_{j_*,\rho}|\le\frac{\gm}{j_*^{\frac{p-1}p}}|A_{o,2\rho}|^{\frac1{p}}|A_{o,\rho}|^{\frac{p-1}p}.
\]
Similarly, replacing $\rho$ by $2\rho$, we have
\[
	|A_{j_*,2\rho}|\le\frac{\gm}{j_*^{\frac{p-1}p}}|A_{o,4\rho}|^{\frac1{p}}|A_{o,2\rho}|^{\frac{p-1}p}.
\]
The constant $\gm(\text{data}, \al_*)$ appearing in the above two inequalities may be different,
but we can take the larger one.

Suppose $j_*$ has been chosen for the moment.
We may repeat the above arguments for the same $k_j$ with $j=j_*,\cdots, 2j_*$.
Restricting $|\mu^-|\le 2^{-2j_*}\xi\om$ implies $k_j\ge0$.
Hence the energy estimate can be written in the same form with such choices of $k_j$,
and the geometry condition \eqref{geometry} will allow us to apply \cite[Chapter I, Lemma 2.2]{DB} just like above.
Consequently this will lead us to
\begin{align*}
	&|A_{2j_*,\rho}|\le\frac{\gm}{j_*^{\frac{p-1}p}}|A_{j_*,2\rho}|^{\frac1{p}}|A_{j_*,\rho}|^{\frac{p-1}p},\\
	&|A_{2j_*,2\rho}|\le\frac{\gm}{j_*^{\frac{p-1}p}}|A_{j_*,4\rho}|^{\frac1{p}}|A_{j_*,2\rho}|^{\frac{p-1}p}.
\end{align*}
Here the constant $\gm$ is the same as above.

Combining the above estimates would yield that
\[
|A_{2j_*,\rho}|\le\frac{\gm^2 4^{2(N+2)}}{j_*^{2\frac{p-1}p}}|Q_\rho(\widetilde{\theta})|.
\]
Iterating the procedure $m$ times would give us that
\[
|A_{mj_*,\rho}|\le\frac{\gm^m 4^{m(N+2)}}{j_*^{m\frac{p-1}p}}|Q_\rho(\widetilde{\theta})|.%\quad\text{ for all }m\in\nn.
\]
Finally we may choose $j_*$ so large that
\[
\frac{\gm 4^{(N+2)}}{j_*^{\frac{p-1}p}}\le\frac12.
\]
The proof is completed.
%Hence we have
%\[
%|A_{mj_*,\rho}|=\Big|\Big[u<\mu^-+\frac{\xi \om}{2^{mj_*}}\Big]\cap Q_\rho(\theta)\Big|\le\frac1{2^m}|Q_\rho(\theta)|\quad\text{ for all }m\in\nn
%\]
%To conclude, we may substituting $\dl=2^{-m}$.
\end{proof}
Next we apply the energy estimate in Lemma~\ref{Lm:energy} to show a De Giorgi type lemma.
Notice that the time scaling used here is different from the one in Lemmas~\ref{Lm:DG:1} -- \ref{Lm:DG:2};
recall Remark~\ref{Rmk:2:1} and Remark~\ref{Rmk:3:2} for an explanation.
\begin{lemma}\label{Lm:DG:3}
Suppose the hypotheses in Lemma~\ref{Lm:energy} hold.
Let $\dl\in(0,1)$. There exists a  constant $c_1\in(0,1)$ depending only on the data, such that
if
\[
\Big|\Big[u<\mu^-+2\dl\xi\om\Big]\cap Q_{4\rho}(\widetilde{\theta})\Big|\le c_1 (\xi\om)^{b}|Q_{4\rho}(\widetilde{\theta})|\quad\text{ where }\widetilde{\theta}=(\dl\xi\om)^{1-p},
\]
then
either $|\mu^-|>\tfrac12\dl\xi\om$ or
\[
u\ge\mu^-+\dl\xi\om\quad\text{ a.e. in }Q_{2\rho}(\widetilde{\theta}).
\]
\end{lemma}
\begin{proof}
In order to use the energy estimate in Lemma~\ref{Lm:energy}, we set as in \eqref{choices:k_n} for $n=0,1,\cdots,$
%the sequences $k_n$, $\tilde{k}_n$, $\rho_n$, $\widetilde{\rho}_n$, $K_n$, $\widetilde{K}_n$, $Q_n$ and $\widetilde{Q}_n$,
%with $\rho$, $\xi\om$ and $\theta$ replaced by $4\rho$, $\dl\xi\om$ and $\widetilde{\theta}$.
\begin{align*}%\label{choices:k_n}
	\left\{
	\begin{array}{c}
	\displaystyle k_n=\mu^-+ \dl\xi\om+\frac{\dl\xi\om}{2^{n}},\quad \tilde{k}_n=\frac{k_n+k_{n+1}}2,\\[5pt]
	\displaystyle \varrho_n=2\varrho+\frac{\varrho}{2^{n-1}},
	\quad\tilde{\varrho}_n=\frac{\varrho_n+\varrho_{n+1}}2,\\[5pt]
	\displaystyle K_n=K_{\varrho_n},\quad \widetilde{K}_n=K_{\tilde{\varrho}_n},\\[5pt] 
	\displaystyle Q_n=Q_{\rho_n}(\widetilde{\theta}),\quad
	\widetilde{Q}_n=Q_{\tilde\rho_n}(\widetilde{\theta}).
	\end{array}
	\right.
\end{align*}
We will use the energy estimate in Lemma~\ref{Lm:energy} with the pair of cylinders $\widetilde{Q}_n\subset Q_n$.
%$Q_{\rho_n}(\widetilde{\theta})=K_n\times(-\widetilde{\theta}\varrho_n^p,0)$ and $Q_{\tilde\rho_n}(\widetilde{\theta})=\widetilde{K}_n\times(-\widetilde{\theta}\tilde{\varrho}_n^p,0)$. 
Now the constant $\bar\dl$ in Lemma~\ref{Lm:energy} is replaced by $(1+2^{-n})\dl$ here as indicated in the definition of $k_n$.
%Introduce the cutoff function $0\le\z\le 1$ vanishing on the parabolic boundary of $Q_{n}$ and
%equal to identity in $\widetilde{Q}_{n}$, such that
%\begin{equation*}
%	|D\z|\le \gm\frac{2^n}{\varrho}
%	\quad\text{and}\quad 
%	|\z_t|\le \gm\frac{2^{pn}}{\widetilde{\theta}\varrho^p}.
%\end{equation*}
In this setting, enforcing $|\mu^-|\le\tfrac12\dl\xi\om$, the energy estimate in Lemma~\ref{Lm:energy} may be rephrased as
\begin{align*}
	(\dl\xi\om)^{-1}(\xi\om)^b&\essup_{-\widetilde{\theta}\tilde{\varrho}_n^p<t<0}
	\int_{\widetilde{K}_n} (u-\tilde{k}_n)_-^2\,\dx
	+
	\iint_{\widetilde{Q}_n}|D(u-\tilde{k}_n)_-|^p \,\dx\dt\\
	&\qquad\le
	 \gm \frac{2^{pn}}{\varrho^p}(\dl\xi\om)^{p}|A_n|,
\end{align*}
where
\begin{equation*}
	A_n=\big[u<k_n\big]\cap Q_n.
\end{equation*}
Now setting $0\le\phi\le1$ to be a cutoff function which vanishes on the parabolic boundary of $\widetilde{Q}_n$
and equals the identity in $Q_{n+1}$, an application of the H\"older inequality  and the Sobolev imbedding
\cite[Chapter I, Proposition~3.1]{DB} gives that
\begin{align*}
	\frac{\dl\xi\om}{2^{n+3}}
	|A_{n+1}|
	&\le 
	\iint_{\widetilde{Q}_n}\big(u-\tilde{k}_n\big)_-\phi\,\dx\dt\\
	&\le
	\bigg[\iint_{\widetilde{Q}_n}\big[\big(u-\tilde{k}_n\big)_-\phi\big]^{p\frac{N+2}{N}}
	\,\dx\dt\bigg]^{\frac{N}{p(N+2)}}|A_n|^{1-\frac{N}{p(N+2)}}\\
	&\le \gm
	\bigg[\iint_{\tilde{Q}_n}\big|D\big[(u-\tilde{k}_n)_-\phi\big]\big|^p\,
	\dx\dt\bigg]^{\frac{N}{p(N+2)}}\\
	&\quad\ 
	\times\bigg[\essup_{-\widetilde{\theta}\tilde{\varrho}_n^p<t<0}
	\int_{\widetilde{K}_n}\big(u-\tilde{k}_n\big)^{2}_-\,\dx\bigg]^{\frac{1}{N+2}}
	 |A_n|^{1-\frac{N}{p(N+2)}}\\
	&\le 
	\gm (\xi\om)^{-\frac{b}{N+2}}(\dl\xi\om)^{\frac1{N+2}}
	%&\quad\times
	\bigg[(\dl\xi\om)^{p}\frac{2^{pn}}{\rho^p}\bigg]^{\frac{N+p}{p(N+2)}}
	|A_n|^{1+\frac{1}{N+2}}. %\\
\end{align*}
In the second-to-last line we used the above energy estimate.
In terms of $  Y_n=|A_n|/|Q_n|$, this can be rewritten as
\begin{equation*}
	 Y_{n+1}
	\le
	\frac{\gm C^n}{(\xi\om)^{\frac{b}{N+2}}} %\bigg(\frac{\widetilde{\theta}}{(\dl\xi\om)^{1-p}}\bigg)^{\frac1{N+2}} 
	%\bigg(1+\frac{(\dl\xi\om)^{1-p}}{\widetilde{\theta}}\bigg)^{\frac{N+p}{p(N+2)}}
	Y_n^{1+\frac{1}{N+2}},
\end{equation*}
for a constant $\gm$ depending only on the data and with $C=C(N,p)$.
Hence, by \cite[Chapter I, Lemma~4.1]{DB}, 
there exists
a positive constant $c_1$ depending only on the data, such that
$Y_n\to0$ if we require that $Y_o\le c_1(\xi\om)^b$, which amounts to
\begin{equation*}
	|A_o|=\big|\big[u<k_o\big]\cap Q_o\big|
	= 
	\Big|\Big[u<\mu^-+2 \dl\xi\om\Big]\cap Q_{4\varrho}(\widetilde{\theta})\Big|
	\le
	c_1(\xi\om)^b\big| Q_{4\varrho}(\widetilde{\theta})\big|.
\end{equation*}
Since $Y_n\to 0$ as $n\to\infty$, we have
\begin{equation*}
	\Big|\Big[u<\mu^-+  \dl\xi\om \Big]\cap Q_{2 \varrho}(\widetilde{\theta})\Big|
	= 
	0.
\end{equation*}
This concludes the proof of the lemma.
\end{proof}
%%%%%%%%%%%%%%%%%%%%%%%%%%%%
\subsection{Reduction of oscillation near zero concluded}\label{S:reduc-osc-3}
Under the conditions \eqref{Eq:mu-pm}$_1$, \eqref{Eq:alternative}$_2$, \eqref{Eq:opposite:1} and \eqref{Eq:opposite:2}$_1$ ,
we may reduce the oscillation in the following way. First of all,
let $\xi$ be determined in Lemma~\ref{Lm:DG:2} and $j_*$ be determined in Lemma~\ref{Lm:shrink:1}  in terms of the data
and $\al_*$.
We may assume without loss of generality that Lemma~\ref{Lm:shrink:1} holds true for $Q_{4\rho}(\widetilde{\theta})$
instead of $Q_{\rho}(\widetilde{\theta})$.
Then we choose the integer $m$ so large to satisfy that
\begin{equation}\label{Eq:choice-m}
\frac1{2^m}\le c_1(\xi\om)^b,
\end{equation}
where $c_1$ is the constant appearing in Lemma~\ref{Lm:DG:3}.

Next, we can fix $2\dl=2^{-mj_*}$ in Lemma~\ref{Lm:DG:3}.
%that is, by \eqref{Eq:choice-m}, $\dl=$ such that the main hypothesis of Sections~\ref{S:reduc-osc-1} -- \ref{S:reduc-osc-3},
%that is, $|\mu^-|\le\tfrac14\dl\xi\om$, now has been determined quantitatively in terms of the data.
Consequently, by the choice of $m$ in  \eqref{Eq:choice-m}, Lemma~\ref{Lm:DG:3} can be applied
assuming that $|\mu^-|\le\tfrac12\dl\xi\om$, and we arrive at
\[
u\ge\mu^-+ \dl\xi\om \equiv \mu^-+\eta(\xi\om)^{1+q}\quad\text{ a.e. in }Q_{2\rho}(\widetilde{\theta}),
\]
where
\[
\eta=\tfrac12 c_1^{j_*}\quad\text{ and }\quad q=bj_*.
\]
This would give us a reduction of oscillation
\begin{equation}\label{Eq:reduc-osc-2}
\essosc_{Q_{2\rho}(\widetilde{\theta})}u\le(1-\eta\om^q)\om
\end{equation}
with a properly redefined $\eta$. 
Attention is called to the choice of $\dl$ above: by  \eqref{Eq:choice-m}, 
we actually have chosen
\begin{equation*}%\label{Eq:choice-dl}
\dl=\tfrac12[c_1(\xi\om)^b]^{j_*}\equiv\bar{\xi}\om^q
\end{equation*}
where the positive constants $\bar\xi$ and $q$
are now quantitatively determined by the data and $\al_*$ only.

To summarize the achievements in Sections~\ref{S:reduc-osc-1} -- \ref{S:reduc-osc-3},  taking \eqref{Eq:rho-control},
\eqref{Eq:g-control}, \eqref{Eq:reduc-osc-1} and \eqref{Eq:reduc-osc-2} all into account,
 if \eqref{Eq:start-cylinder}, \eqref{Eq:alternative}$_2$  and $|\mu^-|\le\frac12\dl\xi\om$ hold true, then we have
\begin{equation*}%\label{Eq:reduc-osc-3}
\essosc_{Q_{\frac14\rho}(\theta)}u\le (1-\eta\om^q)\om\quad
\text{ or }\quad \om\le  \max\Big\{A\rho^{\al_o},\, 2\osc_{\widetilde{Q}_o\cap S_T}g\Big\},
\end{equation*}
where 
\[
\theta=(\xi\om)^{2-p},\quad\al_o=[(p-1)(1+q)]^{-1},\quad A=(8^{\frac{p}{p-1}}\xi\bar\xi)^{-\frac1{1+q}}.
\]
All the constants $\{\xi,\bar\xi, \eta, q\}$ have been determined in terms of the data and $\al_*$.
%%%%%%%%%%%%%%%%%%%%%%%%%%%%%
\subsection{Reduction of oscillation away from zero}\label{S:reduc-osc-4}
In this section, we handle the case when the equation \eqref{Eq:1:1} resembles the parabolic $p$-Laplacian \eqref{Eq:p-Laplace}.
The reduction of oscillation is summarized in the following two lemmas; 
the proofs rely on the theory for the parabolic $p$-Laplacian \eqref{Eq:p-Laplace}
in \cite[Chapter~III, Section~12]{DB}, 
which we collect in Appendix~\ref{A:1} for completeness.
\begin{lemma}\label{Lm:reduc-osc-1}
Let $u$ be a  weak sub-solution to \eqref{Dirichlet} with \eqref{Eq:1:2} in $E_T$.
Suppose \eqref{Eq:alternative}$_1$ holds true and $\mu^+-\frac14\om\ge0$. 
There exist positive constants $\gm$  and $\xi_1$ depending on
the data and $\al_*$, such that for $\theta=(\xi_1\om)^{2-p}$ we have
\[
\essosc_{Q_{\frac14\rho}(\theta)}u\le (1-\tfrac12\xi_1)\om\quad
\text{ or }\quad \om\le  \max\Big\{A\rho^{\al_o},\, 2\osc_{\widetilde{Q}_o\cap S_T}g\Big\}.
\]
\end{lemma}
%%%%
Let $\xi$ and $\dl=\bar\xi\om^q$ be determined in Sections~\ref{S:reduc-osc-1} -- \ref{S:reduc-osc-3}. We now examine the case
when \eqref{Eq:mu-minus} does not hold.
\begin{lemma}\label{Lm:reduc-osc-2}
Let $u$ be a  weak super-solution to \eqref{Dirichlet} with \eqref{Eq:1:2} in $E_T$.
Suppose \eqref{Eq:alternative}$_2$ holds true and $|\mu^-|>\frac12\dl\xi\om$. 
There exist positive constants $\gm$  and $\xi_2$ depending on
the data and $\al_*$, such that for $ \theta=(\xi_2\dl\xi\om)^{2-p}$ we have
\[
\essosc_{Q_{\frac14\rho}( \theta)}u\le (1- \tfrac12\xi_2\dl\xi)\om\quad
\text{ or }\quad \om\le  \max\Big\{A\rho^{\al_o},\, 2\osc_{\widetilde{Q}_o\cap S_T}g\Big\}.
\]
\end{lemma}
Under the pointwise conditions of Lemmas~\ref{Lm:reduc-osc-1} -- \ref{Lm:reduc-osc-2},
the singularity of $\be(\cdot)$ is avoided, and
the energy estimates of Proposition~\ref{Prop:A:1} are verified for $(u-k)_+$
and for $(u-k)_-$, with $\om$ there replaced by $\frac14\om$ and
$\frac12\dl\xi\om$ respectively.
% The quantity $\frac12\dl\xi\om$ plays the role of $\om$ in Appendix
%%%%%%%%%%%%%%%%%%%%%%%%%%%%%
\subsection{Derivation of the modulus of continuity}\label{S:bdry-modulus}
This is the final part of the proof of Proposition~\ref{Prop:bdry:D}.
Recall that we have defined the quantities $\mu^\pm$ and $\om$
over $\widetilde{Q}_o=K_{8\rho}(x_o)\times(t_o-\rho^{p-1}, t_o)$ for some $\rho\in(0,\frac18\bar\rho)$, 
and we have performed the arguments in Sections~\ref{S:reduc-osc-1} -- \ref{S:reduc-osc-4} starting from
the intrinsic relation \eqref{Eq:start-cylinder}$_2$, that is,
\[
\essosc_{Q_{8\rho}(\widetilde{\theta})}u\le \om\quad\text{ where }\widetilde\theta=(\dl\xi\om)^{1-p}\text{ and } \dl=\bar\xi\om^q.
\]
The positive constants $\{\xi, \bar\xi, q\}$ have been selected in terms of the data and $\al_*$, in the course of proof in 
Sections~\ref{S:reduc-osc-1} -- \ref{S:reduc-osc-4}.
%For $\varep\in(0,1)$ fixed, it is easily checked that 
%\[
%Q_{\rho^{1+\varep}}\subset Q_{\frac14\rho}\quad\text{ for any } \rho\in (0, 4^{-\varep}).
%\]
Therefore, combining all the arguments in Sections~\ref{S:reduc-osc-1} -- \ref{S:reduc-osc-4},
%and taking \eqref{Eq:rho-control} into account, 
we arrive at the reduction of oscillation
\begin{equation}\label{Eq:induction-1}
\essosc_{Q_{\frac14\rho}(\theta)}u\le (1-\eta\om^q)\om\quad
\text{ or }\quad \om\le  \max\Big\{A\rho^{\al_o},\, 2\osc_{\widetilde{Q}_o\cap S_T}g\Big\},
\end{equation}
where $\theta=(\xi\om)^{2-p}$ and $\{\xi,\bar\xi, q, \al_o, \eta, A\}$  have been properly 
%redefined to take into account all relevant constants 
determined in terms of the data and $\al_*$.

In order to iterate the arguments of Sections~\ref{S:reduc-osc-1} -- \ref{S:reduc-osc-4}, 
we set 
\[
\om_1=\max\Big\{(1-\eta\om^q)\om,\, 2\osc_{\widetilde{Q}_o\cap S_T}g,\, A\rho^{\al_o}\Big\}.
\]
We need to choose $\rho_1>0$, such that
the following two set inclusions hold:
\[
Q_{8\rho_1}(\widetilde\theta_1)\subset Q_{\frac14\rho}(\theta),\quad Q_{8\rho_1}(\widetilde\theta_1)\subset \widetilde{Q}_o \quad\text{ where }\widetilde\theta_1=(\dl_1\xi\om_1)^{1-p}\text{ and } \dl_1=\bar\xi\om_1^q.
\]
Just like  \eqref{Eq:start-cylinder}, they will validate the intrinsic relation needed to repeat 
the arguments of Sections~\ref{S:reduc-osc-1} -- \ref{S:reduc-osc-4}.
To this end, we may assume that $\eta\om^q\le\frac12$ and then observe that 
\[
\widetilde\theta_1 (8\rho_1)^p\le[(\tfrac12\om)^{1+q}\xi\bar\xi]^{1-p}(8\rho_1)^p.
\]
Hence to satisfy the first set inclusion, it suffices to choose $\rho_1$ by setting
\[
[(\tfrac12\om)^{1+q}\xi\bar\xi]^{1-p}(8\rho_1)^p=(\tfrac14\rho)^p(\xi\om)^{2-p},\quad\text{ i.e. }
\quad \rho_1^p=\widetilde\xi\om^{\bar{q}}\rho^p, %\text{ for }\bar{q}=1+(p-1)q,
\]
where $\bar{q}:=1+(p-1)q$ and $\widetilde\xi$ depends only on the data and $\al_*$.
It is not hard to verify the second inclusion also with such a choice of $\rho_1$.
As a result of the set inclusions and \eqref{Eq:induction-1}, we obtain the following intrinsic relation
\[
\essosc_{Q_{8\rho_1}(\widetilde\theta_1)}u\le\om_1\quad\text{ where }\widetilde\theta_1=(\dl_1\xi\om_1)^{1-p}\text{ and } \dl_1=\bar\xi\om_1^q,
\]
which takes the place of \eqref{Eq:start-cylinder}$_2$ in the next stage.
Therefore the arguments of Sections~\ref{S:reduc-osc-1} -- \ref{S:reduc-osc-4} 
can be repeated within $Q_{8\rho_1}(\widetilde\theta_1)$, to obtain that either
\[
\essosc_{Q_{\frac14\rho_1}(\theta_1)}u\le (1-\eta\om_1^q)\om_1\quad
\text{ or }\quad \om_1\le  \max\Big\{A\rho_1^{\al_o},\, 2\osc_{\widetilde{Q}_1\cap S_T}g\Big\},
\]
where
\[
\widetilde{Q}_1=K_{8\rho_1}(x_o)\times(t_o-\rho_1^{p-1},t_o).
\]
Here the constants $\{\xi,\bar\xi, q, \al_o, \eta, A\}$ are the same as in the first step.

Now we may proceed by induction and introduce the following notations:
\begin{equation*}
\left\{
\begin{array}{cc}
\rho_o=\rho,\quad  \rho_{n+1}^p=\widetilde\xi\om_n^{\bar{q}}\rho_n^p,\quad\widetilde{\theta}_n=(\xi\bar\xi\om_n^{1+q})^{1-p},\,\quad\theta_n=(\xi\om_n)^{2-p}\\[5pt]
\dsty\om_o=\om,\quad \om_{n+1}=\max\Big\{\om_n(1-\eta\om^q_n),\,A\rho_n^{\al_o},\,2\osc_{\widetilde{Q}_n\cap S_T}g\Big\},\\[5pt]
\widetilde{Q}_n=K_{8\rho_n}(x_o)\times(t_o-\rho_n^{p-1},t_o),\quad Q_n=Q_{\frac14\rho_n}(\theta_n),\quad
Q'_n=Q_{8\rho_n}(\widetilde\theta_n).
\end{array}\right.
\end{equation*}
Upon using the induction we have that for all $n=0,1,\cdots$,
\begin{equation}\label{Eq:induction-n}
 Q'_{n+1}\subset Q_n\quad\text{ and }\quad\essosc_{Q_n} u\le \om_{n+1}.%\quad\text{ where }Q_n=Q_{\frac14\rho_n}(\theta_n).
\end{equation}

%To proceed, let us suppose that 
%\[
%\om_n(1-\eta\om^q_n)>\gm\rho^{\frac{1}{p-1}}+2\osc_{Q_o}g\quad\text{ for all }n=0,1,\cdots,
%\]
%such that
%\[
%\om_{n+1}=\om_n(1-\eta\om^q_n)\quad\text{ for all }n=0,1,\cdots.
%\]

Next we derive an explicit modulus of continuity of $u$ inherent in \eqref{Eq:induction-n}. 
To this end, we first define a sequence $\{r_n\}$ and cylinders $\{\widehat{Q}_n\}$ by
\begin{equation*}
\begin{array}{cc}
r_o=1,\quad  r_{n+1}^p=\widetilde\xi r_n^p,\quad
%\quad\widetilde{\theta}_n=(\dl\xi a_n)^{1-p},\,\quad\theta_n=(\xi a_n)^{2-p}\\[5pt]
\widehat{Q}_n=K_{8r_n}(x_o)\times(t_o-r_n^{p-1},t_o).
%\quad Q_n=Q_{\frac14\rho_n}(\theta_n).
\end{array}
\end{equation*}
Notice that $\rho_n\le r_n$ and $\widetilde{Q}_n\subset\widehat{Q}_n$ for all $n\ge0$.
Observe that if there exists $n_o\in\nn$, such that a sequence $\{a_n\}$ satisfies
\[
%\dsty a_o=1,\quad 
a_{n+1}\ge\max\Big\{a_n(1-\eta a^q_n),\,A r_n^{\al_o},\,2\osc_{\widehat{Q}_n\cap S_T}g\Big\}
\]
 for all $n\ge n_o$ and $a_{n_o}\ge\om_{n_o}$,
%\[
%a_n\ge a_{n-1}(1-\eta a_{n-1}^q)\quad\text{ for all }n\ge n_o, \quad\text{ and }\quad a_{n_o}\ge\om_{n_o},
%\]
then $a_n\ge \om_n$ for all $n\ge n_o$. As a matter of fact, we may take $a_n=(1+n)^{-\sig} \max\{1,\om\}$ with $\sig\in(0,\frac1q)$. 
In such a case,
the number $n_o$ is determined by the parameters $\{\eta, q, \sig, \al_o, A, \widetilde{\xi}, C_g,\lm\}$. 
Indeed, the following two assertions are verified by elementary computations:
\[
a_{n+1}\ge a_n(1-\eta a^q_n),\quad a_{n+1}\ge A r_n^{\be},
\]
for all $n\ge n_o$ and $n_o$ is determined by $\{\eta, q,\sig, \al_o, A\}$. Using the modulus of continuity of $g$,
there exists $\gm=\gm(\widetilde{\xi}, C_g,\lm)$, such that
\[
\osc_{\widehat{Q}_n\cap S_T}g\le \frac{\gm}{n^{\lm}}.%\le\frac1{(n+1)^\sig}.
\]
Recall that $\lm>\sig$. Hence $n_o$ here can be determined in terms of $\{\widetilde{\xi}, C_g,\lm\}$ 
to ensure that for any $n\ge n_o$,
\[
a_{n+1}\ge\osc_{\widehat{Q}_n\cap S_T}g.
\] 
Once the number $n_o$ has been determined by the parameters $\{\eta, q, \sig, \al_o, A, \widetilde{\xi}, C_g,\lm\}$
 and independent of $\om$,
we may assume that %$a_{n_o}\ge\om$ with no loss of generality
$n_o=0$ for simplicity. As a result of $a_o\ge\om_o$, we arrive at
\[
\essosc_{Q_n} u\le \om_{n+1}\le\frac{\max\{1,\om\}}{(n+2)^\sig}\quad\text{ for all }n=0,1,\cdots.
\]

Let us take  a number $4r\in(0,\rho)$. 
There must be some $n$ such that
\[
\rho_{n+1}< 4r\le\rho_{n}.
\]
The right-hand side inequality will yield that $Q_r(\om^{2-p})\subset Q_{\frac14\rho_n}(\theta_n)=Q_n$, and consequently
\begin{equation}\label{Eq:osc-final-1}
\essosc_{Q_r(\om^{2-p})}u\le \essosc_{Q_n}\le\om_{n+1}\le\frac{\max\{1,\om\}}{(n+2)^\sig}.
\end{equation}
Next we analyze the left-hand side inequality.
First notice that by the definition of $\om_n$ and assuming $\eta\om_n^q\le\frac12$, we can estimate
\[
\om_n\ge\frac{\om}{2^n}.
\]
Use this to estimate $\rho_{n+1}^p$ from below by
\begin{align*}
(4r)^p&>\rho^p_{n+1}=\widetilde{\xi}\om_n^{\bar{q}}\rho^p_n\ge\widetilde{\xi}^2\Big(\frac{\om}{2^n}\Big)^{\bar{q}}\rho^p_{n}\ge\cdots\\
&\ge
\widetilde{\xi}^{n+1}\rho^p\prod_{i=1}^{n}\Big(\frac{\om}{2^i}\Big)^{\bar{q}}
=\widetilde{\xi}^{n+1}\Big(\frac{\om^n}{2^{\frac{n(n+1)}2}}\Big)^{\bar{q}}\rho^p.
\end{align*}
Then taking logarithim on both sides, a simple calculation will give us a lower bound of $n$ by
\[
n^2\ge\frac{\gm}{1-\ln\min\{1,\om\}}\ln\frac{\rho}{r},\quad\text{ for some }\gm=\gm(p,q,\widetilde{\xi}).
\]
Substituting this into \eqref{Eq:osc-final-1}, we obtain the modulus of continuity
%\[
%\om_n\le\frac{1}{n^{\sig}}\le\gm\Big(\frac{\ln r}{\ln\rho}\Big)^{\sig}.
%\]
%As a result, we arrive at
\[
\essosc_{Q_r(\om^{2-p})}u\le\gm\max\{1,\om\}\big(1-\ln\min\{1,\om\}\big)^{\frac{\sig}2}\Big(\ln\frac{\rho}{r}\Big)^{-\frac{\sig}2}\quad\text{ for all }r\in(0,\rho).
\]
This concludes the proof of Proposition~\ref{Prop:bdry:D}
%%%%%%%%%%%%%%%%%%%%%%%%%%%%%%%%%%%%
\section{Reduction of interior oscillation}\label{S:interior}
For  a cylinder $\widetilde{Q}_{o}=K_{\rho}(x_o)\times(t_o-\rho^{p-1},t_o)\subset E_T$ with $\rho\in(0,1)$,
we introduce the numbers $\mu^{\pm}$ and $\om$ satisfying
\begin{equation*}
	\mu^+\ge\essup_{\widetilde{Q}_{o}} u,
	\quad 
	\mu^-\le\essinf_{\widetilde{Q}_{o}} u,
	\quad
	\om\ge\mu^+ - \mu^-.
\end{equation*}
The main result of this section is the following.
\begin{proposition}\label{Prop:int-reg}
	Let $u$ be a locally bounded,  local weak solution to \eqref{Eq:1:1} with \eqref{Eq:1:2} in $E_T$.
Then there exist positive constants $\gm$, $\sig$ and $q$  depending only on the data, such that
\[
\essosc_{Q_r(\om^{(2-p)(1+q)})}u\le\gm \max\{1,\om\}\Big(\ln\frac{\rho}{r}\Big)^{-\sig}\quad\text{ for all }r\in(0,\rho).
\]
%where $Q_r(\om^{(2-p)(1+q)})=K_r(x_o)\times(t_o-\om^{(2-p)(1+q)}r^p,t_o)$ for some $q>0$ depending only on the data.
\end{proposition}
%%%%
\subsection{Preliminaries}
We first collect a set of lemmas from which the proof of Proposition~\ref{Prop:int-reg} will follow.
Let us remark that the energy estimate in Proposition~\ref{Prop:2:1} still holds true in the interior with arbitrary $k\in\rr$.
Keeping this remark in mind, we present the following De Giorgi type lemma similar to Lemma~\ref{Lm:DG:1}.

Let $\theta>0$ be  a parameter to be determined.
The cylinder $Q_\rho(\theta)$ is coaxial with $\widetilde{Q}_{o}$ and with the same vertex $(x_o,t_o)$;
moreover, we will impose that 
\[
\theta\rho^p<\rho^{p-1},\quad\text{ such that }\quad Q_\rho(\theta) \subset \widetilde{Q}_{o}.
\]
Then we have the following.

\begin{lemma}\label{Lm:DG:int}
Let $u$ be a  local weak super-solution to \eqref{Eq:1:1} with \eqref{Eq:1:2} in $E_T$. For $\xi\in(0,1)$,
set $\theta=(\xi\om)^{2-p}$ and assume that $Q_{\rho}(\theta)\subset E_T$. 
% and suppose that $$\mu^-+\xi\om\le \inf_{Q_{R,S}\cap S_T}g.$$
There exists a positive constant $c_o$ depending only on the data, such that if
\[
|[u\le\mu^-+\xi\om]\cap Q_{\rho}(\theta)|\le c_o(\xi\om)^{\frac{N+p}p}|Q_{\rho}(\theta)|,
\]
then
\[
u\ge \mu^-+\tfrac12\xi\om\quad\text{ a.e. in }Q_{\frac12\rho}(\theta).
\]
\end{lemma}
\begin{proof}
This is exactly Lemma~\ref{Lm:DG:1} in the interior. The proof runs almost verbatim. 
\end{proof}
%%%%
The next lemma is a variant of the previous one, involving quantitative initial data.
\begin{lemma}\label{Lm:DG:initial:1}
Let $u$ be a  local weak super-solution to \eqref{Eq:1:1} with \eqref{Eq:1:2} in $E_T$. For $\xi\in(0,1)$,
set $\theta=(\xi\om)^{2-p}$. %and assume that $Q_{\rho}(\theta)\subset E_T$. 
% and suppose that $$\mu^-+\xi\om\le \inf_{Q_{R,S}\cap S_T}g.$$
There exists a positive constant $\gm_o$ depending only on the data, such that if
\[
u(\cdot, t_o)\ge\mu^-+ \xi\om,\quad\text{ a.e. in } K_{\rho}(x_o),
\]
then
\[
u\ge\mu^-+\tfrac12\xi\om\quad\text{ a.e. in }K_{\frac12\rho}(x_o)\times(t_o,t_o+\gm_o\theta\rho^p),
\]
provided the cylinders are included in $\widetilde{Q}_o$.
\end{lemma}
\begin{proof}
We will use the energy estimate in Proposition~\ref{Prop:2:1} (the interior version) 
in the cylinder $K_{\rho}(x_o)\times(t_o,t_o+\gm_o\theta\rho^p)$, 
with a cutoff function $\z(x)$ in $K_\rho(x_o)$ independent of $t$, and with levels
\[
k_n=\mu^-+\frac{\xi\om}2+\frac{\xi\om}{2^{n+1}},\quad n=0,1,\cdots.
\] 
Let us first analyze $\Phi_-$ under this setting. If $k_n\le0$ , then $\Phi_{-}(k_n, t_o,t,\z)=0$.
If instead, $k_n>0$ for $n\in\{0,1,\cdots,n_*\}$, then $k_o=\mu^-+\xi\om>0$, and  we calculate
for $n\in\{0,1,\cdots,n_*\}$ that, omitting $x_o$ for simplicity,
\begin{align*}
\Phi_{-}&(k_n, t_o,t,\z)=\int_{K_\rho\times\{\tau\}}\nu \chi_{[u\le0]}(u-k_n)_-\z^p\,\dx\Big|_{t_o}^t\\
&\quad-\int_{t_o}^t\int_{K_\rho}\nu\chi_{[u\le0]}\pl_t (u-k_n)_-\z^p\,\dx\d\tau\\
&=k_n\int_{K_\rho}\nu\chi_{[u\le0]}\z^p\,\dx\Big|_{t_o}^t+\int_{K_\rho}\nu u_-\z^p\,\dx\Big|_{t_o}^t
-\int_{t_o}^t\int_{K_\rho}\nu \pl_t u_-\z^p\,\dx\d\tau\\
% -\int_{t_o}^t\int_{K_\rho(x_o)}\nu \chi_{[u\le0]}(u-k_n)_-\pl_t\z^p\,\dx\d\tau\\
&=k_n\int_{K_\rho\times\{t\}}\nu\chi_{[u\le0]}\z^p\,\dx-k_n\int_{K_\rho\times\{t_o\}}\nu\chi_{[u\le0]}\z^p\,\dx.
%-\int_{t_o}^t\int_{K_\rho(x_o)}\nu \chi_{[u\le0]}(u-k_n)_-\pl_t\z^p\,\dx\d\tau.
\end{align*} 
The first term in the last line is non-negative, while the second term vanishes because of the initial information
\[
u(\cdot, t_o)\ge\mu^-+ \xi\om>0,\quad\text{ a.e. in } K_{\rho}.
\]
Therefore, in any case the term $\Phi_-$ vanishes in the energy estimate of Proposition~\ref{Prop:2:1},
whereas the last spatial integral at $t_o$ also vanishes due to the initial information.
As a result, we obtain energy estimates similar to that of the parabolic $p$-Laplacian \eqref{Eq:p-Laplace},
and the proof can be completed as in \cite[Chapter~3, Section~4]{DBGV-mono}.
\end{proof}
\begin{remark}\label{Rmk:DG:initial:1}\upshape
Lemma~\ref{Lm:DG:initial:1} admits a variant version involving the lateral boundary datum $g$.
To this end, we only need to impose the additional condition that 
$$\mu^-+\xi\om\le \inf_{Q^+_{\rho}(\theta)\cap S_T} g,
\quad\text{ where }Q^+_{\rho}(\theta)=K_{\rho}(x_o)\times(t_o,t_o+\theta\rho^p),$$ 
such that
the levels $k_n$ become admissible in the application of Proposition~\ref{Prop:2:1}.
\end{remark}
%%
%%%%
The next lemma asserts that if we truncate a solution by a positive number,
the resultant function is a sub-solution to the parabolic $p$-Laplacian \eqref{Eq:p-Laplace}.
\begin{lemma}\label{Lm:sub-solution}
Let $u$ be a  local weak sub-solution to \eqref{Eq:1:1} with \eqref{Eq:1:2} in $E_T$. 
Then for any $k>0$ the truncation $(u-k)_+$ is a local weak sub-solution to \eqref{Eq:p-Laplace} with \eqref{Eq:1:2} in $E_T$.
\end{lemma}
\begin{proof}
Upon using the test function
\[
\z=\frac{(u-k)_+}{(u-k)_++\varep}\vp
\]
 where $\varep>0$, and $\vp\ge0$ satisfying \eqref{Eq:function-space} in the weak formulation \eqref{Eq:int-form} in Section~\ref{S:1:2:1}, 
 and noticing that the terms involving $\chi_{[u\le0]}$ all vanish due to the choice $k>0$,
the proof then runs similarly to the one in \cite[Lemma~1.1, Chapter~1]{DBGV-mono}.
\end{proof}
%%%%
Based on Lemma~\ref{Lm:sub-solution}, together with the expansion of positivity 
of \cite[Chapter~5, Proposition~7.1]{DBGV-mono}, we have the following.
\begin{lemma}\label{Lm:expansion:p}
Let $u$ be a  local weak sub-solution to \eqref{Eq:1:1} with \eqref{Eq:1:2} in $E_T$.
Suppose that $\mu^+-\frac14\om>0$ and that for some $\al\in(0,1)$ there holds
\[
|[\mu^+-u(\cdot, t_o)\ge\tfrac14\om]\cap K_{\rho}(x_o)|\ge\al |K_{\rho}|.
\]
Then there exist positive constants $\eta$ and $\kappa$ depending only on the data, such that
\[
\mu^+-u(\cdot, t)\ge \eta_*\om\quad\text{ a.e. in } K_{\frac12\rho}(x_o),\quad\text{ where }\eta_*=\eta\al^\kappa,
\]
for all times
\[
t_o+\tfrac12(\eta_*\om)^{2-p}\rho^p\le t\le t_o+(\eta_*\om)^{2-p}\rho^p,
\]
provided the cylinders are included in $\widetilde{Q}_o$.
\end{lemma}
\begin{proof}
Let $k=\mu^+-\frac14\om$. By Lemma~\ref{Lm:sub-solution}, the function $(u-k)_+$ is a sub-solution to \eqref{Eq:p-Laplace}.
Since $(u-k)_+\le\frac14\om$ in $\widetilde{Q}_o$, the function $v:=\frac14\om-(u-k)_+$ is a non-negative,
local, weak super-solution to \eqref{Eq:p-Laplace}
 in $\widetilde{Q}_o$.
The measure theoretical information given here then reads
\[
|[v(\cdot, t_o)\ge\tfrac14\om]\cap K_{\rho}(x_o)|\ge\al |K_{\rho}|.
\]
Now an application of \cite[Chapter~5, Proposition~7.1]{DBGV-mono} will complete the proof.
\end{proof}
%%%%
\begin{remark}\upshape
The original time interval given in \cite[Chapter~5, Proposition~7.1]{DBGV-mono} is
\[
t_o+\tfrac12\dl b^{p-2}(\eta_*\om)^{2-p}\rho^p\le t\le t_o+\dl b^{p-2}(\eta_*\om)^{2-p}\rho^p.
\]
A careful inspection of the proof there actually reveals that $\dl$ and $b$ can be selected independent of $\al$.
Hence we have absorbed them into $\eta$ here. Moreover, Lemma~\ref{Lm:expansion:p} is stable as $p\to2$.
\end{remark}
%%%%%%%
%\begin{lemma}\label{Lm:DG:initial:2}
%Let $u$ be a  local weak sub-solution to \eqref{Eq:1:1} with \eqref{Eq:1:2} in $E_T$. For $\xi\in(0,1)$,
%set $\theta=(\xi\om)^{2-p}$. %and assume that $Q_{\rho}(\theta)\subset E_T$. 
%% and suppose that $$\mu^-+\xi\om\le \inf_{Q_{R,S}\cap S_T}g.$$
%There exists a positive constant $\gm$ depending only on the data, such that if
%\[
%\mu^+-u(\cdot, t_o)\ge \xi\om,\quad\text{ a.e. in } K_{\rho}(x_o),
%\]
%then
%\[
%\mu^+-u\ge\tfrac12\xi\om\quad\text{ a.e. in }K_{\frac12\rho}(x_o)\times(t_o,t_o+\gm\theta\rho^p),
%\]
%provided the cylinders are included in $E_T$.
%\end{lemma}
%\begin{proof}
%This is exactly Lemma~\ref{Lm:DG:1} in the interior. The proof runs almost verbatim. 
%\end{proof}
%%%%%%
\subsection{Proof of Proposition~\ref{Prop:int-reg}}
%%%%%%%%%%%%%%%%%%%%%%%%%%%%%%%%%%%%
%\section{Proof of Theorem~\ref{Thm:1:1}}
Assume $(x_o,t_o)=(0,0)$.
Let $\theta=L(\frac14\om)^{2-p}$ with $L=L_o\om^{-a}$ for some $a,\, L_o>0$ to determined in terms of the data. 
To proceed we will assume that
\[
Q_\rho(\theta)\subset \widetilde{Q}_o=K_\rho\times(-\rho^{p-1},0),\quad\text{ such that }\quad \essosc_{Q_{\rho}(\theta)}u\le\om;
\]
otherwise we would have 
\begin{equation}\label{Eq:extra-control}
\om\le A\rho^{\frac1{a+p-2}}\quad\text{ where }A=L_o^{\frac1{a+p-2}} 4^{\frac{p-2}{a+p-2}}.
\end{equation}
As in \eqref{Eq:mu-pm} we may also assume that  %$\mu^+\ge|\mu^-|$ as the other case is similar. Consequently, we have
$\mu^+-\frac14\om> \frac14\om>0$, and hence Lemma~\ref{Lm:expansion:p} is at our disposal.

Let us suppose $L$ is a large integer with no loss of generality, and slice $Q_{\rho}(\theta)$ into a vertical stack of $L$ coaxial cylinders
\[
Q_i=K_{\rho}\times(t_i-\bar\theta\rho^p,t_i),\quad t_i=-i\bar\theta\rho^p,\quad i=0,\cdots, L-1,
\] 
all congruent with $Q_\rho(\bar\theta)$.
Here we have set $\bar\theta:=(\frac14\om)^{2-p}$ for ease of notation.

Let us focus on the bottom cylinder $Q_{L-1}$. The proof unfolds along the following two alternatives:
\begin{equation}\label{Eq:alternative-int}
\left\{
\begin{array}{cc}
|[u\le\mu^-+\tfrac14\om]\cap Q_{L-1}|\le c_o(\tfrac14\om)^{\frac{N+p}p}|Q_{L-1}|,\\[5pt]
|[u\le\mu^-+\tfrac14\om]\cap Q_{L-1}|> c_o(\tfrac14\om)^{\frac{N+p}p}|Q_{L-1}|.
\end{array}\right.
\end{equation}

Let us first suppose \eqref{Eq:alternative-int}$_1$ holds true. An application of Lemma~\ref{Lm:DG:int}
would give us that
\[
u\ge \mu^-+\tfrac18\om\quad\text{ a.e. in }\tfrac12Q_{L-1}.
\]
Here the notation $\tfrac12Q_{L-1}$ should be self-explanatory in view of Lemma~\ref{Lm:DG:int}.
In particular, the above pointwise lower bound of $u$ holds at the time level $t_o=-(L-1)\bar\theta\rho^p$
and it serves as the initial datum for an application of Lemma~\ref{Lm:DG:initial:1}.
As a result, by choosing $\xi(\om)$ so small that $\gm_o(\xi\om)^{2-p}\ge L(\frac14\om)^{2-p}$,  we obtain that
\begin{equation}\label{Eq:reduc-osc-int-1}
u\ge\mu^-+\tfrac12\xi\om\quad\text{ a.e. in }K_{\frac12\rho}\times (t_o,0).
\end{equation}
Note that the choice of $\xi(\om)$ is made out of the data and $L=L_o\om^{-a}$, that is, 
$$\xi=\tfrac14(\gm_o/L_o)^{\frac1{p-2}}\om^{\frac{a}{p-2}}$$
where $a, L_o>0$ will be determined later in terms of the data.

Next, we turn our attention to the second alternative \eqref{Eq:alternative-int}$_2$.
Since $\mu^+-\frac14\om\ge\mu^-+\frac14\om$, we may rephrase \eqref{Eq:alternative-int}$_2$ as
\[
|[\mu^+-u\ge\tfrac14\om]\cap Q_{L-1}|> c_o(\tfrac14\om)^{\frac{N+p}p}|Q_{L-1}|:=\al |Q_{L-1}|.
\]
Then it is not hard to see that there exists 
$$t_*\in \Big[t_{L-1}-\bar\theta\rho^p, t_{L-1}- \tfrac12 \al\bar\theta\rho^p\Big],$$ 
such that
\[
|[\mu^+-u(\cdot, t_*)\ge\tfrac14\om]\cap K_\rho|>\al |K_\rho|.
\]
Based on this measure theoretical information at $t_*$, an application of Lemma~\ref{Lm:expansion:p}
 yields that
 \begin{equation}\label{Eq:reduc-osc-int-2}
 \mu^+-u(\cdot, t)\ge\eta_o\al^{\kappa}\om\quad\text{ a.e. in }K_{\frac12\rho}
 \end{equation}
 for all times
 \[
 t_*+\tfrac12(\eta_o\al^{\kappa}\om)^{2-p}\rho^p\le t\le t_*+(\eta_o\al^{\kappa}\om)^{2-p}\rho^p.
 \]
Now we may determine $L$ by setting
 \[
(\eta_o\al^{\kappa}\om)^{2-p}=L\bar\theta,\quad\text{ i.e. }L=(4\eta_o\al^\kappa)^{-(p-2)}
\equiv L_o\om^{-a}.
 \]
Hence we have chosen $a:=\frac{N+p}p(p-2)\kappa$ and $L_o:=(4\eta_o c_o^\kappa 4^{-\frac{N+p}{p} \kappa})^{2-p}$.

Combining \eqref{Eq:reduc-osc-int-1} and \eqref{Eq:reduc-osc-int-2},
and taking \eqref{Eq:extra-control} into account, we obtain the reduction of oscillation
\[
\essosc_{Q_{\frac12\rho}(\theta)}u\le (1-\eta\om^q)\om\quad\text{ or }\quad \om\le A\rho^{\frac1{a+p-2}},
\]
where $q=\frac{N+p}p \kappa$ and $\eta$ depends only on the data.
In order to iterate the above arguments, we define
\[
\om_1=\max\Big\{(1-\eta\om^q)\om,\,A\rho^{\frac1{a+p-2}}\Big\};
\]
we need to choose $\rho_1>0$, such that
\[
%\gm\om_1^a(\tfrac14\om)^{2-p}\rho_1^p\le(\tfrac12\rho)^p\theta,%\gm \om^{-a}(\tfrac14\om)^{2-p},
%\quad\text{ such that } 
Q_{\rho_1}(\theta_1)\subset Q_{\frac12\rho}(\theta), \quad
Q_{\rho_1}(\theta_1)\subset \widetilde{Q}_o,
\quad\text{ where }\theta_1=L_o\om_1^{-a} (\tfrac14\om_1)^{2-p}.
\]
Since we may assume $\om_1\ge\frac12\om$ and estimate
\[
\theta_1\rho_1^p=L_o\om_1^{-a}(\tfrac14\om_1)^{2-p}\rho_1^p\le L_o(\tfrac12\om)^{-a}(\tfrac18\om)^{2-p}\rho_1^p,
\]
the first set inclusion amounts to requiring that
\[
L_o(\tfrac12\om)^{-a}(\tfrac18\om)^{2-p}\rho_1^p=\theta(\tfrac12\rho)^p=L_o\om^{-a}(\tfrac14\om)^{2-p}\rho^p,\quad\text{ i.e. }\rho_1=c\rho
\]
for some $c\in(0,1)$ determined by the data only. The second set inclusion is verified similarly for such 
a choice of $\rho_1$. As a result, we arrive at
\[
\essosc_{Q_{\rho_1}(\theta_1)}\le\om_1.
\]

Now we may proceed by induction and introduce the following notations:
\begin{equation*}
\left\{
\begin{array}{cc}
\rho_o=\rho,\quad  \rho_{n+1}=c\rho_n,\quad\theta_n=L_o\om_n^{-a} (\tfrac14\om_n)^{2-p},\\[5pt]
\dsty\om_o=\om,\quad \om_{n+1}=\max\Big\{\om_n(1-\eta\om^q_n), \,A\rho_n^{\frac1{a+p-2}}\Big\},\\[5pt]
 Q_n=Q_{\frac12\rho_n}(\theta_n),\quad Q'_{n}=Q_{\rho_n}(\theta_n),\quad Q'_{n+1}\subset Q_n.
\end{array}\right.
\end{equation*}
Upon using the induction we have that for all $n=0,1,\cdots$,
\begin{equation*}%\label{Eq:induction-n}
\essosc_{Q_n} u\le \om_{n}.%\quad\text{ where }Q_n=Q_{\frac14\rho_n}(\theta_n).
\end{equation*}
From now on the deduction of a logarithmic type modulus of continuity is performed as in Section~\ref{S:bdry-modulus}.
%%%%%%%%%%%%%%%%%%%%%%%%%%%%%%%%%%%%
\section{Proof of Theorem~\ref{Thm:1:1}}\label{S:Thm-proof}
Before starting the proof of Theorem~\ref{Thm:1:1}, we still have to deal with the continuity at the initial level.
This is considerably simpler than the lateral boundary case and the interior case.
Let us define $\mu^\pm$ and $\om$ to be the supreme/infimum and the oscillation of $u$
 over the forward cylinder $\widetilde{Q}^+_\rho=K_{\rho}(x_o)\times(0,\rho^{p-1})$ 
 for some $x_o\in \overline{E}$ and $\rho\in(0,1)$. We will impose that $\om^{2-p}\rho^p\le\rho^{p-1}$, 
 such that $$Q^+_{\rho}(\theta):=K_{\rho}(x_o)\times(0,\theta\rho^{p})\subset \widetilde{Q}^+_\rho.$$
 Then we have the following oscillation decay estimate.
\begin{proposition}\label{Prop:bdry:Initial}
Suppose the hypotheses in Theorem~\ref{Thm:1:1} hold true.
There exist $\gm>1$ and  $\al_1\in(0,1)$  depending only on the data, such that
\[
\essosc_{Q^+_r(\om^{2-p})\cap E_T}u\le  \om \Big(\frac{r}{\rho}\Big)^{\al_1}+
\gm \osc_{K_{\sqrt{r\rho}}(x_o)\cap E}u_o+\gm \osc_{\widetilde{Q}^+_{\sqrt{r\rho}}\cap S_T} g\quad\text{ for all }r\in(0,\rho).
\]
%There exist $\gm>1$ and  $\sig\in(0,1)$ depending only on the data, %and $\al_*$,
%such that if  %$u_o$ verifies
%\[
%\osc_{K_{r}\cap E}u_o\le \frac{C_{u_o}}{|\ln r|^{\lm}}\quad\text{ and }\quad
%\osc_{Q^+_r\cap S_T} g\le\frac{C_{g}}{|\ln r|^{\lm}},\quad\text{ for all }r\in(0,1),
%\]
%with $C_{u_o}, \, C_g>0$ and $\lm\in(0,\sig)$,
%then
%\[
%\essosc_{Q^+_r(\om^{2-p})\cap E_T}u\le  \gm\max\{1,\om\} \big(1-\ln\min\{1,\om\}\big)^{\frac{\sig}2}\Big(\ln\frac{\rho}{r}\Big)^{-\frac{\sig}2}\quad\text{ for all }r\in(0,\rho),
%\]
%where $Q^+_r(\om^{2-p})=K_r(x_o)\times(0,\om^{2-p}r^p)$.
\end{proposition}
\begin{proof}
Observe first that one of the following two alternatives must hold:
\[
\mu^+-\tfrac14\om>\sup_{\widetilde{Q}^+_\rho\cap \pl_{\mathcal{P}}E_T}\{u_o,\,g\}\quad\text{ or }\quad
\mu^-+\tfrac14\om<\inf_{\widetilde{Q}^+_\rho\cap \pl_{\mathcal{P}}E_T}\{u_o,\,g\};
\]
otherwise we obtain
\begin{equation*}%\label{Eq:osc-u0}
\om\le 2\max\Big\{\osc_{K_\rho(x_o)\cap E}u_o,\, \osc_{\widetilde{Q}^+_\rho\cap S_T} g\Big\}.
\end{equation*}
With no loss of generality, let us suppose the second alternative holds true.
Then we are in the position to apply Lemma~\ref{Lm:DG:initial:1} (recalling Remark~\ref{Rmk:DG:initial:1}), 
such that for some $\gm_o\in(0,1)$
depending only on the data, there holds
\[
u\ge\mu^-+\tfrac18\om\quad\text{ a.e. in }\{K_{\frac12\rho}(x_o)\cap E\}\times(t_o,t_o+\gm_o\theta\rho^p),
\]
where $\theta=(\tfrac14\om)^{2-p}$. This in turn yields a reduction of oscillation
\[
\essosc_{Q^+_{\frac12\rho}(\gm_o\theta)\cap E_T}u\le\tfrac78\om.
\]
Notice that the above reduction of oscillation at the initial level bears no singularity of $\be(\cdot)$
and in fact yields a H\"older type decay.
As such the final oscillation decay is dominated by that of $u_o$ and $g$ near the initial level.
The technical realization runs similar to Section~\ref{S:bdry-modulus}; see also \cite[Section~7.1]{BDL}. 
We refrain from giving further details, to avoid repetition. 
\end{proof}
%%%%
%\begin{remark}\upshape
%The above proof can be extended to obtain the oscillation decay
%in a cylinder with the vertex $(x_o,0)\in \pl E\times\{0\}$.
%\end{remark}

The proof of Theorem~\ref{Thm:1:1} consists of several cases, which rely on
Proposition~\ref{Prop:bdry:D} near the lateral boundary, Proposition~\ref{Prop:int-reg}
in the interior and Proposition~\ref{Prop:bdry:Initial} near the initial level.

Let us consider two points $z_i:=(x_i,t_i)$ in $\overline{E_T}$ with $i=1,2$,
satisfying 
\[
\dist_p(z_1,z_2):=|x_1-x_2|+|t_1-t_2|^{\frac1{p}}\le\tfrac14\bar\rho
\]
where $\bar\rho$ is defined in the geometric property \eqref{geometry} of $\pl E$.

Suppose first that $\dist_p(z_i,\pl_{\mathcal{P}} E_T):=\inf\{\dist_p(z_i,Q): Q\in\pl_{\mathcal{P}} E_T\}>\frac12\bar\rho$ for $i=1,2$, then according to the interior estimate of Proposition~\ref{Prop:int-reg}, we have
for some proper $\sig\in(0,1)$ that
\begin{equation}\label{Eq:modulus-final}
|u(x_1,t_1)-u(x_2,t_2)|\le C\Bigg(\ln \frac{\bar\rho}{|x_1 - x_2|+|t_1 - t_2|^{\frac1p}}\Bigg)^{-\sig},
\end{equation}
where $C>0$ takes into account the data and $\|u\|_{\infty, E_T}$.

Next suppose that one of $\{z_1,z_2\}$, $z_1$ for instance, satisfies $\dist_p(z_1,\pl_{\mathcal{P}} E_T)\le\frac12\bar\rho$.
Suppose further that $t_1, t_2>(\frac12\bar\rho)^{p-1}$.
Then we may apply Proposition~\ref{Prop:bdry:D} to obtain the same type of estimate as \eqref{Eq:modulus-final}.
The present constant $C$ involves the data, $\al_*$, $\lm$, $C_g$ and $\|u\|_{\infty, E_T}$.
%\[
%|u(P_1)-u(P_2)|\le C\Big(\ln \frac{\bar\rho}{|P_1 - P_2|}\Big)^{-\sig},
%\]
%where $C>0$ takes into account the data and $\|u\|_{\infty, E_T}$.

Finally if one of $\{t_1,t_2\}$,  $t_1$ for instance, satisfies that $t_1\le(\frac12\bar\rho)^{p-1}$, 
then we may apply Proposition~\ref{Prop:bdry:Initial} to obtain the same type of estimate as \eqref{Eq:modulus-final}.
The present constant $C$ involves the data, $\lm$, $C_g$, $C_{u_o}$ and $\|u\|_{\infty, E_T}$.
%are away from the parabolic boundary $\Gm$
%Let $(x_o,t_o)\in \overline{E}_T$ and set $d:=\dist\{x_o,\pl E\}$.
%First assume that 
%$$\min\Big\{d,\, t_o^{\frac1p}\Big\}\le\tfrac12\bar{\rho},$$ where $\bar\rho$
%is defined in the geometric property \eqref{geometry} of $\pl E$.

%\noi{\it Case I.} When $d< t_o^{\frac1p}$. There exist $y_o\in\pl E$, such that $d=|x_o-y_o|$.
%According to Proposition~\ref{Prop:int-reg}, we have
%\[
%\essosc_{(x_o,t_o)+Q_r(\om^{(2-p)(1+q)})}u\le\gm \max\{1,\om\}\Big(\ln\frac{d}{r}\Big)^{-\sig}\quad\text{ for all }r\in(0,d).
%\]
%Here $\om$ is given by
%\[
%\om=\essosc_{(x_o,t_o)+Q_d}u.
%\]
%On the other hand, observe that $(x_o,t_o)+Q_d\subset (y_o,t_o)+Q_{2d}$ and Proposition~\ref{Prop:bdry:D} gives that
%\[
%\essosc_{(y_o,t_o)+Q_{2d}(\om^{2-p})}u\le\frac{\gm\max\{1,\om\}}{(1+|\ln\om|)^{\frac{\sig}2}}\Big(\ln\frac{\bar\rho}{d}\Big)^{-\frac{\sig}2}\quad\text{ for all }r\in(0,\bar\rho).
%\]
%%%%%%%%%%%%%%%%%%%%%%%%%%%%%%%%%%%%
\section{The Neumann problem}
In this section, we will focus on the Neumann problem \eqref{Neumann}.
The treatment more or less parallels the interior regularity in Section~\ref{S:interior}.
However, we are not able to treat a general $p$-Laplacian  for the moment; see Remark~\ref{Rmk:6:2}.

First of all, an energy estimate is in order.
\begin{proposition}\label{Prop:energy:Neumann}
	Let $u$ be a bounded weak solution to the Neumann problem 
\eqref{Neumann} under the condition \eqref{Eq:1:2} with $p=2$.
	Assume that $\pl E$ is of class $C^1$ and \eqref{N-data} holds.
	There exists a constant $\gm>0$ depending on $C_o$, $C_1$ and the structure of $\pl E$, such that
 	for all cylinders $Q_{R,S}=K_R(x_o)\times (t_o-S,t_o)$, %with the vertex $(x_o,t_o)\in S_T$,
 	every $k\in\rr$, and every non-negative, piecewise smooth cutoff function
 	$\z$ vanishing on $K_R(x_o)\times (t_o-S,t_o)$,  there holds
\begin{align*}
	\essup_{t_o-S<t<t_o}&\Big\{\int_{\{K_R(x_o)\cap E\}\times\{t\}}	
	\z^2 (u-k)_\pm^2\,\dx +\Phi_{\pm}(k, t_o-S,t,\z)\Big\}\\
	&\quad+
	\iint_{Q_{R,S}\cap E_T}\z^2|D(u-k)_\pm|^2\,\dx\dt\\
	&\le
	\gm\iint_{Q_{R,S}\cap E_T}
		\Big[
		(u-k)^{2}_\pm|D\z|^2 + (u-k)_\pm^2|\partial_t\z^2|
		\Big]
		\,\dx\dt,\\
	&\quad
	+\int_{\{K_R(x_o)\cap E\}\times \{t_o-S\}} \z^2 (u-k)_\pm^2\,\dx\\
	&\quad
	+\gamma  C_2^{2}\iint_{Q_{R,S}\cap E_T} \z^2\chi_{[(u-k)_\pm>0]}\,\dx\dt,
\end{align*}
where 
\begin{align*}
\Phi_{\pm}(k, t_o-S,t,\z)=&-\int_{\{K_R(x_o)\cap E\}\times\{\tau\}}\nu(x,t)\chi_{[u\le0]}[\pm(u-k)_\pm]\z^2\,\dx\Big|_{t_o-S}^t\\
&+\int_{t_o-S}^t\int_{K_R(x_o)\cap E}\nu(x,t) \chi_{[u\le0]}\pl_t[\pm(u-k)_\pm\z^2]\,\dx\d\tau
\end{align*}
and $\nu(x,t)$ is a selection out of $[0,\nu]$ for $u(x,t)=0$, and $\nu(x,t)=\nu$ for $u(x,t)<0$.
\end{proposition}
\begin{proof}
We deal with sub-solutions only as the other case is analogous. 
Upon using the test function $(u-k)_+\z^2$
in the weak formulation of Section~\ref{S:1:4:4} and taking the variational datum into account, similar calculations as in Proposition~\ref{Prop:2:1} can be reproduced here to obtain
\begin{align*}
	\essup_{t_o-S<t<t_o}&\Big\{\int_{\{K_R(x_o)\cap E\}\times\{t\}}	
	\z^2 (u-k)_+^2\,\dx +\Phi_{+}(k, t_o-S,t,\z)\Big\}\\
	&\quad+
	\iint_{Q_{R,S}\cap E_T}\z^2|D(u-k)_+|^2\,\dx\dt\\
	&\le
	\gm\iint_{Q_{R,S}\cap E_T}
		\Big[
		(u-k)^{2}_+|D\z|^2 + (u-k)_+^2|\partial_t\z^2|
		\Big]
		\,\dx\dt,\\
	&\quad
	+\iint_{Q_{R,S}\cap S_T}\z^p\psi(x,t,u)(u-k)_+\,\d \sig\dt\\
	&\quad
	+\int_{\{K_R(x_o)\cap E\}\times \{t_o-S\}} \z^2 (u-k)_+^2\,\dx.
\end{align*}
In order to treat the boundary integral, we make use of \eqref{N-data}, apply the {\it trace inequality} (cf.~\cite[Proposition~18.1]{DB-RA})
for each time slice and then integrate in time, and use Young's inequality
to estimate:
\begin{align*}
&\iint_{Q_{R,S}\cap S_T}\psi(x,t,u)(u-k)_+\z^2\,\d\sig\dt\\
&\quad\le C_2\iint_{\pl (K_R\cap E)\times (t_o-S,t_o)}(u-k)_+\z^2\,\d \sig\dt\\
&\quad\le  \gamma C_2\iint_{Q_{R,S}\cap E_T}
\Big[|D(u-k)_+|\z^2+(u-k)_+\big(\z^2+|D\z^2|\big)\Big]\,\dx\dt\\
&\quad\le\frac{C_o}{4}\iint_{Q_{R,S}\cap E_T} \zeta^2 |D(u-k)_+|^2 \dx\dt
  + \gamma \iint_{Q_{R,S}\cap E_T} (u-k)^2_+|D\zeta|^2\,\dx\dt\\
  &\qquad+ \gamma  C_2^{2}\iint_{Q_{R,S}\cap E_T} \z^2\chi_{[u>k]}\,\dx\dt.
\end{align*}
Substituting it back yields the desired energy estimate.
Notice also that the structure of $\pl E$ enters the constant $\gm$ through the trace inequality.
\end{proof}
%%%
\begin{remark}\label{Rmk:6:1}\upshape
The vertex $(x_o,t_o)$ of the cylinder $Q_{R,S}$
needs not to be on the lateral boundary $S_T$.
If $Q_{R,S}\subset E_T$, then the energy estimate in Proposition~\ref{Prop:energy:Neumann}
reduces to the interior case in Proposition~\ref{Prop:2:1} for $p=2$ and $C_2=0$.
This remark also applies to the following Lemmas~\ref{Lm:DG:Neumann:1} -- \ref{Lm:expansion}.
\end{remark}
%%%%%%%%%%%%%%%%%%
\subsection{Preliminaries}\label{S:6:1}
%Suppose the number $\mu$ satisfies
%\begin{equation}\label{Eq:mu}
%2\mu\ge\essup_{Q_R\cap E_T}|u|.
%\end{equation}
Let $(x_o,t_o)\in \overline{E}_T$ and assume $t_o-\rho^2>0$.
%and set the quantity
% $$2\mu:=\essup_{[(x_o,t_o)+Q_\rho]\cap E_T} |u|.$$
% We may assume that $(x_o,t_o)$ coincides with the origin.
For  a cylinder $Q_{o}=Q_\rho\equiv K_{\rho}(x_o)\times(t_o-\rho^{2},t_o)$ with $\rho\in(0,1)$,
we introduce the numbers $\mu^{\pm}$ and $\om$ satisfying
\begin{equation*}
	\mu^+\ge\essup_{Q_{o}\cap E_T} u,
	\quad 
	\mu^-\le\essinf_{Q_{o}\cap E_T} u,
	\quad
	\om\ge\mu^+ - \mu^-.
\end{equation*}
The numbers $\mu^{\pm}$ and $\om$ are used in the following lemmas.
%Notice that the intersection of cubes or cylinders with $E_T$ is allowed to be empty.
%In such a case, the statements concerns the interior.
\begin{lemma}\label{Lm:DG:Neumann:1}
%Let $u$ be a  local weak super(sub)-solution to \eqref{Neumann} with \eqref{Eq:1:2} and $p=2$ in $E_T$.
Let the hypotheses in Proposition~\ref{Prop:energy:Neumann} hold true and $\xi\in(0,1)$.
There exist  constants $\gm>0$ and $c_o\in(0,1)$ depending only on the data
and the structure of $\pl E$, such that if $\gm C_2\rho<\xi\om$ and 
\[
|[ u<\mu^-+\xi\om]\cap Q_{\rho}\cap E_T|\le c_o(\xi\om)^{\frac{N+2}2}|Q_{\rho}\cap E_T|,
\]
then
\[
 u\ge\mu^-+\tfrac12\xi\om\quad\text{ a.e. in }Q_{\frac12\rho}\cap E_T.
\]
\end{lemma}
\begin{proof}
%This is exactly Lemma~\ref{Lm:DG:1} in the interior. 
The proof runs almost verbatim as in Lemma~\ref{Lm:DG:1}. 
One needs only to intersect all cylinders and cubes with $E_T$ and take $p=2$.
As for the term of $C_2$, it is absorbed into other terms on the right by imposing $\gm C_2\rho<\xi\om$.
Certainly, if $Q_\rho\subset E_T$, the term of $C_2$ does not appear and we are back to
Lemma~\ref{Lm:DG:int} for $p=2$.
\end{proof}
%%%%%%%%%
%\subsection{Expansion of Positivity}
%Let $\mu$ be defined as in \eqref{Eq:mu}. 
We present the following lemma on expansion of positivity.
\begin{lemma}\label{Lm:expansion}
%Let $u$ be a  local weak sub-solution to \eqref{Neumann} with \eqref{Eq:1:2} and $p=2$ in $E_T$.
Let the hypotheses in Proposition~\ref{Prop:energy:Neumann} hold true.
Suppose that $\mu^+-\tfrac14\om>0$ and that for some $\al\in(0,1)$ and some $t_*\in(t_o-\rho^2,t_o)$ there holds
\[
|[\mu^+ -u(\cdot, t_*)>\tfrac14\om]\cap K_{\rho}(x_o)\cap E|\ge\al |K_{\rho}(x_o)\cap E|.
\]
Then there exist positive constants $\gm$, $\dl$, $\eta$ and $\kappa$ depending only on the data
and the structure of $\pl E$, such that
either $\gm C_2\rho\ge\om$ or
\[
\mu^+-u(\cdot, t)\ge\eta\al^\kappa \om\quad\text{ a.e. in } K_{\rho}(x_o)\cap E,
\]
for all times
\[
t_*+\tfrac12\dl\rho^2\le t\le t_*+\dl\rho^2,
\]
provided $t_*+\dl\rho^2\le t_o$.
\end{lemma}
\begin{proof}
We will employ the energy estimate for sub-solutions in Proposition~\ref{Prop:energy:Neumann}.
For this purpose, we restrict the level $k$ in $(\mu^+-\tfrac14\om,\mu^+)$. In this way the term $\Phi_+$ vanishes
as $k>0$, because of the assumption that $\mu^+-\tfrac14\om>0$,
and the energy estimate becomes
\begin{align*}
	\essup_{t_o-S<t<t_o}&\int_{\{K_R(x_o)\cap E\}\times\{t\}}	
	\z^2 (u-k)_+^2\,\dx+
	\iint_{Q_{R,S}\cap E_T}\z^2|D(u-k)_+|^2\,\dx\dt\\
	&\le
	\gm\iint_{Q_{R,S}\cap E_T}
		\Big[
		(u-k)^{2}_+|D\z|^2 + (u-k)_+^2|\partial_t\z^2|
		\Big]
		\,\dx\dt,\\
	&\quad
	+\int_{\{K_R(x_o)\cap E\}\times \{t_o-S\}} \z^2 (u-k)_+^2\,\dx\\
	&\quad
	+ \gamma  C_2^{2}\iint_{Q_{R,S}\cap E_T} \z^2\chi_{[u>k]}\,\dx\dt.
\end{align*}
Next we introduce the non-negative function $v:=(\mu^+-u)/\om$ and 
let $\ell:=(\mu^+-k)/\om\in(0,\tfrac14)$, such that the above energy estimate is equivalent to
\begin{align*}
	\essup_{t_o-S<t<t_o}&\int_{\{K_R(x_o)\cap E\}\times\{t\}}	
	\z^2 (v-\ell)_-^2\,\dx+
	\iint_{Q_{R,S}\cap E_T}\z^2|D(v-\ell)_-|^2\,\dx\dt\\
	&\le
	\gm\iint_{Q_{R,S}\cap E_T}
		\Big[
		(v-\ell)^{2}_-|D\z|^2 + (v-\ell)_-^2|\partial_t\z^2|
		\Big]
		\,\dx\dt,\\
	&\quad
	+\int_{\{K_R(x_o)\cap E\}\times \{t_o-S\}} \z^2 (v-\ell)_-^2\,\dx\\
	&\quad
	+ \gamma  \frac{C_2^{2}}{\om^2}\iint_{Q_{R,S}\cap E_T} \z^2\chi_{[v<\ell]}\,\dx\dt.
\end{align*}
When we work with the pair of cylinders $Q_{\sig\rho}\subset Q_{\rho}$ for some $\sig\in(0,1)$ and a standard cutoff function $\z$
in $Q_\rho$, the term containing $C_2$ can be absorbed
into other terms on the right-hand side via the assumption that $\gm C_2\rho<\om$,
which with no loss of generality we may always assume.
Now we observe that the above inequalities indicate that the non-negative function $v$
is a member of certain parabolic De Giorgi class. The machinery of De Giorgi can be reproduced
near $S_T$; see \cite[Chapter~III, Section~13]{DB}.
Taking Remark~\ref{Rmk:6:1} into consideration, 
an application of  \cite[Proposition~6.1]{Liao}
would allow us to conclude. 
\end{proof}
%%%
\begin{remark}\label{Rmk:6:2}\upshape
Lemma~\ref{Lm:expansion} for $p=2$ parallels Lemma~\ref{Lm:expansion:p} for $p\ge2$,
both concerning the expansion of positivity. 
However, there is a fundamental difference between them. Namely,
Lemma~\ref{Lm:expansion:p} presents the expansion of positivity in the context
of  {\it partial differential equations} (cf.~Lemma~\ref{Lm:sub-solution}), 
whereas Lemma~\ref{Lm:expansion} rests upon {\it parabolic De Giorgi classes} (cf.~\cite{LSU, Liao, W}),
which are much more general and meanwhile lack the same level of understanding as the partial differential equations.
This is ultimately the reason why the case $p>2$ has been excluded in Lemma~\ref{Lm:expansion}.
\end{remark}
%%%%%%%%%%%%%%%%%%%%%%%%%%%%%%%%%%%%%%
\subsection{Proof of Theorem~\ref{Thm:1:2}}
%Let $(x_o,t_o)\in \overline{E}_T$ and assume $t_o-\rho^2>0$.
%%and set the quantity
%% $$2\mu:=\essup_{[(x_o,t_o)+Q_\rho]\cap E_T} |u|.$$
%% We may assume that $(x_o,t_o)$ coincides with the origin.
%For  a cylinder $\widetilde{Q}_{o}=K_{\rho}(x_o)\times(t_o-\rho^{2},t_o)$ with $\rho\in(0,1)$,
%we introduce the numbers $\mu^{\pm}$ and $\om$ satisfying
%\begin{equation*}
%	\mu^+\ge\essup_{\widetilde{Q}_{o}\cap E_T} u,
%	\quad 
%	\mu^-\le\essinf_{\widetilde{Q}_{o}\cap E_T} u,
%	\quad
%	\om\ge\mu^+ - \mu^-.
%\end{equation*}
Introduce the cylinder $Q_o=Q_\rho$ and the numbers $\mu^\pm$ and $\om$ as in Section~\ref{S:6:1}.
As in \eqref{Eq:mu-pm}, we may assume 
\begin{equation}\label{Eq:mu-pm-1}
\mu^+-\tfrac14\om\ge \tfrac14\om>0,
\end{equation} 
such that
Lemma~\ref{Lm:expansion} is at our disposal.
%\begin{equation}\label{Eq:mu-pm-1}
%\mu^+-\tfrac14\om\ge\big|\mu^-+\tfrac14\om\big|.
%\end{equation}
%%%%%%%%%%%%%%%%%%%%%%%%%%%%%%%%
%\subsubsection{Reduction of oscillation near zero}
Let us first suppose that the condition of Lemma~\ref{Lm:DG:Neumann:1} holds true with $\xi=\frac14$, that is,
\begin{equation}\label{Eq:N:alt:1}
|[ u<\mu^-+\tfrac14\om]\cap Q_{\rho}\cap E_T|\le c_o(\tfrac14\om)^{c}|Q_{\rho}\cap E_T|,
\end{equation}
where we have denoted $(N+2)/2$ by $c$ for simplicity.
Then according to Lemma~\ref{Lm:DG:Neumann:1}, we have either $\gm C_2\rho\ge\om$ or 
\[
 u\ge\mu^-+\tfrac18\om\quad\text{ a.e. in }Q_{\frac12\rho}\cap E_T,
\]
which in turn yields a reduction of oscillation
\begin{equation}\label{Eq:reduc-osc-1-N}
\essosc_{Q_{\frac12\rho}\cap E_T}u\le \tfrac78\om.
\end{equation}

Next, we examine the case when \eqref{Eq:N:alt:1} does not hold, that is,
\[
|[ u<\mu^-+\tfrac14\om]\cap Q_{\rho}\cap E_T|> c_o(\tfrac14\om)^{c}|Q_{\rho}\cap E_T|.
\]
%Using \eqref{Eq:mu-pm-1}, 
Since $\mu^+-\frac14\om>\mu^-+\frac14\om$, this gives that
\[
|[ \mu^+-u>\tfrac14\om]\cap Q_{\rho}\cap E_T|> c_o(\tfrac14\om)^{c}|Q_{\rho}\cap E_T|,
\]
and consequently, it is not hard to see that
there exists $$t_*\in\big[t_o-\rho^2, t_o-\tfrac12c_o(\tfrac14\om)^{c}\rho^2\big],$$ such that 
\[
|[ \mu^+-u(\cdot, t_*)>\tfrac14\om]\cap K_{\rho}(x_o)\cap E|> \tfrac12 c_o(\tfrac14\om)^{c}|K_{\rho}(x_o)\cap E|.
\]
Using this measure information, thanks to \eqref{Eq:mu-pm-1}, we are allowed to apply Lemma~\ref{Lm:expansion}
and obtain that there exist positive constants $\{\gm,\kappa,\dl,\eta\}$ depending only on the data, such that
either $\gm C_2\rho\ge\om$ or
\[
\mu^+-u(\cdot, t)\ge\eta\om^{1+q}\quad\text{ a.e. in } K_{\rho}(x_o)\cap E,
\]
for $q=c\kappa$ and for all times
\[
t_*+\tfrac12\dl\rho^2\le t\le t_*+\dl\rho^2.
\]
Starting from this pointwise information, finite times (at most $1/\dl$) of further applications of Lemma~\ref{Lm:expansion} with $\al=1$
will give that, after a proper redefinition of $\eta$,
\[
\mu^+-u(\cdot, t)\ge\eta\om^{1+q}\quad\text{ a.e. in } K_{\rho}(x_o)\cap E,
\]
for all times
\[
t\in\big[t_o-\tfrac12c_o(\tfrac14\om)^{c}\rho^2,t_o\big],
\]
which in turn gives a reduction of oscillation
\begin{equation}\label{Eq:reduc-osc-2-N}
\essosc_{Q_{\frac12\rho}(\theta)\cap E_T} u\le \om(1-\eta\om^q),\quad\text{ where  }\theta=c_o(\tfrac14\om)^{c}.
\end{equation}
In order to iterate, we set
\[
\om_1:=\max\Big\{\om(1-\eta\om^q), \gm C_2\rho\Big\},
\]
and select $\rho_1$ satisfying
\[
\rho_1^2=(\tfrac12\rho)^2\theta\quad\implies\quad Q_{\rho_1}\subset Q_{\frac12\rho}(\theta),
\]
and hence \eqref{Eq:reduc-osc-1-N} and \eqref{Eq:reduc-osc-2-N} imply that
\[
\essosc_{Q_{\rho_1}\cap E_T}u\le \om_1.
\]
Repeating in this fashion, we obtain  the following construction for $n\ge0$:
\begin{equation*}
\left\{
\begin{array}{cc}
\rho_o=\rho,\quad  \rho^2_{n+1}=(\tfrac12\rho_n)^2\theta_n,\quad\theta_n=c_o(\tfrac14\om_n)^{c},\\[5pt]
\dsty\om_o=\om,\quad \om_{n+1}=\max\Big\{\om_n(1-\eta\om^q_n),\gm C_2\rho_n\Big\},\\[5pt]
 Q'_n=Q_{\frac12\rho_n}(\theta_n),\quad Q_{n}=Q_{\rho_n},\quad Q_{n+1}\subset Q'_n.
\end{array}\right.
\end{equation*}
By induction, we have for all $n\ge0$ that
\[
\essosc_{Q_{n}\cap E_T} u\le \om_n.
\]
Then the derivation of the modulus of continuity inherent in the above estimate
can be performed as in the Section~\ref{S:bdry-modulus}.

Finally, we recall that the vertex of $Q_r$ could be any point in $\overline{E}_T$
and therefore the above oscillation decay holds in the interior and up to the lateral boundary.
This joint with the continuity up to the initial level presented in Proposition~\ref{Prop:bdry:Initial}
would allow us to conclude the proof of Theorem~\ref{Thm:1:2},
just like in Section~\ref{S:Thm-proof} for the Dirichlet problem.
%%%%%%%%%%%%%%%%%%%%%%%%%%%%%%%%%%%%
\section{Uniform approximations}\label{S:approx}
Existence of weak solutions to the Stefan problem \eqref{Eq:1:1}, given proper initial and boundary conditions,
is an issue of independent interest.
A standard device in the construction of weak solutions to boundary value problems
%the Dirichlet problem \eqref{Dirichlet} and the Neumann problem \eqref{Neumann} 
consists in first solving regularized versions, deriving {\it a priori} estimates that suggest where to find a solution,
and then obtaining a solution in a proper limiting process via the {\it a priori} estimates and compactness arguments.
In general, such a limiting process requires the function $\bl{A}(x,t,u, \xi)$
to satisfy more stringent structural conditions than \eqref{Eq:1:2}, such as monotonicity in $\xi$,
cf.~\cite[Chapter~V, Theorem~6.7]{LSU}.
The so-obtained limit function  will be in the function space
\[
	L^{\infty}\big(0,T;L^2(E)\big)\cap L^p\big(0,T; W^{1,p}(E)\big),
\]
and meanwhile, verifies one of the integral formulations in Section~\ref{S:1:2}, cf.~\cite{Friedman-68, Urbano-97} and
\cite[Chapter~V, Section~9]{LSU}.
Nevertheless, Theorems~\ref{Thm:1:1} -- \ref{Thm:1:2}
do not grant continuity to this kind of solution as it does not possess time derivative in the Sobolev sense.
{\it A priori} estimates on the time derivative are generally unavailable.
On the other hand, the importance of continuous weak solutions (temperatures) lies in their physical bearings.

The purpose of the present section is to exhibit   that the arguments presented in the previous sections
permit us to identify a limit function, which is continuous
with the same kind of moduli as in Theorems~\ref{Thm:1:1} -- \ref{Thm:1:2}, with the aid of the Ascoli-Arzela theorem
(cf.~\cite[Chapter~5, Section~19]{DB-RA}).
This, joint with the existence results described above, will yield continuous weak solutions
without any knowledge on the time derivative.

To this end, let $H_\varep(s)$ be the mollification with $\varep\in(0,1)$, by the standard Friedrichs kernel
supported in $(-\varep,\varep)$ (cf.~\cite[Chapter~6, Section~18]{DB-RA}), of the function
\begin{equation*}
H(s):=\left\{
\begin{array}{cl}
0,\quad&s>0,\\[5pt]
-\nu,\quad& s\le0.
\end{array}\right.
\end{equation*}
Here $\nu$ is from the definition of $\be(\cdot)$. 
Clearly, the function $s\mapsto s+H_\varep(s)$ is an approximation of $\be(s)$.

Consider the regularized Dirichlet problem:
\begin{equation}\label{Dirichlet-reg}
\left\{
\begin{aligned}
&\pl_t [u+H_\varep(u)] -\dvg\bl{A}(x,t,u, Du) = 0\quad \text{ weakly in }\> E_T\\
&u(\cdot,t)\Big|_{\partial E}=g(\cdot,t)\quad \text{ a.e. }\ t\in(0,T]\\
&u(\cdot,0)=u_o.
\end{aligned}
\right.
\end{equation}
Here the boundary datum $g$ and the initial datum $u_o$ are as in Theorem~\ref{Thm:1:1}.
For each fixed $\varep$,  
%the equation \eqref{Dirichlet-reg}$_1$ is actually of the parabolic $p$-Laplace type \eqref{Eq:p-Laplace}. 
the notion of  solution to \eqref{Dirichlet-reg}
can be defined via a similar integral identity as in Section~\ref{S:1:4:3}.
A main difference now is that the function space for solutions becomes
\begin{equation}\label{Eq:func-space-approx}
\left\{
\begin{array}{cc}
\dsty\int_0^u [1+H'_{\varep}(s)]s\,\d s\in C\big(0,T;L^1(E)\big),\\[5pt]
u\in L^p\big(0,T; W^{1,p}(E)\big).
\end{array}\right.
\end{equation}
This notion does not require any {\it a priori} knowledge on the time derivative
and is similar to the one for \eqref{Eq:p-Laplace} in \cite[Chapter~II]{DB}, cf.~\eqref{Eq:func-space-1}.
The following theorem can be viewed as a ``cousin" of Theorem~\ref{Thm:1:1}.
% \cite[Chapter~II, Section~2]{DB}. In particular, solutions are in the function space \eqref{Eq:func-space-1}
%and H\"older continuous, but the H\"older estimate deteriorates as $\varep\to0$, cf.~\cite[Chapter~III]{DB}.

\begin{theorem}\label{Thm:7:1}
Let $\{u_\varep\}$ be a family of weak solutions to the Dirichlet problem \eqref{Dirichlet-reg} 
under the condition \eqref{Eq:1:2} with $p\ge2$. Assume that \eqref{D}, \eqref{I} and \eqref{geometry} hold. 
Then $\{u_\varep\}$ is equibounded by $M:=\max\{\|u_o\|_{\infty,E},\|g\|_{\infty,S_T}\}$
and is equicontinuous in %any compact set 
$\overline{ E_T}$.
More precisely, there exist positive constants $\gm$ and $q$ depending only on the data
and $\al_*$, and a modulus of continuity $\boldsymbol\om(\cdot)$,
determined by the data, $\al_*$, $\bar\rho$, $M$, $\boldsymbol\om_{o}(\cdot)$ and $\boldsymbol\om_{g}(\cdot)$,
independent of $\varep$, such that
	\begin{equation*}
	\big|u_{\varep}(x_1,t_1)-u_{\varep}(x_2,t_2)\big|
	\le
	%\boldsymbol \gm
	%\|u\|_{\infty,E_T}\,
	\boldsymbol\om\!\left(|x_1-x_2|+|t_1-t_2|^{\frac1p}\right)+\gm\varep^{\frac1{1+q}},
	\end{equation*}
for every pair of points $(x_1,t_1), (x_2,t_2)\in \overline{E_T}$.
In particular, there exists $\sig\in(0,1)$ depending only on the data and $\al_*$,
such that if  %if $g$ satisfies
\[
\boldsymbol\om_g(r)\le \frac{C_g}{|\ln r|^{\lm}}\quad\text{ and }\quad\boldsymbol \om_o(r)\le \frac{C_{u_o}}{|\ln r|^{\lm}}\quad\text{ for all }r\in(0,\bar\rho),
\]
where $C_g,\,C_{u_o}>0$ and $\lm>\sig$,
then the modulus of continuity is
$$\boldsymbol\om(r)=C \Big(\ln \frac{\bar\rho}{r}\Big)^{-\frac{\sig}2}\quad\text{ for all }r\in(0,  \bar\rho)$$ with some $C>0$
 depending on the data, $\lm$, $\al_*$,   $M$, $C_{u_o}$ and $C_{g}$.
\end{theorem}
Similarly we may consider a regularized Neumann problem:
\begin{equation}\label{Neumann-reg}
\left\{
\begin{aligned}
&\pl_t [u+H_\varep(u)]-\dvg\bl{A}(x,t,u, Du) = 0\quad \text{ weakly in }\> E_T\\
&\bl{A}(x,t,u, Du)\cdot {\bf n}=\psi(x,t, u)\quad \text{ a.e. }\ \text{ on }S_T\\
&u(\cdot,0)=u_o(\cdot).
\end{aligned}
\right.
\end{equation}
Here the boundary datum $\psi$ and the initial datum $u_o$ are as in Theorem~\ref{Thm:1:2}. 
The notion of  solution to \eqref{Neumann-reg} can be defined via a similar integral identity as in Section~\ref{S:1:4:4}.
In particular, the function space in \eqref{Eq:func-space-approx} is used.
The following theorem can be regarded as a ``cousin" of Theorem~\ref{Thm:1:2}.
%can be found in \cite[Chapter~II, Section~2]{DB}.
%In particular, solutions are in the function space \eqref{Eq:func-space-1}
%and H\"older continuous, but the H\"older estimate deteriorates as $\varep\to0$, cf.~\cite[Chapter~III]{DB}.
%On the Neumann datum $\psi$ we assume for simplicity that, for some absolute constant $C_2$,
%there holds
%\begin{equation}\label{N-data}\tag{\bf{N}}
%|\psi(x,t, u)|\le C_2\quad \text{ for a.e. }(x,t, u)\in S_T\times\rr.
%\end{equation}
%More general conditions should also work (cf.~Section~2, Chapter~II, \cite{DB}).
%The formal definitions of weak solutions to \eqref{Neumann} will be given in Section~\ref{S:1:2}.
%Now we are ready to present the results concerning regularity of solutions to
%\eqref{Neumann} up to the parabolic boundary $\pl_{\mathcal{P}}E_T$.

\begin{theorem}\label{Thm:7:2}
Let $\{u_\varep\}$ be a family of weak solutions to the Neumann problem 
\eqref{Neumann-reg} under the condition \eqref{Eq:1:2} with $p=2$. 
Assume that $\pl E$ is of class $C^1$, and \eqref{N-data} and \eqref{I} hold. 
 Then $\{u_\varep\}$ is equibounded by a constant $M$ depending only on
 the data, $|E|$, $T$, $C_2$, $\|u_o\|_{\infty,E}$,
 and the structure of $\pl E$, and is equicontinuous in $\overline{E_T}$.
 More precisely, there exists %positive constants $\gm$ and $q$ depending only on the data,
a modulus of continuity $\boldsymbol\om(\cdot)$,
determined by the data, the structure of $\pl E$, $C_2$, $M$  and $\boldsymbol\om_{o}(\cdot)$, 
independent of $\varep$, such that  %$h(\varep)\to0$ as $\varep\to0$ and
	\begin{equation*}
	\big|u_{\varep}(x_1,t_1)-u_{\varep}(x_2,t_2)\big|
	\le
	%\boldsymbol \gm
	%\|u\|_{\infty,E_T}\,
	\boldsymbol\om\!\left(|x_1-x_2|+|t_1-t_2|^{\frac12}\right)+4\varep,
	\end{equation*}
for every pair of points $(x_1,t_1), (x_2,t_2)\in \overline{E_T}$.
In particular, there exists $\sig\in(0,1)$ depending only on the data, $C_2$ and the structure of $\pl E$,
such that if %if $g$ satisfies
\[
\boldsymbol \om_o(r)\le \frac{C_{u_o}}{|\ln r|^{\lm}}\quad\text{ for all }r\in(0,1),
\]
where $C_{u_o}>0$ and $\lm>\sig$,
then the modulus of continuity is 
$$\boldsymbol\om(r)=C \Big(\ln \frac{1}{r}\Big)^{-\frac{\sig}2}\quad\text{ for all }r\in(0, 1),$$ with some $C>0$
 depending on the data, the structure of $\pl E$, $C_2$, $M$ and $C_{u_o}$.
%More precisely, there exist constants $\boldsymbol\gm>1$ and $\be\in(0,1)$
%determined by the data, $C_2$, $\dist(\mathcal{K};\{t=0\})$ and the structure of $\pl E$, such that
%	\begin{equation*}
%	\big|u(x_1,t_1)-u(x_2,t_2)\big|
%	\le
%	\boldsymbol \gm
%	\|u\|_{\infty,E_T}\,
%	\left(|x_1-x_2|+|t_1-t_2|^{\frac1p}\right)^\be,
%	\end{equation*}
%for every pair of points $(x_1,t_1), (x_2,t_2)\in \mathcal{K}$.
\end{theorem}
%%%%%%%
\begin{remark}\upshape
Theorems~\ref{Thm:7:1} -- \ref{Thm:7:2} lay the foundation
of obtaining continuous solutions to the Dirichlet problem \eqref{Dirichlet} or the Neumann problem \eqref{Neumann},
once we can construct solutions $\{u_\varep\}$ to \eqref{Dirichlet-reg} or \eqref{Neumann-reg} and identify
the limit function $u$ as a proper weak solution to \eqref{Dirichlet} or \eqref{Neumann}.
Clearly, the so-obtained solution $u$ will enjoy the logarithmic type modulus of continuity.
\end{remark}

We will treat Theorem~\ref{Thm:7:1} only as Theorem~\ref{Thm:7:2} can be dealt with in a similar way.
Concentration will be made on examining the boundary  arguments in Section~\ref{S:bdry}.
The subscript $\varep$ will suppressed from $u$, $\mu^\pm$, $\om$, $\theta$, $\widetilde{\theta}$, etc.
The idea is to adapt the arguments in  Section~\ref{S:bdry} and to determine the quantities, such as $\bar\xi$, $\xi$, $\eta$, $A$,
independent of $\varep$.
In this way the reduction of oscillation of $u$ can be achieved just like in  Section~\ref{S:bdry}, independent of $\varep$.

Let us first observe that, the arguments in  Section~\ref{S:bdry} %proof of Theorem~\ref{Thm:1:1}   
 hinges solely on the energy estimates in Proposition~\ref{Prop:2:1}.
Now the test function
\[
\zeta^p(x,t) \big(u(x,t)-k\big)_\pm
\]
against \eqref{Dirichlet-reg}$_1$ and
with $k$ satisfying \eqref{Eq:k-restriction}, 
is justified modulo an averaging process in the time variable; this can be done as in \cite[Proposition~3.1]{BDL}.
No {\it a priori} knowledge on the time derivative is needed, because now $s\mapsto s+H_\varep(s)$
is a smooth, increasing function.

After standard calculations, the  energy estimates for the weak solution $u$ to \eqref{Dirichlet-reg} becomes,
omitting the reference to $x_o$,
\begin{equation}\label{Eq:energy-approx}
\begin{aligned}
	\essup_{t_o-S<t<t_o}&\Big\{\int_{K_R\times\{t\}}	
	\z^p (u-k)_\pm^2\,\dx \pm \int_{K_R\times\{t\}}\int_k^u H_\varep'(s)(s-k)_\pm\,\d s\,\z^p\dx \Big\}\\
	&\quad+
	\iint_{Q_{R,S}}\z^p|D(u-k)_\pm|^p\,\dx\dt\\
	&\le
	\gm\iint_{Q_{R,S}}
		\Big[
		(u-k)^{p}_\pm|D\z|^p + (u-k)_\pm^2|\partial_t\z^p|
		\Big]
		\,\dx\dt\\
			&\quad\pm\iint_{Q_{R,S}}\int_k^u H_\varep'(s)(s-k)_\pm\,\d s |\pl_t\z^p|\, \dx\dt\\
			&\quad
	+\int_{K_R\times \{t_o-S\}} \z^p (u-k)_\pm^2\,\dx\\
	&\quad  \pm \int_{K_R\times\{t_o-S\}}\int_k^u H_\varep'(s)(s-k)_\pm\,\d s\,\z^p\dx.
\end{aligned}
\end{equation}
%where 
%\begin{align*}
%\Phi_{\pm}(k, t_o-S,t,\z)=&-\int_{K_R(x_o)\times\{\tau\}}\nu(x,\tau)\chi_{[u\le0]}[\pm(u-k)_\pm\z^p]\,\dx\Big|_{t_o-S}^t\\
%&+\int_{t_o-S}^t\int_{K_R(x_o)}\nu(x,\tau)\chi_{[u\le0]}\pl_t[\pm(u-k)_\pm\z^p]\,\dx\d\tau
%\end{align*}
The three terms containing $H'_\varep$ here play the role of $\Phi_{\pm}$ in Proposition~\ref{Prop:2:1},
which preserve the singularity of $\be(\cdot)$ at the origin as $\varep\to0$.

Let us consider the case of super-solution, i.e. $(u-k)_-$. 
The term containing $H'_\varep$ (together with the minus sign in the front) on the left-hand side is non-negative
and hence can be discarded. The first term containing $H'_\varep$ on the right-hand side
is estimated by
\[
\iint_{Q_{R,S}}\int_u^k H_\varep'(s)(s-k)_-\,\d s |\pl_t\z^p|\, \dx\dt
\le \nu\iint_{Q_{R,S}}(u-k)_- |\pl_t\z^p|\,\dx\dt.
\]
The second term containing $H'_\varep$ on the right-hand side is discarded 
because now $\z=0$ on $\pl_{\mathcal{P}}Q_{R,S}$.
Using these remarks, one can perform the De Giorgi iteration in Lemma~\ref{Lm:DG:1}
and reach the same conclusion.

As for Lemma~\ref{Lm:DG:2}, now the condition becomes $|\mu^-|\le\xi\om$ and, letting $k=\mu^-+\xi\om$, 
\begin{equation}\label{Eq:measure-2-approx}
\iint_{Q_{\rho}(\theta)}\int_u^k H_\varep'(s) \,\d s\, \dx\dt
\le\xi\om|[u\le \mu^-+\tfrac12\xi\om]\cap Q_{\frac12\rho}(\theta)|.
\end{equation}
Then the same conclusion holds as in  Lemma~\ref{Lm:DG:2}.

With the modified versions of  Lemma~\ref{Lm:DG:1} and Lemma~\ref{Lm:DG:2} at hand,
 we can start the arguments as in Section~\ref{S:bdry}.
 Notice also that Proposition~\ref{Prop:2:2} now holds for $k\ge\varep$
 in the case of sub-solution and for $k\le-\varep$  in the case of super-solution.

If \eqref{Eq:alternative}$_1$ holds, then thanks to \eqref{Eq:mu-pm}$_1$, the function $(u-k)_+$ with $k=\mu^+-2^{-n}\om$ will
satisfy the energy estimate in Proposition~\ref{Prop:2:2} for all $n\ge2$, provided we assume that $\frac14\om\ge\varep$.
In this case, the reduction of oscillation can be achieved as in Section~\ref{S:reduc-osc-4}, cf.~Lemma~\ref{Lm:reduc-osc-1}.
Consequently we need only to examine the case when \eqref{Eq:alternative}$_2$ holds.
This case splits into two sub-cases that parallel Sections~\ref{S:reduc-osc-1} -- \ref{S:reduc-osc-4},
which we now examine.
%In Sections~\ref{S:reduc-osc-1} -- \ref{S:reduc-osc-3}  we deal with the situation when $|\mu^-|$
%is near zero; whereas the opposite situation is treated in Section~\ref{S:reduc-osc-4}, cf. Lemma~\ref{Lm:reduc-osc-2}. 
 
 The conclusion \eqref{Eq:reduc-osc-1} is reached as in Section~\ref{S:reduc-osc-1}.
 The only change is that we need to assume \eqref{Eq:measure-2-approx} instead of \eqref{Eq:measure-2}.
 
Let us examine Section~\ref{S:reduc-osc-2}.
A key change appears in Section~\ref{S:reduc-osc-2} as \eqref{Eq:opposite:2} becomes
the opposite of \eqref{Eq:measure-2-approx}, that is, letting $k=\mu^-+\xi\om$,
\[
\iint_{Q_{\rho}(\theta)}\int_u^k H_\varep'(s) \,\d s\, \dx\dt
>\xi\om|[u\le \mu^-+\tfrac12\xi\om]\cap Q_{\frac12\rho}(\theta)|.
\]
This, joint with \eqref{Eq:opposite:1},  yields a variant of \eqref{Eq:opposite:3},
that is, for all $r\in[2\rho,8\rho]$ we have
\begin{align*}
\iint_{Q_{r}(\theta)}\int_u^k H_\varep'(s)\,\d s\, \dx\dt
&>\xi\om|[u\le \mu^-+\tfrac12\xi\om]\cap Q_{\frac12\rho}(\theta)|\\
&\ge c_o(\xi\om)^{b}|Q_{\frac12\rho}(\theta)|\ge\widetilde{\gm}(\xi\om)^{b}|Q_{r}(\theta)|,
\end{align*}
where just like in  \eqref{Eq:opposite:3} we have set $\widetilde{\gm}=c_o16^{-N-p}$ and
$b=1+\tfrac{N+p}p$.
Using \eqref{Eq:time:2}, this in turn gives 
\begin{align*}
&\essup_{t_o-\bar{\theta}r^p<t<t_o}\int_{K_{r}}\int_u^k H_\varep'(s)(s-k)_-\,\d s\, \dx\\
&\qquad\ge(k-\varep)\essup_{t_o-\bar{\theta}r^p<t<t_o}\int_{K_{r}}\int_u^k H_\varep'(s)\,\d s\, \dx\\
&\qquad\ge\widetilde{\gm}(k-\varep)(\xi\om)^{b}|K_{r}|\\
&\qquad\ge \widetilde{\gm}(k-\varep)(\xi\om)^{b}
(\bar\dl\xi\om)^{-2}\essup_{t_o-\bar{\theta} r^p<t<t_o}\int_{K_r\times\{t\}}[u-(\mu^-+\bar\dl\xi\om)]_-^2\,\dx.
\end{align*}
Consequently, Lemma~\ref{Lm:energy} holds true in view of the energy estimate \eqref{Eq:energy-approx}, 
and Lemmas~\ref{Lm:shrink:1} -- \ref{Lm:DG:3} can be reproduced, once the condition
$\varep\le\frac14\dl\xi\om\equiv \frac14\xi\bar\xi\om^{1+q}$ is imposed;
 otherwise, it just gives us an extra control $\om\le \gm \varep^{\frac1{1+q}}$ for some positive $\gm(\text{data},\al_*)$.
The rest of the arguments in Sections~\ref{S:reduc-osc-3} -- \ref{S:bdry-modulus} remains unchanged.

The reduction of interior oscillation can be reproduced as in Section~\ref{S:interior}.
Lemmas~\ref{Lm:DG:int} -- \ref{Lm:DG:initial:1} are unchanged. Lemma~\ref{Lm:expansion:p}
still holds if we impose $\mu^+-\frac14\om>\varep$.
This is not a problem, as we may assume %that $\mu^+\ge|\mu^-|$ and hence 
that $\mu^+-\frac14\om\ge\frac14\om>\varep$, cf.~\eqref{Eq:mu-pm};
 otherwise, it just gives us an extra control $\om\le 4\varep$.
 
 The derivation of modulus of continuity runs similar to Section~\ref{S:bdry-modulus};
  here it is affected only by an extra control of order $\varep$ or $\varep^{\frac1{1+q}}$.
Therefore we may conclude the proof of Theorem~\ref{Thm:7:1}.
%A tailored version as in Proposition~\ref{Prop:2:2} holds for
%$k\ge\varep$ in the case of sub-solutions or for $k\le-\varep$ in the case of super-solutions.
%The version that corresponds to Proposition~\ref{Prop:2:3} reads
%\begin{align*}
%	\essup_{t_o-S<t<t_o}&\Big\{\int_{K_R\times\{t\}}	
%	\z^p (u-k)_-^2\,\dx + \int_{K_R\times\{t\}}\int_u^k H_\varep'(s)(s-k)_-\,\d s\,\dx \Big\}\\
%	&\quad+
%	\iint_{Q_{R,S}}\z^p|D(u-k)_-|^p\,\dx\dt\\
%	&\le
%	\gm\iint_{Q_{R,S}}
%		\Big[
%		(u-k)^{p}_-|D\z|^p + (u-k)_-^2|\partial_t\z^p|
%		\Big]
%		\,\dx\dt\\
%	%&\quad+\int_{K_R\times \{t_o-S\}} \z^p (u-k)_\pm^2\,\dx\\
%	&\quad +\iint_{Q_{R,S}}\int_u^k H_\varep'(s)(s-k)_-\,\d s |\pl_t\z^p|\, \dx\dt.
%\end{align*}
%Here $k\ge\varep$ is imposed.
%%%%%%%%%%%%%%%%%%%%%%%%%%%%%%%%%%%%
\appendix

\section{Boundary regularity for the parabolic $p$-Laplacian}\label{A:1}
%We collect the known facts about boundary regularity for the parabolic $p$-Laplacian here.
Let $\widetilde{Q}_o$ be as in Section~\ref{S:bdry} with its vertex $(x_o,t_o)\in S_T$ and $t_o\ge\rho^{p-1}$.
Define $\mu^{\pm}$ to be the supreme/infimum of $u$ over $\widetilde{Q}_o\cap E_T$.
Introduce parameters $a\in(0,1)$, $\om>0$ and 
\begin{equation}\label{Eq:level:A}
k=\left\{
\begin{array}{ll}
\dsty\mu^+- a \om\ge \sup_{\widetilde{Q}_o\cap S_T}g,&\quad\text{ for }(u-k)_+,\\[5pt]
\dsty\mu^-+ a \om\le \inf_{\widetilde{Q}_o\cap S_T}g,&\quad\text{ for }(u-k)_-.
\end{array}\right.
\end{equation}
We will consider a cylinder $Q_r(\theta)$ with its vertex $(x_o,t_o)$ and $\theta=(\xi \om)^{2-p}$,
for some $\xi\in(0,1)$ to be determined later. We assume that
 $Q_r(\theta)\subset \widetilde{Q}_o$ for all $r\in[\frac12\rho, 2\rho]$. Now we state the boundary regularity
 for the parabolic $p$-Laplacian with $p\ge2$, which has been sketchily presented in \cite[Chapter~III, Section~12]{DB}.
\begin{proposition}\label{Prop:A:1}
%For $\xi\in(0,1)$, let $k=\mu^+-\xi M$ for $(u-k)_+$ and $k=\mu^-+\xi M$ for $(u-k)_-$. 
Let $\bar\theta=(a\om)^{2-p}$ for   $a\in(\xi,1)$ and 
let $\z$ be a piecewise smooth function in $Q_{r}(\bar\theta)$ that vanishes on $\pl_{\mathcal{P}}Q_{r}(\bar\theta)$. 
Suppose that $u$ verifies
the following energy inequalities in $Q_{r}(\bar\theta)$ for some positive $\widetilde{C}$,
\begin{align*}
	\essup_{t_o-\bar\theta r^p<t<t_o}&\int_{K_r(x_o)\times\{t\}}	
	\z^p (u-k)_\pm^2\,\dx+
	\iint_{Q_{r}(\bar\theta)}\z^p|D(u-k)_\pm|^p\,\dx\dt\\
	&\le
	\widetilde{C}\iint_{Q_{r}(\bar\theta)}
		\Big[
		(u-k)^{p}_\pm|D\z|^p + (u-k)_\pm^2|\partial_t\z^p|
		\Big]
		\,\dx\dt,
%	&\quad
%	+\int_{K_R(x_o)\times \{t_o-S\}} \z^p (u-k)_\pm^2\,\dx,
\end{align*}
with the choice of $k$ in \eqref{Eq:level:A} for all $a\in(\xi,1)$
and all $r\in[\frac12\rho, 2\rho]$.
Then there exists $\xi\in(0,1)$ determined by $\{\widetilde{C}, N, p,\al_*\}$, such that
\[
\pm(\mu^{\pm}-u)\ge\tfrac12\xi\om\quad\text{ a.e. in }Q_{\frac12\rho}(\theta)\quad\text{ where }\theta=(\xi \om)^{2-p}.
%\essosc_{Q_{\frac12\rho}(\theta)}u\le(1-\eta)M
\]
\end{proposition}
%In this section, we will call the parameters $\{\widetilde{C}, N, p\}$ the data.
The proof of Proposition~\ref{Prop:A:1} is a direct consequence of the following lemmas.
\begin{lemma}
For $j_*\in\nn$, let $\theta=(2^{-j_*}\om)^{2-p}$. There exists $\gm>0$ depending only on $\{\widetilde{C}, N, p,\al_*\}$, such that
\[
\Big| \Big[ \pm(\mu^{\pm}-u)\le\frac{\om}{2^{j_*}}\Big]\cap Q_{\rho}(\theta)\Big|\le\frac{\gm}{j_*^{\frac{p-1}p}}|Q_{\rho}(\theta)|.
\]
\end{lemma}
\begin{proof}
We only treat the case of $u-\mu^-$ here as the other case is analogous.
The proof runs similar to Lemma~\ref{Lm:shrink:1}. 
We define $k_j=\mu^-+2^{-j}\om$ for $j=0,1,\cdots, j_*$, and work within $Q_{2\rho}(\theta)$.
After using a standard cutoff function $\z$ in $Q_{2\rho}(\theta)$ that equals the identity in $Q_{\rho}(\theta)$
and vanishes on $\pl_{\mathcal{P}}Q_{2\rho}(\theta)$,  the energy estimate in Proposition~\ref{Prop:A:1} gives that
\begin{equation*}
\begin{aligned}
	\iint_{Q_\rho(\theta)}|D(u-k_j)_-|^p\,\dx\dt&\le\frac{\gm}{\varrho^p}\bigg(\frac{\om}{2^j}\bigg)^p\bigg[1+\frac1{\theta}\bigg(\frac{\om}{2^j}\bigg)^{2-p}\bigg] |A_{j,2\rho}|\\
	&\le\frac{\gm}{\varrho^p}\bigg(\frac{\om}{2^j}\bigg)^p|A_{j,2\rho}|,
\end{aligned}
\end{equation*}
where $ A_{j,2\rho}:= \big[u<k_{j}\big]\cap Q_{2\varrho}(\theta)$.
The rest of the proof can be carried out as in the proof of Lemma~\ref{Lm:shrink:1} almost verbatim. 
Clearly, we do not need to iterate $m$ times here. %as in the proof of Lemma~\ref{Lm:shrink:1}.
\end{proof}
%%%%
\begin{lemma}
For $\xi\in(0,1)$, let $\theta=(\xi \om)^{2-p}$. There exists $c_o\in(0,1)$ depending only on $\{\widetilde{C}, N, p\}$, such that
if
\[
\big| \big[ \pm(\mu^{\pm}-u)\le \xi \om \big]\cap Q_{\rho}(\theta) \big|\le c_o |Q_{\rho}(\theta)|,
\]
then
\[
\pm(\mu^{\pm}-u)\ge\tfrac12\xi \om\quad\text{ a.e. in }Q_{\frac12\rho}(\theta).  %\quad\text{ where }\theta=(2\eta M)^{2-p}.
%\essosc_{Q_{\frac12\rho}(\theta)}u\le(1-\eta)M
\]
\end{lemma}
\begin{proof}
We only treat the case of $u-\mu^-$ here as the other case is analogous.
The argument runs similar to that of Lemma~\ref{Lm:DG:1}.
We may first define the various quantities, cubes and cylinders as in \eqref{choices:k_n}.
The difference is that the energy estimate now, according to Proposition~\ref{Prop:A:1}, becomes
\begin{align*}
	\essup_{-\theta\tilde{\varrho}_n^p<t<0}
	&\int_{\widetilde{K}_n} (u-\tilde{k}_n)_-^2\,\dx
	+
	\iint_{\widetilde{Q}_n}|D(u-\tilde{k}_n)_-|^p \,\dx\dt\\
	&\qquad\le
	 \gm \frac{2^{pn}}{\varrho^p}(\xi\om)^{p}|A_n|,
\end{align*}
where $A_n=\big[u<k_n\big]\cap Q_n$.
After an application of the above energy estimate, the H\"older inequality  and the Sobolev imbedding
\cite[Chapter I, Proposition~3.1]{DB} as in Lemma~\ref{Lm:DG:1},  the recurrence of $Y_n=|A_n|/|Q_n|$ becomes
\begin{equation*}
	 Y_{n+1}
	\le\gm C^n Y_n^{1+\frac{1}{N+2}},
\end{equation*}
where $\gm$ and $C$ depends only on $\{\widetilde{C}, N, p\}$.
As in Lemma~\ref{Lm:DG:1} we may conclude.
\end{proof}
%%%%%%%%%%%%%%%%%%%%%%%%%%%%%%%%%%%%
%%%%%%%%%%%%%%%%%%%%%%%%%%%%%%%%%%%%

\end{document}